\documentclass[12pt,a4paper]{article}
\usepackage{amsmath,amstext,amssymb,amscd, color}

\newcommand{\Xcomment}[1]{}

\newtheorem{theorem}{Theorem}[section]
\newtheorem{lemma}[theorem]{Lemma}
\newtheorem{corollary}[theorem]{Corollary}

\newtheorem{prop}[theorem]{Proposition}

\makeatletter \@addtoreset{equation}{section} \makeatother

\newenvironment{proof}{\noindent{\bf Proof}\/}%
{\hfill$\qed$\medskip}

\def\qed{ \ \vrule width.1cm height.3cm depth0cm}

\newenvironment{numitem}{\refstepcounter{equation}\begin{enumerate}%
\item[(\thesection.\arabic{equation})]$\quad$}{\end{enumerate}}
\newenvironment{numitem1}{\refstepcounter{equation}\begin{enumerate}%
\item[(\thesection.\arabic{equation})]}{\end{enumerate}}

\newcommand{\refeq}[1]{(\ref{eq:#1})}  

 \makeatletter
\renewcommand{\section}{\@startsection{section}{1}{0pt}%
{-3.5ex plus -1ex minus -.2ex}{2.3ex plus .2ex}%
{\normalfont\Large}}
 \makeatother

 \makeatletter
\renewcommand{\subsection}{\@startsection{subsection}{2}{0pt}%
{-3.0ex plus -1ex minus -.2ex}{1.5ex plus .2ex}%
{\normalfont\normalsize\bf}}
 \makeatother

\def\Rset{{\mathbb R}}
\def\Zset{{\mathbb Z}}

\def\Ascr{{\cal A}}
\def\Bscr{{\cal B}}
\def\Cscr{{\cal C}}
\def\Dscr{{\cal D}}
\def\Fscr{{\cal F}}

\def\Iscr{{\cal I}}
\def\Lscr{{\cal L}}

\def\Pscr{{\cal P}}

\def\Sscr{{\cal S}}
\def\Tscr{{\cal T}}
\def\Wscr{{\cal W}}
\def\Xscr{{\cal X}}
\def\Yscr{{\cal Y}}

\def\tilde{\widetilde}
\def\hat{\widehat}
\def\bar{\overline}
\def\eps{\epsilon}

\def\intlens{\;\hbox{\unitlength=1mm\begin{picture}(3,4)%
\put(1,0){$)$}%
\put(0.5,0){$($}%
\end{picture}}}

\def\apartlens{\;\hbox{\unitlength=1mm\begin{picture}(5,4)%
\put(0,0){$)$}%
\put(1.5,0){$($}%
\end{picture}}\;}

\begin{document}

 \begin{center}
{\large\bf Pl\"ucker environments, wiring and tiling diagrams, \\and weakly
separated set-systems}
 \end{center}

 \begin{center}
{\sc Vladimir~I.~Danilov}\footnote[2] {Central Institute of Economics and
Mathematics of the RAS, 47, Nakhimovskii Prospect, 117418 Moscow, Russia;
emails: danilov@cemi.rssi.ru (V.I.~Danilov); koshevoy@cemi.rssi.ru
(G.A.~Koshevoy).},
{\sc Alexander~V.~Karzanov}\footnote[3]{Institute for System Analysis of the
RAS, 9, Prospect 60 Let Oktyabrya, 117312 Moscow, Russia; email:
sasha@cs.isa.ru.},
{\sc Gleb~A.~Koshevoy}$^2$
\end{center}


 \begin{quote}
 {\bf Abstract.} \small
For the ordered set $[n]$ of $n$ elements, we consider the class $\Bscr_n$ of
bases $B$ of tropical Pl\"ucker functions on $2^{[n]}$ such that $B$ can be
obtained by a series of mutations (flips) from the basis formed by the
intervals in $[n]$. We show that these bases are representable by special
wiring diagrams and by certain arrangements generalizing rhombus tilings on the
$n$-zonogon. Based on the generalized tiling representation, we then prove that
each weakly separated set-system in $2^{[n]}$ having maximum possible size
belongs to $\Bscr_n$, thus answering affirmatively a conjecture due to Leclerc
and Zelevinsky. We also prove an analogous result for a hyper-simplex
$\Delta_n^m=\{S\subseteq[n]\colon |S|=m\}$.


  \end{quote}

\parskip=3pt


\section{\Large Introduction}  \label{sec:intr}

For a positive integer $n$, let $[n]$ denote the ordered set of elements
$1,2,\ldots,n$. In this paper we consider a certain ``class'' $\Bscr_n\subseteq
2^{2^{[n]}}$. The collections (set-systems) $B\subseteq 2^{[n]}$ constituting
$\Bscr_n$ have equal cardinalities $|B|$, and for some pairs of collections,
one can be obtained from the other by a single ``mutation'' (or ``flip'') that
consists in exchanging a pair of elements of a very special form in these
collections. The class we deal with arises, in particular, in a study of bases
of so-called tropical Pl\"ucker functions (this seems to be the simplest
source; one more source will be indicated later). For this reason, we may
liberally call $\Bscr_n$ along with mutations on it a {\em Pl\"ucker
environment}.

More precisely, let $f$ be a real-valued function on the subsets of $[n]$, or
on the Boolean cube $2^{[n]}$. Following~\cite{BFZ}, $f$ is said to be a {\em
tropical Pl\"ucker function}, or a {\em TP-function} for short, if it satisfies
  \begin{equation}  \label{eq:P3}
 f(Xik)+f(Xj)=\max\{f(Xij)+f(Xk),\; f(Xi)+f(Xjk)\}
  \end{equation}
for any triple $i<j<k$ in $[n]$ and any subset $X\subseteq [n]-\{i,j,k\}$.
Throughout, for brevity we write $Xi'\ldots j'$ instead of
$X\cup\{i'\}\cup\ldots\cup\{j'\}$. The set of TP-functions on $2^{[n]}$ is
denoted by $\Tscr\Pscr_n$. \medskip

\noindent {\bf Definition.} A subset $B\subseteq 2^{[n]}$ is called a {\em
TP-basis}, or simply a {\em basis}, if the restriction map
$res:\mathcal{TP}_n\to \mathbb{R}^{B}$ is a bijection. In other words, each
TP-function is determined by its values on $B$, and moreover, values on $B$ can
be chosen arbitrarily.\medskip

Such a basis does exist and the simplest instance is the set $\Iscr_n$ of all
intervals $\{p,p+1,\ldots,q\}$ in $[n]$ (including the empty set); see,
e.g.,~\cite{DKK-08}. In particular, the dimension of the polyhedral conic
complex $\mathcal{TP}_n$ is equal to $|\Iscr_n|=\binom{n+1}{2}+1$. The basis
$\Iscr_n$ is called {\em standard}.

(Note that the notion of a TP-function is extended to other domains, of which
most popular are an {\em integer box} ${\bf B}^{n,a}:=\{x\in \Zset^{[n]}\colon
0\le x\le a\}$ for $a\in\Zset^{[n]}$ and a {\em hyper-simplex}
$\Delta_n^m=\{S\subseteq [n]\colon |S|=m\}$ for $m\in\Zset$ (in the later
case,~\refeq{P3} should be replaced by a relation on quadruples $i<j<k<\ell$).
Aspects involving TP-functions are studied in~\cite{BFZ,K1,LZ,Po,Sc,Sp,SW} and
some other works. Generalizing some earlier known examples,~\cite{DKK-08}
constructs a TP-basis for a ``truncated integer box'' $\{x\in {\bf
B}^{n,a}\colon m\le x_1+\ldots+x_n\le m'\}$, where $0\le m\le m'\le n$. The
domains different from Boolean cubes are beyond the main part of this paper;
they will appear in Section~\ref{sec:gen} and Appendix.)

One can see that for a basis $B$, the collection $\{[n]-X\colon X\in B\}$ forms
a basis as well, called the {\em complementary basis} of $B$ and denoted by
co-$B$. An important instance is the collection co-$\Iscr_n$ of co-intervals in
$[n]$.

Once we are given a basis $B$ (e.g., the standard one), we can produce more
bases by making a series of elementary transformations relying on~\refeq{P3}.
More precisely, suppose there is a cortege $(X,i,j,k)$ such that the four sets
occurring in the right hand side of~\refeq{P3} and one set $Y\in\{Xj,Xik\}$ in
the left hand side belong to $B$. Then the replacement in $B$ of $Y$ by the
other set $Y'$ in the left hand side results in a basis $B'$ as well (and we
can further transform the latter basis in a similar way). The basis $B'$ is
said to be obtained from $B$ by the {\em flip} (or {\em mutation}) with respect
to $X,i,j,k$. When $Xj$ is replaced by $Xik$ (thus increasing the total size of
sets in the basis by 1), the flip is called {\em raising}. When $Xik$ is
replaced by $Xj$, the flip is called {\em lowering}. We write
$j\rightsquigarrow ik$ and $ik\rightsquigarrow j$ for such flips. The standard
basis $\Iscr_n$ does not admit lowering flips, whereas its complementary basis
co-$\Iscr_n$ does not admit raising flips.

We distinguish between two sorts of flip (mutations), which inspire
consideration of two classes of bases.\medskip

\noindent {\bf Definitions.} For a TP-basis $B$ and a cortege $(X,i,j,k)$ as
above, the flip (mutation) $j\rightsquigarrow ik$ or $ik\rightsquigarrow j$ is
called {\em strong} if both sets $X$ and $Xijk$ belong to $B$ as well, and {\em
weak} in general. (The former (latter) is also called the flip in the presence
of six (res. four) ``witnesses'', in terminology of~\cite{LZ}.) A basis is
called {\em normal} (by terminology in~\cite{DKK-08}) if it can be obtained by
a series of {\em strong} flips starting from $\Iscr_n$. A basis is called {\em
semi-normal} if it can be obtained by a series of {\em any} flips starting from
$\Iscr_n$.
\medskip

Leclerc and Zelevinsky~\cite{LZ} showed that the normal bases (in our
terminology) are exactly the collections $C\subseteq 2^{[n]}$ of maximum
possible size $|C|$ that possess the strong separation property (defined
later). Also the class of normal bases admits a nice ``graphical''
characterization, even for a natural generalization to the integer boxes
(see~\cite{DKK-08,HS}): such bases one-to-one correspond to the rhombus tilings
on the related {\em zonogon}.

Let $\Bscr_n$ denote the set of semi-normal TP-bases for the Boolean cube
$2^{[n]}$; this set (together with weak flips on its members) is just the
Pl\"ucker environment of our interest mentioned at the beginning. (Note that it
is still open at present whether there exists a non-semi-normal (or ``wild'')
basis; we conjecture that there is none.)

The first goal of this paper is to characterize $\Bscr_n$. We give two
characterizations for semi-normal bases: via a bijection to special collections
of curves, that we call {\em proper wirings}, and via a bijection to certain
graphical arrangements, called {\em generalized tilings\/} or, briefly, {\em
g-tilings} (in fact, these characterizations are interrelated via planar
duality). We associate to a proper wiring $W$ (a g-tiling $T$) a certain
collection of subsets of $[n]$ called its {\em spectrum}.  It turns out that
proper wirings and g-tilings are rigid objects, in the sense that any of these is
determined by its spectrum.

(More precisely, by a wiring we mean a set of $n$ directed
non-self-intersecting curves $w_1,\ldots,w_n$ in a region $R$ of the plane
homeomorphic to a circle, where each $w_i$ begins at a point $s_i$ and ends at
a point $s'_i$, and the points $s_1,\ldots,s_n,s'_1,\ldots,s'_n$ are different
and occur in this order clockwise in the boundary of $R$. A special wiring $W$
that we deal with is defined by three axioms (W1)--(W3). Axiom~(W1) is
standard, it says that $W$ preserves (topologically) under small deformations,
i.e., no three wires have a common point, any two wires meet at a finite number
of points and they cross, not touch, at each of these points. (W2) says that
the common points of $w_i,w_j$ follow in the opposed orders along these wires.
The crucial axiom~(W3) says that in the planar graph induced by $W$, there is a
certain bijection between the faces (``chambers'') whose boundary is a directed
cycle and the regions (``lenses'') surrounded by pieces of two wires between
their neighboring common points. $W$ is called proper if none of ``cyclic''
faces is a whole lens. The spectrum of $W$ is the collection of subsets
$X\subseteq [n]$ one-to-one corresponding to the ``non-cyclic'' faces $F$,
where $X$ consists of the elements $i$ such that $F$ ``lies on the right'' from
$w_i$ (according to the direction of $w_i$). When any two wires intersect
exactly once, the dual planar graph is realized by a rhombus tiling, and vice
versa (for a more general result of this sort, see~\cite{KS}). In a general
case, the construction of a g-tiling is more intricate. Axiom~(W2) is
encountered in~\cite{Po}. Another sort of wirings, related to set-systems in a
hyper-simplex $\Delta_n^m$, is studied in~\cite{Po,Sc}.)
\medskip

The characterization of semi-normal bases via generalized tilings helps answer
one conjecture of Leclerc and Zelevinsky concerning weakly separated set-systems;
this is the second goal of our work. Recall some definitions from~\cite{LZ}.
Hereinafter for sets $A,B$, we write~$A-B$ for $A\setminus B=\{e\colon A\ni
e\not\in B\}$. Let $X,Y\subseteq[n]$. We write $X\prec Y$ if $Y-X\ne\emptyset$
and $i<j$ for any $i\in X-Y$ and $j\in Y-X$ (which slightly differs from the
meaning of $\prec$ in~\cite{LZ}); note that this relation need not be transitive.
We write $X\rhd Y$ if $Y-X$ has a (unique) bipartition $\{Y_1,Y_2\}$ such that
$Y_1,Y_2,X-Y\ne\emptyset$ and $Y_1\prec X-Y\prec Y_2$.
\medskip

 \noindent {\bf Definitions.}
Sets $X,Y\subseteq[n]$ are called: (a) {\em strongly separated} if either
$X\prec Y$ or $Y\prec X$, and (b) {\em weakly separated} if either $X\prec Y$,
or $Y\prec X$, or $X\rhd Y$ and $|X|\ge |Y|$, or $Y\rhd X$ and $|Y|\ge |X|$.
Accordingly, a collection $C\subseteq 2^{[n]}$ is called strongly (weakly)
separated if any two members of $C$ are strongly (resp. weakly) separated.
\medskip

(As is seen from a discussion in~\cite{LZ}, an interest in studying weakly
separated collections is inspired, in particular, by the problem of
characterizing all families of quasicommuting quantum flag minors, which in
turn comes from exploration of Lusztig's canonical bases for certain quantum
groups. It is proved in~\cite{LZ} that, in an $n\times n$ generic $q$-matrix,
the flag minors with column sets $I,J\subseteq[n]$ quasicommute if and only if
the sets $I,J$ are weakly separated. See also~\cite{L}.)

Important properties shown in~\cite{LZ} are that any weakly separated
collection $C\subseteq 2^{[n]}$ has cardinality at most $\binom{n+1}{2}+1$ and
that the set of such collections is closed under weak flips (which are defined
as for TP-bases above). Let $\Cscr_n$ denote the set of {\em largest} weakly
separated collections in $[n]$, i.e., having size $\binom{n+1}{2}+1$. It turns
into a poset by regarding $C$ as being less than $C'$ if $C$ is obtained from
$C'$ by a series of weak lowering flips. This poset contains $\Iscr_n$ and
co-$\Iscr$ as minimal and maximal elements, respectively, and it is conjectured
in~\cite[Conjecture~1.8]{LZ} that there are no other minimal and maximal
elements in it. This would imply that $\Cscr_n$ coincides with $\Bscr_n$. We
prove this conjecture.

The main results in this paper are summarized as follows.

 \medskip  \noindent {\bf Theorem A} (main)
{\em For $B\subseteq 2^{[n]}$, the following statements are equivalent:

{\rm (i)} $B$ is a semi-standard TP-basis;

{\rm (ii)} $B$ is the spectrum of a proper wiring;

{\rm (iii)} $B$ is the spectrum of a generalized tiling;

{\rm (iv)} $B$ is a largest weakly separated collection.}
  \medskip

The paper is organized as follows. Section~\ref{sec:defin} contains basic
definitions and states two results involved in Theorem~A. It introduces the
notions of proper wirings and generalized tilings, claims the equivalence of
(i) and (ii) in the above theorem (Theorem~\ref{tm:bas-wir}) and claims the
equivalence of (i) and (iii) (Theorem~\ref{tm:bas-til}).
Section~\ref{sec:Tprop} describes some ``elementary'' properties of g-tilings
that will be used later. The combined proof of Theorems~\ref{tm:bas-wir}
and~\ref{tm:bas-til} consists of four stages and is lasted throughout
Sections~\ref{sec:bas-to-til}--\ref{sec:wir-to-til}. In fact, g-tilings are the
central objects of treatment in the paper; we take advantages from their nice
graphical visualization and structural features, and all implications that we
explicitly prove involve just g-tilings. (Another preference of g-tilings is
that they admit ``local'' defining axioms; see Remark~1 in
Section~\ref{sec:Tprop}.) Implication (i)$\to$(iii) in Theorem~A is proved in
Section~\ref{sec:bas-to-til}, (iii)$\to$(i) in Section~\ref{sec:til-to-bas},
(iii)$\to$(ii) in Section~\ref{sec:til-to-wir}, and (ii)$\to$(iii) in
Section~\ref{sec:wir-to-til}. Section~\ref{sec:contr-exp} establishes important
interrelations between g-tilings in dimensions $n$ and $n-1$ (giving, as a
consequence, a relation between the classes $\Bscr_n$ and $\Bscr_{n-1}$). Here
we describe two operations, called the $n$-{\em contraction} and $n$-{\em
expansion}; the former canonically transforms a g-tiling for $n$ into one for
$n-1$, and the latter is applied to a pair consisting of a g-tiling for $n-1$
and a certain path in it and transforms this pair into a g-tiling for $n$.
These operations are essentially used in Section~\ref{sec:wsf} where we prove
(iv)$\to$(iii) by induction on $n$, thus answering Leclerc-Zelevinsky's
conjecture mentioned above. This completes the proof of Theorem~A, taking into
account that (i)$\to$(iv) was established in~\cite{LZ}. Section~\ref{sec:gen}
discusses two generalizations of our theorems: to an integer box and to an
arbitrary permutation of $[n]$. In Appendix we show that the equivalence
(i)$\Longleftrightarrow$(iv) as in Theorem~A is valid when, instead of TP-bases
and largest weakly separated collections in $2^{[n]}$, one considers a natural
class of bases of tropical Pl\"ucker functions on a hyper-simplex $\Delta_n^m$
and the weakly separated collections of maximum possible cardinality in
$\Delta_n^m$. \smallskip

{\em Acknowledgement.} We thank Alexander Postnikov who drew our attention to
paper~\cite{LZ} and a possible relation between TP-bases and weakly-separated
set-systems.


\section{\Large Wirings and tilings}  \label{sec:defin}

Throughout the paper we assume that $n>1$. This section gives precise
definitions of the objects that we call special wiring and generalized tiling
diagrams. Such diagrams live within a zonogon, which is defined as follows.

In the upper half-plane $\Rset\times \Rset_+$, take $n$ non-colinear vectors
$\xi_1,\ldots,\xi_n$ so that:
  \begin{numitem} (i) $\xi_1,\ldots,\xi_n$ follow in this order clockwise around
$(0,0)$, and (ii) all integer combinations of these vectors are different.
  \label{eq:xi}
  \end{numitem}
Then the set
  $$
Z=Z_n:=\{\lambda_1\xi_1+\ldots+ \lambda_n\xi_n\colon \lambda_i\in\Rset,\;
0\le\lambda_i\le 1,\; i=1,\ldots,n\}
  $$
is a $2n$-gone. Moreover, $Z$ is a {\em zonogon}, as it is the sum of $n$
line-segments $\{\lambda\xi_i\colon 1\le \lambda\le 1\}$, $i=1,\ldots,n$. Also
it is the image by a linear projection $\pi$ of the solid cube $conv(2^{[n]})$
into the plane $\Rset^2$, defined by $\pi(x)=x_1\xi_1+\ldots +x_n\xi_n.$ The
boundary $bd(Z)$ of $Z$ consists of two parts: the {\em left boundary} $lbd(Z)$
formed by the points (vertices) $p_i:=\xi_1+\ldots+\xi_i$ ($i=0,\ldots,n$)
connected by the line-segments $p_{i-1}p_i:=p_{i-1}+\{\lambda \xi_i\colon
0\le\lambda\le 1\}$, and the {\em right boundary} $rbd(Z)$ formed by the points
$p'_i:=\xi_i+\ldots+\xi_n$ ($i=0,\ldots,n$) connected by the segments
$p'_ip'_{i-1}$. So $p_0=p'_n$ is the minimal vertex and $p_n=p'_0$ is the
maximal vertex of $Z$. We orient each segment $p_{i-1}p_i$ from $p_{i-1}$ to
$p_i$ and orient each segment $p'_ip'_{i-1}$ from $p'_i$ to $p'_{i-1}$. Let
$s_i$ (resp. $s'_i$) denote the median point in the segment $p_{i-1}p_i$ (resp.
$p'_ip'_{i-1}$).

Although the generalized tiling model will be used much more extensively later
on, we prefer to start with describing the special wiring model, which looks
more transparent.

\subsection{Wiring diagrams}  \label{ssec:wiring}

A {\em special wiring diagram}, also called a {\em W-diagram} or a {\em wiring}
for brevity, is an ordered collection $W$ of $n$ wires $w_1,\ldots,w_n$
satisfying three axioms below. A {\em wire} $w_i$ is a continuous injective map
of the segment $[0,1]$ into $Z$ (or the curve in the plane represented by this
map) such that $w_i(0)=s_i$, $w_i(1)=s'_i$, and $w_i(\lambda)$ lies in the
interior of $Z$ for $0<\lambda<1$. We say that $w_i$ begins at $s_i$ and ends
at $s'_i$, and orient $w_i$ from $s_i$ to $s'_i$. The diagram $W$ is considered
up to a homeomorphism of $Z$ stable on $bd(Z)$, and up to parameterizations of
the wires. Axioms (W1)--(W3) specify $W$ as follows.
  \begin{itemize}
\item[(W1)] No three wires $w_i,w_j,w_k$ have a common point, i.e., there are no
$\lambda,\lambda',\lambda''$ such that
$w_i(\lambda)=w_j(\lambda')=w_k(\lambda'')$. Any two wires $w_i,w_j$ ($i\ne j$)
intersect at a finite number of points, and at each of their common points $v$,
the wires {\em cross\/}, not touch (i.e., when passing $v$, the wire $w_i$ goes
from one connected component of $Z- w_j$ to the other one).
  \end{itemize}
  \begin{itemize}
\item[(W2)] for $1\le i<j\le n$, the common points of $w_i,w_j$ follow in opposed orders
along these wires, i.e., if $w_i(\lambda_q)=w_j(\lambda'_{q})$ for
$q=1,\ldots,r$ and if $\lambda_1<\ldots<\lambda_r$, then
$\lambda'_1>\ldots>\lambda'_r$.
  \end{itemize}

Since the order of $s_i,s_j$ in $\ell bd(Z)$ is different from the order of
$s'_i,s'_j$ in $rbd(Z)$, wires $w_i,w_j$ do intersect; moreover, the number
$r=r_{ij}$ of their common points is odd. Assuming that $i<j$, we denote these
points as $x_{ij}(1),\ldots,x_{ij}(r)$ following the direction of $w_i$ from
$w_i(0)$ to $w_i(1)$. When $r>1$, the (bounded) region in the plane surrounded
by the pieces of $w_i,w_j$ between $x_{ij}(q)$ and $x_{ij}(q+1)$ (where
$q=1,\ldots,r-1$) is denoted by $L_{ij}(q)$ and called the $q$-th {\em lens}
for $i,j$. The points $x_{ij}(q)$ and $x_{ij}(q+1)$ are regarded as the {\em
upper} and {\em lower} points of $L_{ij}(q)$, respectively. When $q$ is odd
(even), we say that $L_{ij}(q)$ is an {\em odd} (resp. {\em even}) lens. Note
that at each point $x_{ij}(q)$ with $q$ odd the wire with the bigger number,
namely, $w_j$, crosses the wire with the smaller number ($w_i$) {\em from left
to right} w.r.t. the direction of the latter; we call such a point {\em white}.
In contrast, when $q$ is even, $w_j$ crosses $w_i$ at $x_{ij}(q)$ {\em from
right to left}; in this case (which will be of especial interest), we call
$x_{ij}(q)$ {\em black}, or {\em orientation-reversing}, and say that this
point is the {\em root} of the lenses $L_{ij}(q-1)$ and $L_{ij}(q)$. In the
simplest case, when any two distinct wires intersect exactly once, there are no
lenses at all and all intersection points for $W$ are white. (The adjectives
``white'' and ``black'' for intersection points of wires will match terminology
that we use for corresponding elements of tilings.)

The wiring $W$ is associated, in a natural way, with a planar directed graph
$G_W$ embedded in $Z$. The vertices of $G_W$ are the points $p_i,p'_i,s_i,s'_i$
and the intersection points of wires. The edges of $G_W$ are the corresponding
directed line-segments in $bd(Z)$ and the pieces of wires between neighboring
points of intersection with other wires or with the boundary, which are
directed according to the direction of wires. We say that an edge contained in
a wire $w_i$ has {\em color} $i$, or is an $i$-{\em edge}. Let $\Fscr_W$ be the
set of (inner, or bounded) faces of $G_W$. Here each face $F$ is considered as
the closure of a maximal connected component in $Z-\cup(w\in W)$. We say that a
face $F$ is {\em cyclic} if its boundary $bd(F)$ is a directed cycle in $G_W$.
  \begin{itemize}
\item[(W3)] There is a bijection $\phi$ between the set $\Lscr(W)$ of
lenses in $W$ and the set $\Fscr^{cyc}_W$ of cyclic faces in $G_W$. Moreover,
for each lens $L$, $\phi(L)$ is the (unique) face lying in $L$ and containing
its root.
  \end{itemize}

We say that $W$ is {\em proper} if none of cyclic faces is a whole lens, i.e.,
for each lens $L\in\Lscr(W)$, there is at least one wire going across $L$. An
instance of proper wirings for $n=4$ is illustrated in the picture; here the
cyclic faces are marked by circles and the black rhombus indicates the black
point.

 \begin{center}
  \unitlength=1mm
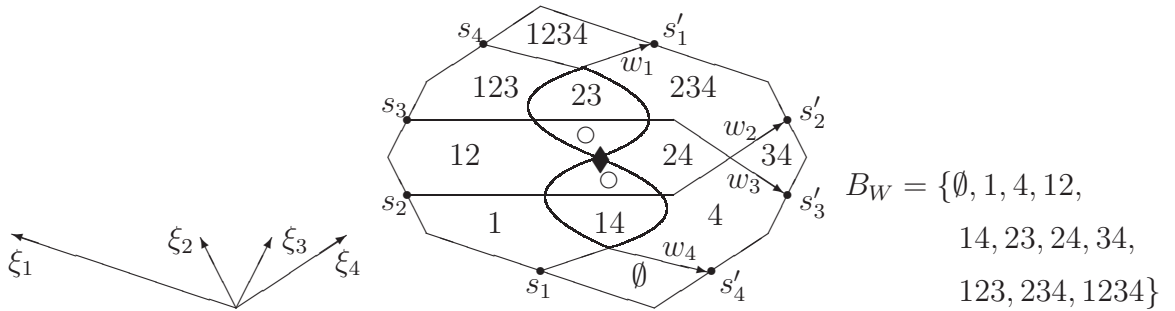
  \begin{picture}(150,40)
   \put(30,0){\begin{picture}(45,13)
  \put(0,0){\vector(-3,1){29.5}}
  \put(0,0){\vector(-1,2){4.7}}
  \put(0,0){\vector(1,2){4.7}}
  \put(0,0){\vector(3,2){14.5}}
  \put(-30,5){$\xi_1$}
  \put(-9,8){$\xi_2$}
  \put(6,8){$\xi_3$}
  \put(13,5){$\xi_4$}
    \end{picture}}
%
   \put(85,0){\begin{picture}(50,40)
  \put(0,0){\line(-3,1){30}}
  \put(-30,10){\line(-1,2){5}}
  \put(-35,20){\line(1,2){5}}
  \put(-30,30){\line(3,2){15}}
  \put(15,30){\line(-3,1){30}}
  \put(20,20){\line(-1,2){5}}
  \put(15,10){\line(1,2){5}}
  \put(0,0){\line(3,2){15}}
  \put(-15,5){\circle*{1}}
  \put(7.5,5){\circle*{1}}
  \put(-32.5,15){\circle*{1}}
  \put(17.5,15){\circle*{1}}
  \put(-32.5,25){\circle*{1}}
  \put(17.5,25){\circle*{1}}
  \put(-22.5,35){\circle*{1}}
  \put(0,35){\circle*{1}}
  \put(-15,5){\line(3,1){9}}
  \put(-6,8){\vector(4,-1){13}}
  \put(-9.5,32){\vector(3,1){9}}
  \put(-22.5,35){\line(4,-1){13}}
  \put(-32.5,15){\line(1,0){35}}
  \put(-32.5,25){\line(1,0){35}}
  \put(2.5,15){\vector(3,2){14.5}}
  \put(2.5,25){\vector(3,-2){14.5}}
  \qbezier(-6,8)(10,13)(-7.5,20)
  \qbezier(-6,8)(-22,15)(-7.5,20)
  \qbezier(-9,32)(-25,27)(-7.5,20)
  \qbezier(-9.5,32)(7,25)(-7.5,20)
  \put(-8.5,18.5){$\blacklozenge$}
  \put(-6,17){\circle{2}}
  \put(-9,23){\circle{2}}
  \put(-3,3){$\emptyset$}
  \put(-22,10){1}
  \put(7,11){4}
  \put(-8,10){14}
  \put(-27,19){12}
  \put(1,19){24}
  \put(14,19){34}
  \put(-24,28){123}
  \put(-11,27){23}
  \put(2,28){234}
  \put(-17,35){1234}
  \put(-17,2){$s_1$}
  \put(-36,13){$s_2$}
  \put(-36,26){$s_3$}
  \put(-26,36){$s_4$}
  \put(1,7){$w_4$}
  \put(9.5,16){$w_3$}
  \put(9,23){$w_2$}
  \put(-4.5,31.5){$w_1$}
  \put(1,36){$s'_1$}
  \put(19,25){$s'_2$}
  \put(19,13){$s'_3$}
  \put(8.5,2){$s'_4$}
    \end{picture}}
%
   \put(110,0){\begin{picture}(30,30)
  \put(0,15){$B_W=\{\emptyset,1,4,12,$}
  \put(15,8){$14,23,24,34,$}
  \put(15,1){$123,234,1234\}$}
    \end{picture}}
  \end{picture}
   \end{center}

Now we associate to $W$ a set-system $B_W\subseteq 2^{[n]}$ as follows. For each
face $F$, let $X(F)$ be the set of elements $i\in[n]$ such that $F$ lies {\em on
the left} from the wire $w_i$, i.e., $F$ and the maximal point $p_n$ lie in the
same of the two connected components of $Z- w_i$. We define
  $$
  B_W:=\{X\subseteq [n] \colon X=X(F)\;\; \mbox{for some}\;\;
  F\in\Fscr_W-\Fscr^{cyc}_W\},
  $$
referring to it as the {\em effective spectrum}, or simply the {\em spectrum}
of $W$. Sometimes it will also be useful to consider the {\em full spectrum}
$\hat B_W$ consisting of all sets $X(F)$, $F\in\Fscr_W$. (In fact, when $W$ is
proper, all sets in $\hat B_W$ are different; see Lemma~\ref{lm:nocopies}. When
$W$ is not proper, there are different faces $F,F'$ with $X(F)=X(F')$. We can
turn $W$ into a proper wiring $W'$ by getting rid, step by step, of lenses
forming faces (by making a series of Reidemeister moves of type II, namely,
\intlens$\to$\apartlens operations). This preserves the effective spectrum:
$B_{W'}=B_W$, whereas the full spectrum may decrease.)

Note that when any two wires intersect at exactly one point (i.e., when no
black points exist), $B_W$ is a normal basis, and conversely, any normal basis
is obtained in this way (see~\cite{DKK-08}).

Our main result on wirings is the following
 \begin{theorem} \label{tm:bas-wir}
For any proper wiring $W$ (obeying (W1)--(W3)), the spectrum $B_W$ is a
semi-normal TP-basis. Conversely, for any semi-normal TP-basis $B$, there
exists a proper wiring $W$ such that $B_W=B$.
 \end{theorem}

This theorem will be obtained in
Sections~\ref{sec:til-to-wir}--\ref{sec:wir-to-til}.

\subsection{Generalized tilings}  \label{ssec:tiling}

When it is not confusing, we identify a subset $X\subseteq [n]$ with the
corresponding vertex of the $n$-cube and with the point $\sum_{i\in X}\xi_i$ in
the zonogon $Z$. Due to~\refeq{xi}(ii), all points $X$ are different
(concerning $Z$).

Assuming that the vectors $\xi_i$ have the same Euclidean norm, a {\em rhombus
tiling diagram} is defined to be a subdivision $T$ of $Z$ into rhombi of the
form $x+\{\lambda\xi_i+\lambda'\xi_j\colon 0\le \lambda,\lambda'\le 1\}$ for
some $i<j$ and a point $x$ in $Z$, i.e., the rhombi are pairwise
non-overlapping (have no common interior points) and their union is $Z$.
From~\refeq{xi}(ii) it follows that for $i,j,x$ as above, $x$ represents a
subset in $[n]-\{i,j\}$. The diagram $T$ is also regarded as a directed planar
graph whose vertices and edges are the vertices and side segments of the
rhombi, respectively. An edge connecting $X$ and $Xi$ is directed from the
former to the latter. It is shown in~\cite{DKK-08,HS} that the vertex set of
$T$ forms a normal basis and that each normal basis is obtained in this way.

It makes no difference whether we take vectors $\xi_1,\ldots,\xi_n$ with equal
or arbitrary norms (subject to~\refeq{xi}); to simplify technical details and
visualization, throughout the paper we will assume that these vectors have {\em
unit height}, i.e., each $\xi_i$ is of the form $(x,1)$. Then we obtain a
subdivision $T$ of $Z$ into parallelograms of height 2, and for convenience we
refer to $T$ as a {\em tiling} and to its elements as {\em tiles}. A tile
$\tau$ defined by $X,i,j$ (with $i<j$) is called an $ij$-{\em tile} at $X$ and
denoted by $\tau(X;i,j)$. According to a natural visualization of $\tau$, its
vertices $X,Xi,Xj,Xij$ are called the {\em bottom, left, right, top} vertices
of $\tau$ and denoted by $b(\tau)$, $\ell(\tau)$, $r(\tau)$, $t(\tau)$,
respectively. Also we will say that: for a point (subset) $Y\subseteq[n]$,
~$|Y|$ is the {\em height} of $Y$; the set of vertices of tiles in $T$ having
height $h$ form $h$-th {\em level}; and a point $Y$ {\em lies on the right}
from a point $Y'$ if $Y,Y'$ have the same height and $\sum_{i\in Y}\xi_i\ge
\sum_{i\in Y'}\xi_i$.

In a {\em generalized tiling}, or a {\em g-tiling}, the union of tiles is again
$Z$ but some tiles may overlap. It is a collection $T$ of tiles which is
partitioned into two subcollections $T^w$ and $T^b$, of {\em white} and {\em
black\/} tiles (say), respectively, obeying axioms (T1)--(T4) below. When
$T^b=\emptyset$, we will obtain a tiling as before, for convenience referring
to it as a {\em pure tiling}. Let $V_T$ and $E_T$ denote the sets of vertices
and edges, respectively, occurring in tiles of $T$, not counting
multiplicities. For a vertex $v\in V_T$, the set of edges incident with $v$ is
denoted by $E_T(v)$, and the set of tiles having a vertex at $v$ is denoted by
$F_T(v)$.
  \begin{itemize}
\item[(T1)] All tiles are contained in $Z$. Each boundary edge of $Z$ belongs to
exactly one tile. Each edge in $E_T$ not contained in $bd(Z)$ belongs to
exactly two tiles. All tiles in $T$ are different (in the sense that no two
coincide in the plane).
  \end{itemize}
  \begin{itemize}
\item[(T2)] Any two white tiles having a common edge do not overlap (in the
sense that they have no common interior point). If a white tile and a black
tile share an edge, then these tiles do overlap. No two black tiles share an
edge.
  \end{itemize}
  \begin{itemize}
\item[(T3)] Let $\tau$ be a black tile. None of $b(\tau),t(\tau)$ is a vertex of
another black tile. All edges in $E_T(b(\tau))$ {\em leave} $b(\tau)$ (i.e.,
are directed from $b(\tau)$). All edges in $E_T(t(\tau))$ {\em enter} $t(\tau)$
(i.e., are directed to $t(\tau)$).
  \end{itemize}

We distinguish between three sorts of vertices by saying that $v\in V_T$ is:
(a) a {\em terminal} vertex if it is the bottom or top vertex of some black
tile; (b) an {\em ordinary} vertex if all tiles in $F_T(v)$ are white; and (c)
a {\em mixed} vertex otherwise (i.e. $v$ is the left or right vertex of some
black tile). Note that a mixed vertex may belong, as the left or right vertex,
to several black tiles.

Each tile $\tau\in T$ is associated, in a natural way, to a square in the solid
$n$-cube $conv(2^{[n]})$, denoted by $\sigma(\tau)$: if $\tau=\tau(X;i,j)$ then
$\sigma(\tau)$ spans the vertices (corresponding to) $X,Xi,Xj,Xij$ in the cube.
In view of (T1), the interiors of these squares are disjoint, and
$\cup(\sigma(\tau)\colon \tau\in T)$ forms a 2-dimensional surface, denoted by
$D_T$, whose boundary is the preimage by $\pi$ of the boundary of $Z$; the
vertices in $bd(D_T)$ correspond to the {\em principal intervals} $\emptyset$,
$[q]$ and $[q..n]$ for $q=1,\ldots,n$. (For $1\le p\le r\le n$, we denote the
interval $\{p,p+1,\ldots,r\}$ by ~$[p..r]$). The last axiom is:
  \begin{itemize}
\item[(T4)] $D_T$ is a disc (i.e., is homeomorphic to
 $\{x\in\Rset^2\colon x_1^2+x_2^2\le 1\}$).
  \end{itemize}

The {\em reversed} g-tiling $T^{rev}$ of a g-tiling $T$ is formed by
replacing each tile $\tau(X;i,j)$ of $T$ by the tile $\tau([n]-Xij;i,j)$
(or by changing the orientation of all edges in $E_T$, in particular, in
$bd(Z)$). Clearly (T1)--(T4) remain valid for $T^{rev}$.

The {\em effective spectrum}, or simply the {\em spectrum}, of a g-tiling $T$
is the collection $B_T$ of (subsets of $[n]$ represented by) {\em non-terminal}
vertices in $T$. The {\em full spectrum} $\hat B_T$ is formed by all vertices
in $T$. An example of g-tilings for $n=4$ is drawn in the picture, where the
unique black tile is indicated by thick lines and the terminal vertices are
surrounded by circles (this is related to the wiring shown on the previous
picture).

 \begin{center}
  \unitlength=1mm
  \begin{picture}(130,40)
   \put(40,0){\begin{picture}(50,40)
  \put(0,0){\vector(-3,1){29.5}}
  \put(-30,10){\vector(-1,2){4.7}}
  \put(-35,20){\vector(1,2){4.7}}
  \put(-30,30){\vector(3,2){14.7}}
  \put(15,30){\vector(-3,1){29.5}}
  \put(20,20){\vector(-1,2){4.7}}
  \put(15,10){\vector(1,2){4.7}}
  \put(0,0){\vector(3,2){14.7}}
  \put(0,0){\circle*{1.5}}
  \put(-30,10){\circle*{1.5}}
  \put(-5,10){\circle{3}}
  \put(15,10){\circle*{1.5}}
  \put(-35,20){\circle*{1.5}}
  \put(-15,20){\circle*{1.5}}
  \put(0,20){\circle*{1.5}}
  \put(10,20){\circle*{1.5}}
  \put(20,20){\circle*{1.5}}
  \put(-30,30){\circle*{1.5}}
  \put(-20,30){\circle{3}}
  \put(15,30){\circle*{1.5}}
  \put(-15,40){\circle*{1.5}}
  \put(0,20){\vector(-3,1){29.5}}
  \put(15,10){\vector(-3,1){29.5}}
  \put(-15,20){\vector(-1,2){4.7}}
  \put(15,10){\vector(-1,2){4.7}}
  \put(-5,10){\vector(1,2){4.7}}
  \put(10,20){\vector(1,2){4.7}}
  \put(-30,10){\vector(3,2){14.7}}
  \put(0,20){\vector(3,2){14.7}}
\thicklines{
  \put(-5,10){\vector(-3,1){29.5}}
  \put(10,20){\vector(-3,1){29.5}}
  \put(-35,20){\vector(3,2){14.7}}
  \put(-5,10){\vector(3,2){14.7}}
}
  \put(-1,2){$\emptyset$}
  \put(-33,8){1}
  \put(17,8){4}
  \put(-40,18){12}
  \put(-21,20){14}
  \put(-6,17.5){23}
  \put(12,19){24}
  \put(22,18){34}
  \put(-37,30){123}
  \put(17,30){234}
  \put(-18,35){1234}
    \end{picture}}
%
   \put(80,0){\begin{picture}(50,30)
  \put(0,25){$B_T=\{\emptyset,1,4,12,14,23,24,34,$}
  \put(15,18){$123,234,1234\}$}
    \end{picture}}
  \end{picture}
   \end{center}
Our main result on g-tilings is the following
 \begin{theorem} \label{tm:bas-til}
For any generalized tiling $T$ (obeying (T1)--(T4)), the spectrum $B_T$ is a
semi-normal TP-basis. Conversely, for any semi-normal TP-basis $B$, there
exists a generalized tiling $T$ such that $B_T=B$.
 \end{theorem}
(In particular, the cardinalities of the spectra of all g-tilings on $Z_n$ are
the same and equal to $\binom{n+1}{2}+1$.) The first part of this theorem will
be proved in Section~\ref{sec:til-to-bas}, and the second one in
Section~\ref{sec:bas-to-til}.

We will explain in Section~\ref{sec:wir-to-til} that for each semi-normal basis
$B$, there are precisely one proper wiring $W$ and precisely one g-tiling $T$
such that $B_W=B_T=B$ (see Theorem~\ref{tm:unique}); this is similar to the
one-to-one correspondence between the normal bases and pure tilings.

\section{Elementary properties of generalized tilings}  \label{sec:Tprop}

In this section we give additional definitions and notation and demonstrate
several corollaries from axioms~(T1)--(T4) which will be used later on. Let $T$
be a g-tiling on $Z=Z_n$.
\smallskip

{\bf 1.} An edge $e$ of $G_T$ is called {\em black} if there is a black tile
containing $e$ (as a side edge); otherwise $e$ is called {\em white}. The sets
of white and black edges incident with a vertex $v$ are denoted by $E_T^w(v)$
and $E_T^b(v)$, respectively. For a vertex $v$ of a tile $\tau$, let
$C(\tau,v)$ denote the minimal cone at $v$ containing $\tau$ (i.e., generated
by the pair of edges of $\tau$ incident to $v$), and let $\alpha(\tau,v)$
denote the angle of this cone taken with sign $+$ if $\tau$ is white, and $-$
if $\tau$ is black. The sum $\sum(\alpha(\tau,v)\colon \tau\in F_T(v))$ is
called the {\em (full) rotation angle} at $v$ and denoted by
$\rho(v)=\rho_T(v)$. We observe from~(T1)--(T3) that terminal vertices behave
as follows.
  \begin{corollary} \label{cor:termv}
Let $v$ be a terminal vertex belonging to a black ~$ij$-tile $\tau$. Then:

{\rm(i)} $v$ is not connected by edge with another terminal vertex (whence
$|E^b_T(v)|=2$);

{\rm(ii)} $|E_T(v)|\ge 3$ (whence $E^w_T(v)\ne\emptyset$);

{\rm(iii)} each edge $e\in E^w_T(v)$ lies in the cone $C(\tau,v)$ (whence $e$
is a $q$-edge for some $i<q<j$);

{\rm(iv)} $\rho(v)=0$;

{\rm(v)} $v$ does not belong to the boundary of $Z$ (whence any tile containing
a boundary edge of $Z$ is white).
  \end{corollary}

Indeed, since each edge of $G_T$ belongs to some tile, at least one of its end
vertices has both entering and leaving edges, and therefore (by~(T3)), this
vertex cannot be terminal (yielding (i)). Next, if $|E_T(v)|=2$, then $F_T(v)$
would consist only of the tile $\tau$ and its white copy; this is not the case
by~(T1) (yielding~(ii)). Next, assume that $v=t(\tau)$. Then $v$ is the top
vertex of all tiles in $F_T(v)$ (by~(T3)). This together with the facts that
all tiles in $F_T(v)-\{\tau\}$ are white and that any two white tiles sharing
an edge do not overlap (by~(T2)) implies~(iii) and~(iv). When $v=b(\tau)$, the
argument is similar. Finally, $v$ cannot be a boundary vertex $p_k$ or $p'_k$
for $k\ne 0,n$ since the latter vertices have both entering and leaving edges.
In case $v=p_0$, the tile $\tau$ would contain both boundary edges $p_0p_1$ and
$p_0p'_{n-1}$ (in view of~(iii)). But then the white tile sharing with $\tau$
the edge $(r(\tau),t(\tau))$ would trespass the boundary of $Z$. The case
$v=p_n$ is impossible for a similar reason (yielding~(v)).

Note that~(iii) in this corollary implies that
   \begin{numitem1}
if a black $ij$-tile $\tau$ and a (white) tile $\tau'$ share an edge $e$,
then $\tau'$ is either an $iq$-tile or a $qj$-tile for some $i<q<j$; also
$\tau'\subset C(\tau,v)$ and $\tau\subset C(\tau',v')$, where $v$ and $v'$
are the terminal and non-terminal ends of $e$, respectively.
   \label{eq:termv}
   \end{numitem1}

{\bf 2.} The following important lemma specifies the rotation angle at
non-terminal vertices.
  \begin{lemma} \label{lm:rotation}
Let $v\in V_T$ be a non-terminal vertex.

{\rm(i)} If $v$ belongs to $bd(Z)$, then $\rho(v)$ is positive and equals the
angle between the boundary edges incident to $v$.

{\rm(ii)} If $v$ is inner (i.e., not in $bd(Z)$), then $\rho(v)=2\pi$.
  \end{lemma}
  \begin{proof}
(i) For $v\in bd(Z)$, let $e,e'$ be the boundary edges incident to $v$, where
$e,Z,e'$ follow clockwise around $v$. Consider the maximal sequence
$e=e_0,\tau_1,e_1,\ldots,\tau_r,e_r$ of edges in $E_T(v)$ and tiles in $F_T(v)$
such that for $q=1,\ldots,r$, ~$e_{q-1},e_q$ are distinct edges of the tile
$\tau_q$, and $\tau_q\ne \tau_{q+1}$ (when $q<n$). Using~\refeq{termv}, one can
see that all tiles in this sequence are different and give the whole $F_T(v)$;
also $e_r=e'$ and the tiles $\tau_1,\tau_r$ are white. For each $q$, the ray at
$v$ containing $e_q$ is obtained by rotating the ray at $v$ containing
$e_{q-1}$ by the angle $\alpha(\tau_q,v)$ (where the rotation is clockwise if
the angle is positive). So the sum of angles over this sequence amounts to
$\rho(v)$ and is equal to the angle of $e,e'$.

To show (ii), let $V:=V_T$ and $E:=E_T$. Also denote the set of terminal
vertices by $V^t$, and the set of {\em inner} non-terminal vertices by $\hat
V$. Since the boundary of $Z$ contains $2n$ vertices and by~(i),
  \begin{equation} \label{eq:nu-rot}
  |V|=|V^t|+|\hat V|+2n\quad \mbox{and}\quad
  \sum\nolimits_{v\in V\cap bd(Z)} \rho(v)=\pi\cdot 2n-2\pi=2\pi(n-1)
  \end{equation}

Let $\Sigma:=\sum(\rho(v)\colon v\in V)$ and $\hat\Sigma:=\sum(\rho(v)\colon
v\in \hat V)$. The contribution to $\Sigma$ from each white (black) tile is
$2\pi$ (resp. $-2\pi$). Therefore, $\Sigma=2\pi(|T^w|-|T^b|)$. On the other
hand, in view of Corollary~\ref{cor:termv}(iv) and the second relation
in~\refeq{nu-rot}, $\Sigma=\hat\Sigma+2\pi(n-1)$. Then
   \begin{equation} \label{eq:hatSigma}
   \hat\Sigma=2\pi(|T^w|-|T^b|-n+1).
   \end{equation}

Considering $G_T$ as a planar graph properly embedded in the disc $D_T$ and
applying Euler formula to it, we have $|V|+|T|=|E|+1$. Each tile has four
edges, the number of boundary edges is $2n$, and each inner edge belongs to two
tiles; therefore, $|E|=2n+(4|T|-2n)/2=2|T|+n$. Then $|V|$ is expressed as
   \begin{equation} \label{eq:exp_nu}
  |V|=|E|-|T|+1=2|T|+n-|T|+1=|T|+n+1.
   \end{equation}

Also $|V|=|\hat V|+2|T^b|+2n$ (using the first equality in~\refeq{nu-rot} and
the equality $|V^t|=2|T^b|$). This and~\refeq{exp_nu} give
 $$
    |\hat V|=|V|-2|T^b|-2n=(|T|+n+1)-2|T^b|-2n=|T^w|-|T^b|-n+1.
    $$
Comparing this with~\refeq{hatSigma}, we obtain $\hat\Sigma=2\pi|\hat V|$. Now
the desired equality $\rho(v)=2\pi$ for each vertex $v\in\hat V$ follows from
the fact that $\rho(v)$ equals $2\pi\cdot d$ for some integer $d\ge 1$. The
latter is shown as follows. Let us begin with a white tile $\tau_1\in F_T(v)$
and its edges $e_0,e_1\in E_T(v)$, in this order clockwise, and form a sequence
$e_0,\tau_1,e_1,\ldots,\tau_r,e_r,\ldots$ ~similar to that in~(i) above, until
we return to the initial edge $e_0$. Let $R_q$ be the ray at $v$ containing
$e_q$. Since $\alpha(\tau_q)>0$ when $\tau_q$ is white, and
$\alpha(\tau_q)+\alpha(\tau_{q+1})>0$ when $\tau_q$ is white and $\tau_{q+1}$
is black (cf.~\refeq{termv}), the current ray $R_\bullet$ must make at least
one turn clockwise before it returns to the initial ray $R_0$. If it happens
that the sequence uses not all tiles in $F_T(v)$, we start with a new white
tile to form a next sequence (for which the corresponding ray makes at least
one turn clockwise as well), and so one. Thus, $d\ge 1$, as required (implying
$d=1$).
  \end{proof}

 \noindent {\bf Remark 1}
~If we postulate property~(ii) in Lemma~\ref{lm:rotation} as axiom~(T4') and
add it to axioms (T1)--(T3), then we can eliminate axiom~(T4); in other words,
(T4') and (T4) are equivalent subject to (T1)--(T3). Indeed, reversing
reasonings in the above proof, one can conclude that $\hat\Sigma=2\pi|\hat V|$
implies $|V|+|T|=|E|+1$. The latter is possible only if $D_T$ is a disc.
(Indeed, if $D_T$ forms a regular surface with $g$ handles and $c$ cross-caps,
from which an open disc is removed, then Euler formula is modified as
$|V|+|T|=|E|+1-2g-c$. Also $|V|$ decreases when some vertices merge.) Note that
each of axioms (T1)--(T3),(T4') is ``local''; this gives rise to a local
characterization for semi-normal bases (as well as for largest weakly separated
set-systems); see Remark~5 in Section~\ref{sec:wir-to-til}.

 \medskip
{\bf 3.} Considering the sequence of rotations of the edge ray $R_\bullet$
around a non-terminal vertex $v$ (like in the proof of
Lemma~\ref{lm:rotation}), one can see that the sets $E_T^w(v)$ and $E_T^b(v)$
are arranged as follows.
  \begin{numitem1}
For an ordinary or mixed vertex $v\in V_T$, let $E_T(v)$ consists of edges
$e_1,\ldots,e_p$ following counterclockwise around $v$, and let $e_1$ enter and
$e_p$ leave $v$. Then:
  \begin{itemize}
\item[(i)] there is $1\le p'<p$ such that $e_1,\ldots,e_{p'}$ enter $v$ and
$e_{p'+1},\ldots,e_p$ leave $v$;
\item[(ii)] if $v$ is the right vertex of $r$ black tiles and the left vertex of
$r'$ black tiles, then $r+r'<\min\{p',p-p'\}$ and the black edges incident
to $v$ are exactly $e_{p-r+1},\ldots,e_p,e_1,\ldots,e_r$ and
$e_{p'-r'+1},\ldots,e_{p'+r'}$;
\item[(iii)] the black tiles in $F_T(v)$ have the following pairs of edges
incident to $v$: $\{e_{p-r+1},e_1\},\ldots,\{e_p,e_r\}$ and
$\{e_{p'-r'+1},e_{p'+1}\},\ldots,\{e_{p'},e_{p'+r'}\}$, while the white tiles
in $F_T(v)$ have the following pairs of edges incident to $v$: (a)
$\{e_{r+1},e_{r+2}\},\ldots,\{e_{p'-r'-1},e_{p'-r'}\}$; (b)
$\{e_{p'+r'+1},e_{p'+r'+2}\},\ldots,\{e_{p-r-1},e_{p-r}\}$; (c)
$\{e_{p-r},e_{1}\},\ldots,\{e_{p},e_{r+1}\}$; (d)
$\{e_{p'-r'},e_{p'+1}\},\ldots,\{e_{p'},e_{p'+r'+1}\}$.
  \end{itemize}
  \label{eq:mixedvert}
  \end{numitem1}
(If $v$ is ordinary, then $r=r'=0$ and each (white) tile in $F_T(v)$ meets
a pair of consecutive edges $e_q,e_{q+1}$ or $e_p,e_1$.) The case with
$p=9$, $p'=5$, $r=2$, $r'=1$ is illustrated in the picture; here the black
edges are drawn in bold and the white (black) tiles at $v$ are indicated
by thin (bold) arcs.

 \begin{center}
  \unitlength=1mm
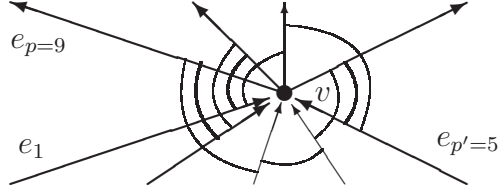
  \begin{picture}(80,25)
 \put(40,12){\circle*{2}}
  \put(36,0){\vector(1,3){3.6}}
  \put(48,0){\vector(-2,3){7}}
  \put(40,12){\vector(0,1){12}}
{\thicklines
  \put(4,0){\vector(3,1){34}}
  \put(22,0){\vector(3,2){16}}
  \put(64,0){\vector(-2,1){22.5}}
  \put(40,12){\vector(2,1){24}}
  \put(40,12){\vector(-1,1){12}}
  \put(40,12){\vector(-3,1){36}}
 }
  \put(44,11){$v$}
  \put(5,4){$e_1$}
  \put(60,5){$e_{p'=5}$}
  \put(4,18){$e_{p=9}$}
 \qbezier(40,17.5)(33,15)(35,10.5)
 \qbezier(33.5,18.5)(28,12)(33,7.5)
 \qbezier(27,16.5)(24,4)(36.5,1.5)
 \qbezier(37,3)(42,1.5)(45,4.5)
 \qbezier(44,6)(49,10)(46,15)
 \qbezier(50,7)(55,20)(40,21)
{\thicklines
 \qbezier(35,17)(31,13)(33.5,10)
 \qbezier(29,16)(27,10)(31.25,6.0)
 \qbezier(48,8)(51,12)(48,16)
  }
   \end{picture}
   \end{center}

Note that~\refeq{mixedvert} implies the following property (which will be used,
in particular, in Subsection~\ref{ssec:part3}):
   \begin{numitem1}
for a tile $\tau\in T$ and a vertex $v\in\{\ell(\tau),r(\tau)\}$, let $e,e'$ be
the edges of $\tau$ entering and leaving $v$, respectively, and suppose that
there is an edge $\tilde e\ne e,e'$ incident to $v$ and contained in
$C(\tau,v)$; then $\tilde e$ is black; furthermore: (a) $e'$ is black if
$\tilde e$ enters $v$; (b) $e$ is black if $\tilde e$ leaves $v$.
   \label{eq:3edges}
   \end{numitem1}

{\bf 4.} We will often use the fact (implied by~\refeq{xi}(ii)) that for any
g-tiling $T$,
   \begin{numitem1}
the graph $G_T=(V_T,E_T)$ is {\em graded} for each color $i\in[n]$, which
means that for any closed path $P$ in $G_T$, the numbers of forward
$i$-edges and backward $i$-edges in $P$ are equal.
   \label{eq:graded}
   \end{numitem1}
Hereinafter, speaking of a path in a directed graph, we mean is a sequence
$P=(\tilde v_0,\tilde e_1,\tilde v_1,\ldots,\tilde e_r,\tilde v_r)$ in which
each $\tilde e_p$ is an edge connecting vertices $\tilde v_{p-1},\tilde v_p$;
an edge $\tilde e_p$ is called {\em forward} if it is directed from $\tilde
v_{p-1}$ to $\tilde v_p$ (denoted as $\tilde e_p=(\tilde v_{p-1},\tilde v_p)$),
and {\em backward} otherwise (when $\tilde e_p=(\tilde v_p,\tilde v_{p-1})$.
The path $P$ is called: {\em closed} if $v_0=v_r$, {\em directed} if all its
edges are forward, and {\em simple} if all vertices $v_0,\ldots,v_r$ are
different. $P^{rev}$ denotes the reversed path $(\tilde v_r,\tilde e_r,\tilde
v_{r-1},\ldots,\tilde e_1,\tilde v_0)$.

\section{From semi-normal bases to generalized tilings}  \label{sec:bas-to-til}

In this section we prove the second assertion in Theorem~\ref{tm:bas-til},
namely, the inclusion
  \begin{equation} \label{eq:BinT}
\Bscr_n\subseteq \Bscr\Tscr_n,
  \end{equation}
where $\Bscr_n$ is the set of semi-normal bases in $2^{[n]}$ and $\Bscr\Tscr_n$
denotes the collection of the spectra of g-tilings on $Z_n$. The proof falls
into three parts, given in Subsections~\ref{ssec:part1}--\ref{ssec:part3}.

  \subsection{Flips in g-tilings} \label{ssec:part1}

Let $T$ be a g-tiling. By an {\em M-configuration} in $T$ we mean a quintuple
of vertices of the form $Xi,Xj,Xk,Xij,Xjk$ with $i<j<k$ (as it resembles the
letter ``M''), which is denoted as $CM(X;i,j,k)$. By a {\em W-configuration} in
$T$ we mean a quintuple of vertices $Xi,Xk,Xij,Xik,Xjk$ with $i<j<k$ (as
resembling ``W''), briefly denoted as $CW(X;i,j,k)$. A configuration is called
{\em feasible} if all five vertices are non-terminal, i.e., they belong to
$B_T$.

We know that any normal basis $B$ (in particular, $B=\Iscr_n$) is expressed as
$B_T$ for some pure tiling $T$, and therefore, $B\in\Bscr\Tscr_n$. Thus, to
conclude with~\refeq{BinT}, it suffices to show the following assertion, which
says that the set of g-tilings is closed under transformations analogous to
flips for semi-normal bases.

  \begin{prop} \label{pr:Tflip}
Let a g-tiling $T$ contain five non-terminal vertices $Xi$, $Xk$, $Xij$, $Xjk$,
$Y$, where $i<j<k$ and $Y\in\{Xik,Xj\}$. Then there exists a g-tiling $T'$ such
that $B_{T'}$ is obtained from $B_T$ by replacing $Y$ by the other member of
$\{Xik,Xj\}$.
  \end{prop}

 \begin{proof}
We may assume that $Y=Xik$, i.e., that we deal with the feasible
W-configuration $CW(X;i,j,k)$ (since an M-configuration in $T$ turns into
a W-configurations in the reversed g-tiling $T^{rev}$). We rely on the
following two facts which will be proved Subsections~\ref{ssec:part2}
and~\ref{ssec:part3}.
 \begin{numitem1}
Any pair of non-terminal vertices $X',X'i'$ in $T$ is connected by edge.
  \label{eq:edge}
  \end{numitem1}
(Therefore, for $T$ as above, $E_T$ contains edges $(Xi,Xij)$, $(Xi,Xik)$,
$(Xk,Xik)$ and $(Xk,Xjk)$. Note that vertices $X',X'i'$ need not be connected
by edge if some of them is terminal; e.g., in the picture before the statement
of Theorem~\ref{tm:bas-til}, the vertices with $X'=\emptyset$ and $i'=2$ are
not connected.)
  \begin{numitem1}
$T$ contains the $jk$-tile $\tau$ with $b(\tau)=Xi$ and the $ij$-tile $\tau'$
with $b(\tau')=Xk$.
  \label{eq:2-tiles}
  \end{numitem1}
Then $\ell(\tau)=Xij$,~ $r(\tau)=\ell(\tau')=Xik$,~ $r(\tau')=Xk$, and
$t(\tau)=t(\tau')=Xijk$. Since the vertices $Xi,Xk$ are non-terminal, both
tiles $\tau,\tau'$ are white. See the picture.

 \begin{center}
  \unitlength=1mm
  \begin{picture}(140,30)
   \put(0,0){\begin{picture}(45,35)
  \put(10,0){\circle*{1.5}}
  \put(10,15){\circle*{1.5}}
  \put(20,15){\circle*{1.5}}
  \put(30,0){\circle*{1.5}}
  \put(30,15){\circle*{1.5}}
  \put(3,0){$Xi$}
  \put(2,15){$Xij$}
  \put(17,10){$Xik$}
  \put(32,0){$Xk$}
  \put(32,15){$Xjk$}
    \end{picture}}
   \put(40,10){\vector(1,0){12}}
   \put(50,0){\begin{picture}(45,35)
  \put(10,0){\circle*{1.5}}
  \put(10,15){\circle*{1.5}}
  \put(20,15){\circle*{1.5}}
  \put(30,0){\circle*{1.5}}
  \put(30,15){\circle*{1.5}}
  \put(10,0){\vector(0,1){14.3}}
  \put(10,0){\vector(2,3){9.3}}
  \put(30,0){\vector(0,1){14.3}}
  \put(30,0){\vector(-2,3){9.3}}
    \end{picture}}
   \put(89,10){\vector(1,0){12}}
   \put(100,0){\begin{picture}(45,35)
  \put(10,0){\circle*{1.5}}
  \put(10,15){\circle*{1.5}}
  \put(20,15){\circle*{1.5}}
  \put(30,0){\circle*{1.5}}
  \put(30,15){\circle*{1.5}}
  \put(20,30){\circle*{1.5}}
  \put(10,0){\vector(0,1){14.3}}
  \put(10,0){\vector(2,3){9.3}}
  \put(30,0){\vector(0,1){14.3}}
  \put(30,0){\vector(-2,3){9.3}}
  \put(10,15){\vector(2,3){9.3}}
  \put(20,15){\vector(0,1){14.3}}
  \put(30,15){\vector(-2,3){9.3}}
  \put(14,14){$\tau$}
  \put(24,14){$\tau'$}
    \end{picture}}
  \end{picture}
   \end{center}

Assuming that~\refeq{edge} and~\refeq{2-tiles} are valid, we argue as follows.
First of all we observe that
  \begin{numitem1}
the vertex $v:=Xik$ is ordinary.
  \label{eq:ordinXik}
  \end{numitem1}
Indeed, since both vertices $Xi,Xik$ are non-terminal, the edge $(Xi,Xik)$
cannot belong to a black tile. So the edge $(Xi,Xik)$, which belongs to the
white tile $\tau$ and enters $v$, is white. Also the edge $(Xik,Xijk)$ of
$\tau$ that leaves $v$ is white (for if it belongs to a black tile $\bar\tau$,
then $\bar\tau$ should have $v':=Xijk$ as its top vertex, but then the cone of
$\bar\tau$ at $v'$ cannot simultaneously contain both edges $(Xij,Xijk)$ and
$(Xjk,Xijk)$, contrary to Corollary~\ref{cor:termv}(iii)). Now one can conclude
from~\refeq{mixedvert} that there is no black tile having its left or right
vertex at $v$. So $v$ is ordinary.

Let $e_0,\ldots,e_q$ be the sequence of edges entering $v$ in the
counterclockwise order; then $e_0=(Xi,Xik)$ and $e_q=(Xk,Xik)$. Since $v$ is
ordinary, each pair $e_{p-1},e_p$ ($p=1,\ldots,q$) belongs to a white tile
$\tau_p$. Two cases are possible.

 \smallskip
{\em Case 1}: The edges $e:=(Xij,Xijk)$ and $e':=(Xjk,Xijk)$ do not belong to
the same black tile. Consider two subcases.

(a) Let $q=1$. We replace in $T$ the tiles $\tau,\tau',\tau_1$ by three new
white tiles: $\tau(X;i,j)$, $\tau(X;j,k)$ and $\tau(Xj;i,k)$ (so the vertex $v$
is replaced by $Xj$). See the picture.

 \begin{center}
  \unitlength=1mm
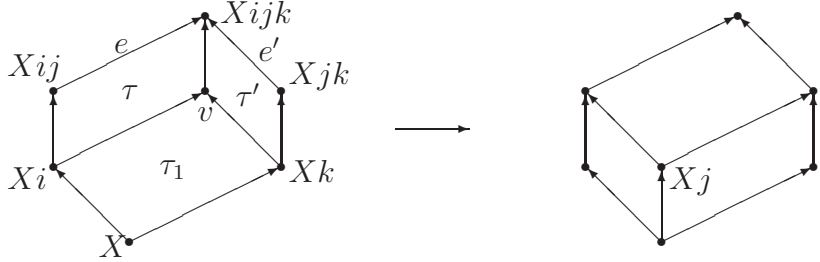
  \begin{picture}(130,30)   
   \put(10,0){\begin{picture}(50,30)
  \put(10,0){\circle*{1}}
  \put(0,10){\circle*{1}}
  \put(30,10){\circle*{1}}
  \put(0,20){\circle*{1}}
  \put(20,20){\circle*{1}}
  \put(30,20){\circle*{1}}
  \put(20,30){\circle*{1}}
  \put(10,0){\vector(-1,1){9.7}}
  \put(30,10){\vector(-1,1){9.7}}
  \put(30,20){\vector(-1,1){9.7}}
  \put(10,0){\vector(2,1){19.7}}
  \put(0,10){\vector(2,1){19.7}}
  \put(0,20){\vector(2,1){19.7}}
  \put(0,10){\vector(0,1){9.5}}
  \put(20,20){\vector(0,1){9.5}}
  \put(30,10){\vector(0,1){9.5}}
  \put(6,-2){$X$}
  \put(-6,7){$Xi$}
  \put(14,9){$\tau_1$}
  \put(31,8){$Xk$}
  \put(-6,22){$Xij$}
  \put(9,19){$\tau$}
  \put(19,16){$v$}
  \put(24,18){$\tau'$}
  \put(31,21){$Xjk$}
  \put(22,29){$Xijk$}
  \put(8,25){$e$}
  \put(27,24){$e'$}
  \put(45,15){\vector(1,0){9.7}}
    \end{picture}}
   \put(80,0){\begin{picture}(50,30)
  \put(10,0){\circle*{1}}
  \put(0,10){\circle*{1}}
  \put(30,10){\circle*{1}}
  \put(0,20){\circle*{1}}
  \put(10,10){\circle*{1}}
  \put(30,20){\circle*{1}}
  \put(20,30){\circle*{1}}
  \put(10,0){\vector(-1,1){9.7}}
  \put(10,10){\vector(-1,1){9.7}}
  \put(30,20){\vector(-1,1){9.7}}
  \put(10,0){\vector(2,1){19.7}}
  \put(10,10){\vector(2,1){19.7}}
  \put(0,20){\vector(2,1){19.7}}
  \put(0,10){\vector(0,1){9.5}}
  \put(10,0){\vector(0,1){9.5}}
  \put(30,10){\vector(0,1){9.5}}
  \put(11,7){$Xj$}
    \end{picture}}
  \end{picture}
   \end{center}

(b) Let $q>1$. We remove the tiles $\tau,\tau'$ and add four new tiles: the
white tiles $\tau(X;i,j)$, $\tau(X;j,k)$, $\tau(Xj;i,k)$ (as before) and the
black tile $\tau(X;i,k)$ (so $v$ becomes terminal). See the picture for $q=3$;
here the added black tile is drawn in bold.

 \begin{center}
  \unitlength=1mm
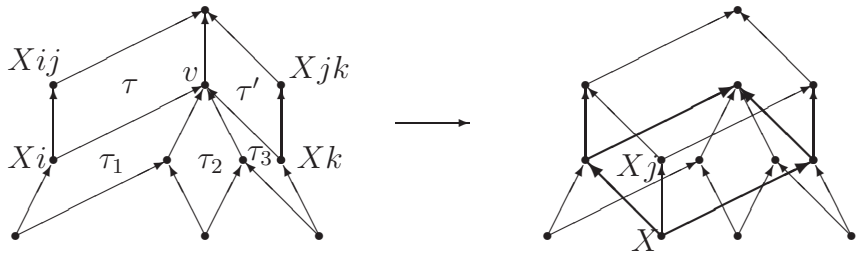
  \begin{picture}(130,30)   
   \put(10,0){\begin{picture}(50,30)
  \put(0,10){\circle*{1}}
  \put(30,10){\circle*{1}}
  \put(0,20){\circle*{1}}
  \put(20,20){\circle*{1}}
  \put(30,20){\circle*{1}}
  \put(20,30){\circle*{1}}
  \put(-5,0){\circle*{1}}
  \put(20,0){\circle*{1}}
  \put(35,0){\circle*{1}}
  \put(15,10){\circle*{1}}
  \put(25,10){\circle*{1}}
  \put(30,10){\circle*{1}}
  \put(30,10){\vector(-1,1){9.7}}
  \put(30,20){\vector(-1,1){9.7}}
  \put(35,0){\vector(-1,1){9.7}}
  \put(0,10){\vector(2,1){19.7}}
  \put(-5,0){\vector(2,1){19.7}}
  \put(0,20){\vector(2,1){19.7}}
  \put(0,10){\vector(0,1){9.5}}
  \put(20,20){\vector(0,1){9.5}}
  \put(30,10){\vector(0,1){9.5}}
  \put(-5,0){\vector(1,2){4.7}}
  \put(15,10){\vector(1,2){4.7}}
  \put(20,0){\vector(1,2){4.7}}
  \put(20,0){\vector(-1,2){4.7}}
  \put(25,10){\vector(-1,2){4.7}}
  \put(35,0){\vector(-1,2){4.7}}
  \put(6,9){$\tau_1$}
  \put(19,9){$\tau_2$}
  \put(25.5,10){$\tau_3$}
  \put(9,19){$\tau$}
  \put(24,18){$\tau'$}
  \put(-6,9){$Xi$}
  \put(32,9){$Xk$}
  \put(-6,22){$Xij$}
  \put(17,20.5){$v$}
  \put(31,21){$Xjk$}
  \put(45,15){\vector(1,0){9.7}}
    \end{picture}}
   \put(80,0){\begin{picture}(50,30)
  \put(0,10){\circle*{1}}
  \put(10,10){\circle*{1}}
  \put(10,0){\circle*{1}}
  \put(30,10){\circle*{1}}
  \put(0,20){\circle*{1}}
  \put(20,20){\circle*{1}}
  \put(30,20){\circle*{1}}
  \put(20,30){\circle*{1}}
  \put(-5,0){\circle*{1}}
  \put(20,0){\circle*{1}}
  \put(35,0){\circle*{1}}
  \put(15,10){\circle*{1}}
  \put(25,10){\circle*{1}}
  \put(30,10){\circle*{1}}
  \put(30,20){\vector(-1,1){9.7}}
  \put(35,0){\vector(-1,1){9.7}}
  \put(-5,0){\vector(2,1){19.7}}
  \put(0,20){\vector(2,1){19.7}}
  \put(0,10){\vector(0,1){9.5}}
  \put(30,10){\vector(0,1){9.5}}
  \put(-5,0){\vector(1,2){4.7}}
  \put(15,10){\vector(1,2){4.7}}
  \put(20,0){\vector(1,2){4.7}}
  \put(20,0){\vector(-1,2){4.7}}
  \put(25,10){\vector(-1,2){4.7}}
  \put(35,0){\vector(-1,2){4.7}}
  \put(10,0){\vector(0,1){9.5}}
  \put(10,10){\vector(-1,1){9.7}}
  \put(10,10){\vector(2,1){19.7}}

{\thicklines
  \put(0,10){\vector(2,1){19.7}}
  \put(10,0){\vector(2,1){19.7}}
  \put(30,10){\vector(-1,1){9.7}}
  \put(10,0){\vector(-1,1){9.7}}
  }
  \put(6,-2){$X$}
  \put(4,8){$Xj$}
    \end{picture}}
  \end{picture}
   \end{center}

{\em Case 2}: Both edges $e$ and $e'$ belong to a black tile $\bar\tau$ (which
is $\tau(Xj;i,k)$). We act as in Case~1 with the only difference that
$\bar\tau$ is removed from $T$ and the white $ik$-tile at $Xj$ (which is a copy
of $\bar\tau$) is not added. Then the vertex $Xijk$ vanishes, $v$ either
vanishes or becomes terminal, and $Xj$ becomes non-terminal. See the picture;
here (a') and (b') concern the subcases $q=1$ and $q>1$, respectively, and the
arc above the vertex $Xj$ indicates the bottom cone of $\bar\tau$ in which some
white edges (not indicated) are located.

 \begin{center}
  \unitlength=1mm
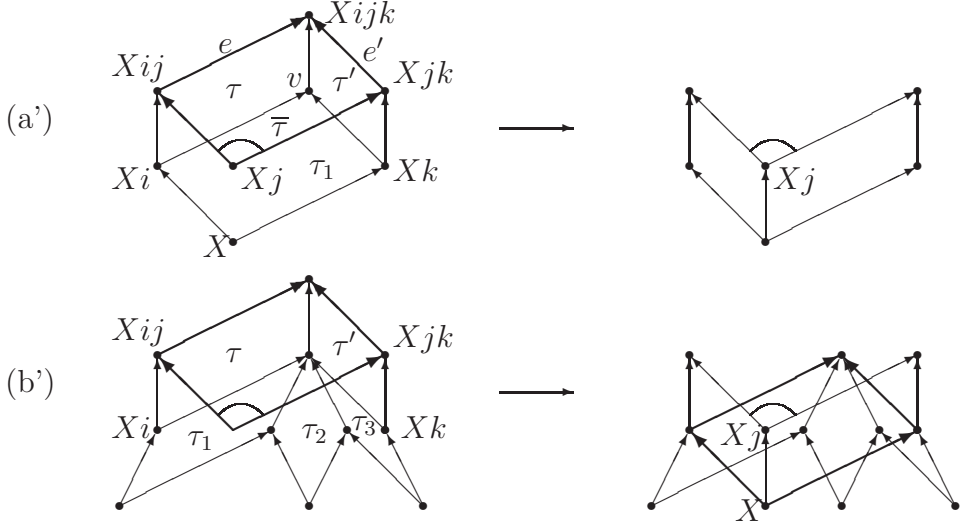
  \begin{picture}(130,65)     
   \put(10,35){\begin{picture}(50,30)
  \put(-20,15){(a')}
  \put(10,0){\circle*{1}}
  \put(0,10){\circle*{1}}
  \put(30,10){\circle*{1}}
  \put(0,20){\circle*{1}}
  \put(20,20){\circle*{1}}
  \put(30,20){\circle*{1}}
  \put(20,30){\circle*{1}}
  \put(10,10){\circle*{1}}
  \put(10,0){\vector(-1,1){9.7}}
  \put(30,10){\vector(-1,1){9.7}}
  \put(30,20){\vector(-1,1){9.7}}
  \put(10,0){\vector(2,1){19.7}}
  \put(0,10){\vector(2,1){19.7}}
  \put(0,20){\vector(2,1){19.7}}
  \put(0,10){\vector(0,1){9.5}}
  \put(20,20){\vector(0,1){9.5}}
  \put(30,10){\vector(0,1){9.5}}
  \put(6,-2){$X$}
  \put(-6,7){$Xi$}
  \put(20,9){$\tau_1$}
  \put(31,8){$Xk$}
  \put(-6,22){$Xij$}
  \put(9,19){$\tau$}
  \put(17,20.5){$v$}
  \put(23,20){$\tau'$}
  \put(15,14){$\bar\tau$}
  \put(31,21){$Xjk$}
  \put(22,29){$Xijk$}
  \put(8,25){$e$}
  \put(27,24){$e'$}
  \put(11,7){$Xj$}
  \qbezier(8,12)(11,15)(14,12)
{\thicklines
  \put(10,10){\vector(-1,1){9.7}}
  \put(10,10){\vector(2,1){19.7}}
  \put(0,20){\vector(2,1){19.7}}
  \put(30,20){\vector(-1,1){9.7}}
 }
  \put(45,15){\vector(1,0){9.7}}
    \end{picture}}
   \put(80,35){\begin{picture}(50,30)
  \put(10,0){\circle*{1}}
  \put(0,10){\circle*{1}}
  \put(30,10){\circle*{1}}
  \put(0,20){\circle*{1}}
  \put(10,10){\circle*{1}}
  \put(30,20){\circle*{1}}
%
  \put(10,0){\vector(-1,1){9.7}}
  \put(10,10){\vector(-1,1){9.7}}
  \put(10,0){\vector(2,1){19.7}}
  \put(10,10){\vector(2,1){19.7}}
  \put(0,10){\vector(0,1){9.5}}
  \put(10,0){\vector(0,1){9.5}}
  \put(30,10){\vector(0,1){9.5}}
  \put(11,7){$Xj$}
  \qbezier(8,12)(11,15)(14,12)
    \end{picture}}
%
   \put(10,0){\begin{picture}(50,30)
  \put(-20,15){(b')}
  \put(0,10){\circle*{1}}
  \put(30,10){\circle*{1}}
  \put(0,20){\circle*{1}}
  \put(20,20){\circle*{1}}
  \put(30,20){\circle*{1}}
  \put(20,30){\circle*{1}}
  \put(-5,0){\circle*{1}}
  \put(20,0){\circle*{1}}
  \put(35,0){\circle*{1}}
  \put(15,10){\circle*{1}}
  \put(25,10){\circle*{1}}
  \put(30,10){\circle*{1}}
  \put(30,10){\vector(-1,1){9.7}}
  \put(30,20){\vector(-1,1){9.7}}
  \put(35,0){\vector(-1,1){9.7}}
  \put(0,10){\vector(2,1){19.7}}
  \put(-5,0){\vector(2,1){19.7}}
  \put(0,20){\vector(2,1){19.7}}
  \put(0,10){\vector(0,1){9.5}}
  \put(20,20){\vector(0,1){9.5}}
  \put(30,10){\vector(0,1){9.5}}
  \put(-5,0){\vector(1,2){4.7}}
  \put(15,10){\vector(1,2){4.7}}
  \put(20,0){\vector(1,2){4.7}}
  \put(20,0){\vector(-1,2){4.7}}
  \put(25,10){\vector(-1,2){4.7}}
  \put(35,0){\vector(-1,2){4.7}}
{\thicklines
  \put(10,10){\vector(-1,1){9.7}}
  \put(10,10){\vector(2,1){19.7}}
  \put(0,20){\vector(2,1){19.7}}
  \put(30,20){\vector(-1,1){9.7}}
 }
  \put(4,8){$\tau_1$}
  \put(19,9){$\tau_2$}
  \put(25.5,10){$\tau_3$}
  \put(9,19){$\tau$}
  \put(23,20){$\tau'$}
  \put(-6,9){$Xi$}
  \put(32,9){$Xk$}
  \put(-6,22){$Xij$}
  \put(31,21){$Xjk$}
  \put(45,15){\vector(1,0){9.7}}
  \qbezier(8,12)(11,15)(14,12)
    \end{picture}}
   \put(80,0){\begin{picture}(50,30)
  \put(0,10){\circle*{1}}
  \put(10,10){\circle*{1}}
  \put(10,0){\circle*{1}}
  \put(30,10){\circle*{1}}
  \put(0,20){\circle*{1}}
  \put(20,20){\circle*{1}}
  \put(30,20){\circle*{1}}
  \put(-5,0){\circle*{1}}
  \put(20,0){\circle*{1}}
  \put(35,0){\circle*{1}}
  \put(15,10){\circle*{1}}
  \put(25,10){\circle*{1}}
  \put(30,10){\circle*{1}}
%
  \put(35,0){\vector(-1,1){9.7}}
  \put(-5,0){\vector(2,1){19.7}}
  \put(0,10){\vector(0,1){9.5}}
  \put(30,10){\vector(0,1){9.5}}
  \put(-5,0){\vector(1,2){4.7}}
  \put(15,10){\vector(1,2){4.7}}
  \put(20,0){\vector(1,2){4.7}}
  \put(20,0){\vector(-1,2){4.7}}
  \put(25,10){\vector(-1,2){4.7}}
  \put(35,0){\vector(-1,2){4.7}}
  \put(10,0){\vector(0,1){9.5}}
  \put(10,10){\vector(-1,1){9.7}}
  \put(10,10){\vector(2,1){19.7}}

{\thicklines
  \put(0,10){\vector(2,1){19.7}}
  \put(10,0){\vector(2,1){19.7}}
  \put(30,10){\vector(-1,1){9.7}}
  \put(10,0){\vector(-1,1){9.7}}
  }
  \put(6,-2){$X$}
  \put(4,8){$Xj$}
  \qbezier(8,12)(11,15)(14,12)
    \end{picture}}
  \end{picture}
   \end{center}

 \smallskip
Let $T'$ be the resulting collection of tiles. It is routine to check that
in all cases the transformation of $T$ into $T'$ maintains the conditions
on tiles and edges involved in axioms (T1)--(T3) at the vertices
$Xi,Xk,Xij,Xjk$, as well as at the vertices $Xik$ and $Xijk$ when the last
ones do not vanish. Also the conditions continue to hold at the vertex $X$
in Cases~1(a) and~2(a') (with $q=1$), and at the vertex $Xj$ in Case 2
(when the terminal vertex $Xj$ becomes non-terminal). A less trivial task
is to verify for $T'$ the correctness at $Xj$ in Case~1 and at $X$ in
Cases~1(b) and~2(b'). We assert that
  \begin{numitem1}
(i) $V_T$ does not contain $Xj$~ in Case~1; and (ii) $V_T$ does not contain $X$
in Cases~1(b) and~2(b').
   \label{eq:nonexist}
   \end{numitem1}

Then these vertices (in the corresponding cases) are indeed new in the arising
$T'$, and now the required properties for them become evident by the
construction. Note that this implies (T4) as well. We will
prove~\refeq{nonexist} in Subsection~\ref{ssec:part3}.

Thus, assuming validity of~\refeq{edge},~\refeq{2-tiles},~\refeq{nonexist}, we
can conclude that $T'$ is a g-tiling and that $B_{T'}=(B_T-\{Xik\})\cup\{Xj\}$,
as required.
   \end{proof}

 \noindent
{\bf Remark 2} ~Adopting terminology for set-systems, we say that for the
g-tilings $T,T'$ as in the proof of Proposition~\ref{pr:Tflip}, $T'$ is
obtained from $T$ by the {\em lowering flip} w.r.t. the feasible
W-configuration $CW(X;i,j,k)$. One can see that $Xi,Xj,Xk,Xij,Xjk$ are
non-terminal vertices in $G_{T'}$; so they form a feasible M-configuration for
$T'$. Moreover, it is not difficult to check that the corresponding lowering
flip applied to the reverse of $T'$ results in the g-tiling $T^{rev}$.
Equivalently: the {\em raising flip} for $T'$ w.r.t. the configuration
$CM(X;i,j,k)$ returns the initial $T$. An important consequence of this fact
will be demonstrated in Section~\ref{sec:wir-to-til} (see
Theorem~\ref{tm:unique}).

  \subsection{Strips in a g-tiling} \label{ssec:part2}

In this subsection we show property~\refeq{2-tiles}. For this purpose, we
introduce the following notion (which will be extensively used subsequently as
well).
\medskip

 \noindent{\bf Definition.}
For $i\in[n]$, an $i$-{\em strip} (or a {\em dual $i$-path}) in a g-tiling
$T$ is a maximal sequence $Q=(e_0,\tau_1,e_1,\ldots,\tau_r,e_r)$ of edges
and tiles in it such that: (a) $\tau_1,\ldots,\tau_r$ are different tiles,
each being an $iq$- or $qi$-tile for some $q$, and (b) for $p=1,\ldots,r$,
~$e_{p-1}$ and $e_p$ are the opposite $i$-edges of $\tau_p$. \medskip

(Recall that speaking of an $i'j'$-tile, we assume that $i'<j'$.) Clearly $Q$
is determined uniquely (up to reversing it and up to shifting cyclically when
$e_0=e_r$) by any of its edges or tiles. Also, unless $e_0=e_r$, one of
$e_0,e_r$ lies on the left boundary, and the other on the right boundary of
$Z$; we default assume that $Q$ is directed so that $e_0\in \ell bd(Z)$. In
this case, going along $Q$, step by step, and using~(T2), one can see that
  \begin{numitem1}
for consecutive elements $e,\tau,e'$ in an $i$-strip $Q$: (a) if $\tau$ is
either a white $iq$-tile or a black $qi$-tile (for some $q$), then $e$
leaves $b(\tau)$ and $e'$ enters $t(\tau)$; and (b) if $\tau$ is either a
white $qi$-tile or a black $iq$-tile, then $e$ enters $t(\tau)$ and $e'$
leaves $b(\tau)$ (see the picture where the $i$-edges $e,e'$ are drawn
vertically).
   \label{eq:strip_tiles}
   \end{numitem1}

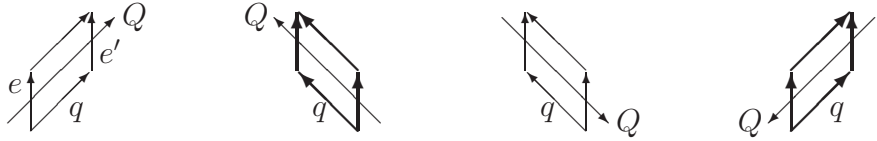
 \begin{center}
  \unitlength=1mm
  \begin{picture}(120,16)
   \put(0,0){\begin{picture}(20,16)
  \put(6,0){\vector(0,1){7.7}}
  \put(14,8){\vector(0,1){7.7}}
  \put(6,0){\vector(1,1){7.7}}
  \put(6,8){\vector(1,1){7.7}}
  \put(3,1){\vector(1,1){14}}
  \put(18,14){$Q$}
  \put(3,5){$e$}
  \put(15,9){$e'$}
  \put(11,2){$q$}
    \end{picture}}
   \put(35,0){\begin{picture}(20,16)
 {\thicklines
  \put(6,8){\vector(0,1){7.7}}
  \put(14,0){\vector(0,1){7.7}}
  \put(14,0){\vector(-1,1){7.7}}
  \put(14,8){\vector(-1,1){7.7}}
  }
  \put(17,1){\vector(-1,1){14}}
  \put(-1,14){$Q$}
  \put(8,2){$q$}
    \end{picture}}
   \put(65,0){\begin{picture}(20,16)
  \put(6,8){\vector(0,1){7.7}}
  \put(14,0){\vector(0,1){7.7}}
  \put(14,0){\vector(-1,1){7.7}}
  \put(14,8){\vector(-1,1){7.7}}
  \put(3,15){\vector(1,-1){14}}
  \put(18,0){$Q$}
  \put(8,2){$q$}
    \end{picture}}
   \put(100,0){\begin{picture}(20,16)
 {\thicklines
  \put(6,0){\vector(0,1){7.7}}
  \put(14,8){\vector(0,1){7.7}}
  \put(6,0){\vector(1,1){7.7}}
  \put(6,8){\vector(1,1){7.7}}
  }
  \put(17,15){\vector(-1,-1){14}}
  \put(-1,0){$Q$}
  \put(11,2){$q$}
    \end{picture}}
  \end{picture}
   \end{center}

Let $v_p$ (resp. $v'_p$) be the beginning (resp. end) vertex of an edge $e_p$
in $Q$. Define the {\em right boundary} of $Q$ to be the path
$R_Q=(v_0,a_1,v_1,\ldots,a_r,v_r)$, where $a_p$ is the edge of $\tau_p$
connecting $v_{p-1},v_p$. The {\em left boundary} $L_Q$ of $Q$ is defined in a
similar way (regarding the vertices $v'_p$). From~\refeq{strip_tiles} it
follows that
  \begin{numitem1}
for an $i$-strip $Q$, the forward edges of $R_Q$ are exactly those edges
in it that belong to either a white $iq$-tile or a black $qi$-tile in $Q$,
and similarly for the forward edges of $L_Q$.
  \label{eq:strip_boundary}
  \end{numitem1}

For $I\subseteq[n]$, we call a {\em maximal alternating $I$-subpath} in
$R_Q$ a maximal subsequence $P$ of {\em consecutive} elements in $R_Q$
such that each $a_p\in P$ is a $q$-edge with $q\in I$, and in each pair
$a_p,a_{p+1}$, one edge is forward and the other is backward in $R_Q$
(i.e., exactly one of the tiles $\tau_p,\tau_{p+1}$ is black). A maximal
alternating $I$-subpath in $L_Q$ is defined in a similar way. The
following fact is of importance.
  \begin{lemma} \label{lm:no-cycl}
A strip $Q$ cannot be cyclic, i.e., its first and last edges are different.
  \end{lemma}
  \begin{proof}
For a contradiction, suppose that some $i$-strip
$Q=(e_0,\tau_1,e_1,\ldots,\tau_r,e_r)$ is cyclic ($e_0=e_r$). One may
assume that~\refeq{strip_tiles} holds for $Q$ (otherwise reverse $Q$).
Take the right boundary $R_Q=(a_1,\ldots,a_r=a_0)$ of $Q$. For $q\in[n]$,
let $\alpha_q$ ($\beta_q$) denote the number of forward (resp. backward)
$q$-edges in $R_Q$. Since $G_T$ is graded, $\alpha_q=\beta_q$
(cf.~\refeq{graded}).

Assume that $R_T$ contains a $q$-edge with $q>i$. Put $I^>:=[i+1..n]$ and
consider a maximal alternating $I^>$-subpath in $R_Q$ (regarding $Q$ up to
shifting cyclically and taking indices modulo $r$). Using~\refeq{termv}, we
observe that if $a_p$ is an edge in $P$ such that $\tau_p$ is black, then the
edges $a_{p-1},a_{p+1}$ are contained in $P$ as well; also both tiles
$\tau_{p-1},\tau_{p+1}$ are white. This together with~\refeq{strip_boundary}
implies that the difference $\Delta_P$ between the number of forward edges and
the number of backward edges in $P$ is equal to 0 or 1, and that $\Delta_P=0$
is possible only if $P$ coincides with the whole $R_Q$ (having equal numbers of
forward and backward edges). On the other hand, the sum of numbers $\Delta_P$
over the maximal alternating $I^>$-subpaths must be equal to
$\sum_{q>i}(\alpha_q-\beta_q)=0$.

So $R_Q$ is an alternating $I^>$-cycle. To see that this is impossible,
notice that if $a_{p-1},a_p,a_{p+1}$ are $q'$-, $q$-, and $q''$-edges,
respectively, and if the tile $\tau_p$ is black, then~\refeq{termv}
implies that $q',q''<q$. Therefore, taking the maximum $q$ such that $R_Q$
contains a $q$-edge, we obtain $\alpha_q=0$ and $\beta_q>0$; a
contradiction. Thus, $R_Q$ has no $q$-edges with $q>i$ at all.

Similarly, considering maximal alternating $[i-1]$-subpaths in $R_Q$ and
using~\refeq{termv} and~\refeq{strip_boundary}, we conclude that $R_Q$ has no
$q$-edge with $q<i$. Thus, a cyclic $i$-strip is impossible.
  \end{proof}

  \begin{corollary} \label{cor:strips}
For a g-tiling $T$ and each $i\in[n]$, there is a unique $i$-strip $Q_i$.
It contains all $i$-edges of $T$, begins at the edge $p_{i-1}p_i$ and ends
at the edge $p'_{i}p'_{i-1}$ of $bd(Z)$.
  \end{corollary}

Using strip techniques, we are now able to prove property~\refeq{2-tiles} in
the assumption that~\refeq{edge} is valid (the latter will be shown in the next
subsection).
\medskip

 \noindent {\bf Proof of~\refeq{2-tiles}}~
Let $X,i,j,k$ as in the hypotheses of Proposition~\ref{pr:Tflip} (with
$Y=Xik$). We consider the part $Q$ of the $j$-strip between the $j$-edges
$e:=(Xi,Xij)$ and $e':=(Xk,Xjk)$ (these edges exist by~\refeq{edge} and $Q$
exists by Corollary~\ref{cor:strips}). Suppose that $Q$ begins at $e$ and ends
at $e'$ and consider the right boundary $R_Q=(a_1,\ldots,a_r)$ of $Q$. This is
a (not necessarily directed) path from $Xi$ to $Xk$. Comparing $R_Q$ with the
path $\tilde P$ from $Xi$ to $Xk$ formed by the forward $k$-edge $(Xi,Xik)$ and
the backward $i$-edge $(Xik,Xk)$, we have (since $G_T$ is graded):
   \begin{equation} \label{eq:ab1}
   \alpha_q-\beta_q= \left\{
  \begin{array}{rl}
 -1 & \quad \mbox{for $q=i$,} \\
 1 & \quad \mbox{for $q=k$,} \\
 0 & \quad \mbox{otherwise,}
   \end{array}
  \right.
   \end{equation}
where $\alpha_q$ ($\beta_q$) is the number of forward (resp. backward)
$q$-edges in $R_Q$. We show that $\alpha_i=0$, $\beta_i=1$, $\alpha_k=1$,
$\beta_k=0$ and $\alpha_q=\beta_q=0$ for $q\ne i,k$, by arguing as in the proof
of Lemma~\ref{lm:no-cycl}.

Let $P_1,\ldots,P_d$ be the sequence of maximal alternating $J^>$-subpaths in
$R_Q$, where $J^>:=[j+1..n]$. Each subpath $P_h$ begins and ends with forward
edges (taking into account~\refeq{termv},\refeq{strip_boundary} and the fact
that the edges $e,e'$ are white). Therefore, $\Delta_{P_h}=1$. Then
$\Delta_1+\ldots+\Delta_{P_d}=\sum_{q>j}(\alpha_q-\beta_q) =\alpha_k-\beta_k=1$
(cf.~\refeq{ab1}) implies $d=1$. Moreover, $|P_1|=1$. For if $|P_1|>1$, then
$P_1$ contains a backward edge (belonging to a black tile in $Q$), and taking
the maximum $q$ such that $P_1$ contains a $q$-edge, we obtain $\alpha_q=0$ and
$\beta_q>0$, which is impossible. Hence $P_1$ consists of a unique forward
edge, and now~\refeq{ab1} implies that it is a $k$-edge.

By similar reasonings, there is only one maximal alternating $[j-1]$-subpath
$P'$ in $R_Q$, and $P'$ consists of a unique backward $i$-edge.

Thus, $R_Q=(a_1,a_2)$, and one of $a_1,a_2$ is a forward $k$-edge, while the
other is a backward $i$-edge in $R_Q$. If $R_Q=\tilde P$ (i.e., $a_1$ is a
$k$-edge), then the tiles in $Q$ are as required in~\refeq{2-tiles}. The case
when $a_1$ is an $i$-edge is impossible. Indeed, this would imply that the
first tile $\tau$ in $Q$ is generated by the edges $a_1=(X,Xi)$ and
$e=(Xi,Xij)$; but then the cone of $\tau$ at $Xi$ contains the white edge
$(Xi,Xik)$, contrary to~\refeq{3edges}.

Now suppose that $Q$ goes from $e'$ to $e$. Then $R_Q$ begins at $Xk$ and ends
at $Xi$. Define the numbers $\alpha_q,\beta_q$ as before. Then
$\sum_{q>j}(\alpha_q-\beta_q)$ (equal to the numbers of maximal alternating
$J^>$-subpaths in $R_Q$) is nonnegative. But a similar value for the path
reverse to $\tilde P$ (going from $Xk$ to $Xi$ as well) is equal to $-1$, due
to the $k$-edge $(Xi,Xik)$ which is backward in this path; a contradiction.
\hfill\qed

  \subsection{Strip contractions} \label{ssec:part3}

The remaining properties~\refeq{edge} and~\refeq{nonexist} are proved by
induction on $n$, relying on a natural contracting operation on g-tilings (also
important for purposes of Sections~\ref{sec:contr-exp},\ref{sec:wsf}).

Let $T$ be a g-tiling on $Z_n$ and $i\in[n]$. We partition $T$ into three
subsets $T^0_i, T^-_i,T^+_i$ conisting, respectively, of all $i\ast$- and
$\ast i$-tiles, of the tiles $\tau(X;i',j')$ with $i',j'\ne i$ and
$i\not\in X$, and of the tiles $\tau(X;i',j')$ with $i',j'\ne i$ and $i\in
X$. Then $T_i^0$ is the set of tiles occurring in the $i$-strip $Q_i$, and
the tiles in $T_i^-$ are vertex disjoint from those in $T_i^+$. \medskip

 \noindent{\bf Definition.}
The $i$-{\em contraction} of $T$ is the collection $T/i$ obtained by removing
the members of $T_i^0$, keeping the members of $T_i^-$, and replacing each
$\tau(X;i',j')\in T_i^+$ by $\tau(X-\{i\};i',j')$. The image of $\tau\in T$ in
$T/i$ is denoted by $\tau/i$ (regarding it as the ``void tile'' if $\tau\in
T_i^0$). The black/white coloring of tiles in $T/i$ is inherited from $T$.
\medskip

The tiles of $T/i$ live within the zonogon generated by the vectors
$\xi_q$, $q\in[n]-\{i\}$ (and cover this zonogon). The regions $D_{T_i^-}$
and $D_{T_i^+}$ of the disc $D_T$ are simply connected, as they arise when
the interior of (the image of) the strip $Q_i$ is removed from $D_T$. The
shape $D_{T/i}$ is obtained as the union of $D_{T_i^-}$ and
$D_{T_i^+}-\eps_i$, where $\eps_i$ is the $i$-th unit base vector in
$\Rset^n$. In other words, $D_i^+$ is shifted by $-\eps_i$ and the (image
of) the left boundary $L_{Q_i}$ of $Q_i$ in it merges with (the image of)
$R_{Q_i}$ in $D_{T_i^-}$. In general, $D_{T_i^-}$ and $D_{T_i^+}-\eps_i$
may intersect at some other points, and therefore, $D_{T/i}$ need not be a
disc (this happens when $G_T$ contains two vertices $X,Xi$ not connected
by edge, or equivalently, such that $X\not\in R_{Q_i}$ and $Xi\not\in
L_{Q_i}$).

For our purposes, it suffices to deal with the case $i=n$. We take advantages
from the important property that $T/n$ is a feasible g-tiling, i.e., obeys
(T1)--(T4). Instead of a direct proof of this property (in which verification
of axiom~(T4) is rather tiresome), we prefer to appeal to explanations in
Section~\ref{sec:wir-to-til} where a similar property is obtained on the
language of wirings; see Corollary~\ref{cor:contract} in Remark~4. (More
precisely, by results in Sections~\ref{sec:til-to-wir},\ref{sec:wir-to-til},
there is a bijection $\beta$ of the g-tilings to the proper wirings.
Furthermore, one shows that removing $w_n$ from a proper $n$-wiring
$W=(w_1,\ldots,w_n)$ results in a proper $(n-1)$-wiring $W'$. It turns out that
the g-tiling $\beta^{-1}(W')$ is just $(\beta^{-1}(W))/n$, yielding the desired
property.)

We will use several facts, exposed in (i),(ii) below.

(i) In light of explanations above, validity of~(T4) for $T/n$ implies that
  \begin{numitem1}
any two vertices of the form $X',X'n$ in $G_T$ are connected by edge
(which is an $n$-edge, and therefore, $X'\in R_{Q_n}$ and $X'n\in
L_{Q_n}$).
   \label{eq:XXn}
  \end{numitem1}

(ii) For a black $i'j'$-tile $\tau\in T$ with $i',j'\ne n$, none of the
vertices $t(\tau),b(\tau)$ can occur in (the boundary of) $Q_n$. Indeed, all
edges incident to such a vertex are $q$-edges with $q\le j'<n$, whereas each
vertex occurring in $Q_n$ is incident to an $n$-edge. Also for a vertex $X'$
not occurring in $Q_n$, the local tile structure of $T/n$ at $X'-\{n\}$
(including the white/black coloring of tiles) is inherited by that of $T$ at
$X'$. It follows that
   \begin{numitem1}
if $X'$ is a non-terminal vertex for $T$, then $X'-\{n\}$ is such for $T/n$.
   \label{eq:non-term}
   \end{numitem1}

Now we are ready to prove~\refeq{edge} and~\refeq{nonexist}. \medskip

 \noindent {\bf Proof of~\refeq{edge}}~
Let $X',X'i'$ be non-terminal vertices for $T$. If $i'=n$ then these vertices
are connected by edge in $G_T$, by~\refeq{XXn}. Now let $i'\ne n$. Then
$G_{T/n}$ contains the vertices $\tilde X,\tilde X i'$, where $\tilde
X:=X'-\{n\}$, and these vertices are non-terminal, by~\refeq{non-term}. We
assume by induction that these vertices are connected by some edge $e$ in
$G_{T/n}$. Let $\tau'$ be a tile in $T/n$ containing $e$. Then the tile
$\tau\in T_n^-\cup T_n^+$ such that $\tau'=\tau/n$ has an edge connecting $X'$
and $X'i'$, as required. \hfill \qed

\medskip
 \noindent {\bf Proof of~\refeq{nonexist}}~
We use notation as in the proof of Proposition~\ref{pr:Tflip} and consider
three possible cases. \smallskip

(A) Let $k<n$ and $n\not\in X$. Then all tiles in $T$ containing the
vertex $v=Xik$ are tiles in $T/n$, and $Xi,Xk,Xik,Xij,Xjk$ are vertices
for $T/n$ forming a feasible W-configuration in it (as they are
non-terminal, by~\refeq{non-term}). By induction $G_{T/n}$ contains as a
vertex neither $Xj$ in Case~1, nor $X$ in Cases~1(b) and 2(b'). Then the
same is true for $G_T$, as required. \smallskip

(B) Let $k<n$ and $n\in X$. The argument is similar to that in (A) (taking
into account that all vertices $X'$ for $T$ that we deal with contain the
element $n$, and the corresponding vertices for $T/n$ are obtained by
removing this element). \smallskip

(C) Let $k=n$. First we consider Case~1 and show~(i) in~\refeq{nonexist}.
Suppose $G_T$ contains the vertex $Xj$. Then $G_T$ contains the $n$-edge
$\tilde e=(Xj,Xjn)$, by~\refeq{XXn}. This edge lies in the cone of $\tau'$ at
$Xjn$ (where $\tau'$ is the white $ij$-tile with $b(\tau')=Xn$, $r(\tau')=Xjn$
and $t(\tau')=Xijn$). By~\refeq{3edges}, the presence of the edge $\tilde e$
(entering $Xjn$) in this cone implies that the edge $e'=(Xjn,Xijn)$ of $\tau'$
is black. Let $\tilde \tau$ be the black tile containing $e'$; then $\tilde
\tau$ has the top vertex at $Xijn$ (since $Xjn$ is non-terminal). The $n$-edge
$e=(Xij,Xijn)$ entering $Xijn$ must be the left-to-top edge of $\tilde\tau$. So
both $e,e'$ are edges of the same black tile $\tilde\tau$, which is not the
case.

Now we consider Cases~1(b) and~2(b') and show~(ii) by arguing in a similar
way. Suppose $G_T$ contains the vertex $X$. Then $G_T$ contains the
$n$-edge $\tilde e=(X,Xn)$, by~\refeq{XXn}. This edge lies in the cone of
$\tau_q$ at $Xn$ (where, according to notation in the proof of
Proposition~\ref{pr:Tflip}, $\tau_q$ is the white tile in $T$ with
$r(\tau_q)=Xn$ and $t(\tau')=Xin$). By~\refeq{3edges}, the edge
$e'':=(Xn,Xin)$ of $\tau_q$ is black (since $\tilde e$ enters $Xn$). But
each of the end vertices $Xn,Xin$ of $e''$ has both entering and leaving
edges, and therefore, it cannot be terminal; a contradiction.
 \hfill \qed \medskip

This completes the proof of inclusion~\refeq{BinT}.

\section{From generalized tilings to semi-normal bases}  \label{sec:til-to-bas}

In this section we complete the proof of Theorem~\ref{tm:bas-til} by
proving the first assertion in it, namely, we show the inclusion
  \begin{equation} \label{eq:TinB}
\Bscr\Tscr_n\subseteq \Bscr_n.
  \end{equation}
This together with the reverse inclusion~\refeq{BinT} will give
$\Bscr\Tscr_n=\Bscr_n$, as required. \smallskip

Let $T$ be a g-tiling. We have to prove that $B_T$ is a semi-normal basis.

If $T$ has no black tile, then $B_T$ is a normal basis, and we are done. So
assume $T^b\ne \emptyset$. Our aim is to show the existence of a feasible
W-configuration $CW(X;i,j,k)$ for $T$ (formed by non-terminal vertices
$Xi,Xk,Xij,Xik,Xjk$, where $i<j<k$). Then we can transform $T$ into a g-tiling
$T'$ as in Proposition~\ref{pr:Tflip}, i.e., with
$B_{T'}=(B_T-\{Xik\})\cup\{Xj\}$. Under such a {\em lowering flip} (concerning
g-tilings), the sum of sizes of the sets involved in $B_{\bullet}$ decreases.
Then the required relation $B_T\in \Bscr_n$ follows by induction on
$\sum(|X'|\colon X'\in B_T)$ (this sort of induction is typical when one deals
with tilings or related objects, cf.~\cite{DKK-08,HS,LZ}).

In what follows by the {\em height} $h(v)$ of a vertex $v\in V_T$ we mean the
size of the corresponding subset of $[n]$. The height $h(\tau)$ of a tile
$\tau\in T$ is defined to be the height of its left vertex; then
$h(\tau)=h(r(\tau))=h(b(\tau))+1=h(t(\tau))-1$. The height of a W-configuration
$CW(X;i,j,k)$ is defined to be $|X|+2$.

In fact, we are able to show the following sharper version of the desired
property.
  \begin{prop} \label{pr:downflip}
Let $h\in[n]$. If a g-tiling $T$ has a black tile of height $h$, then there
exists a feasible W-configuration $CW(X;i,j,k)$ of the same height $h$.
Moreover, such a $CW(X;i,j,k)$ can be chosen so that $Xijk$ is the top vertex
of some black tile (of height $h$).
  \end{prop}
 \begin{proof}
Let $\tau$ be a black tile of height $h$. Denote by $M(\tau)$ the set of
vertices $v$ such that there is a white edge from $v$ to $t(\tau)$. This
set is nonempty (by Corollary~\ref{cor:termv}(ii)) and each vertex in it
is non-terminal. Suppose that some $v\in M(\tau)$ is ordinary, and let
$\lambda$ and $\rho$ be the (white) tiles sharing the edge $(v,t(\tau))$
and such that $v=r(\lambda)=\ell(\rho)$. Then the five vertices
$b(\lambda),b(\rho),\ell(\lambda),v,r(\rho)$ form a W-configuration of
height $h$ (since $h(v)=h(\tau)=h$). Moreover, this configuration is
feasible. Indeed, the vertices $\ell(\lambda),v,r(\rho)$ are non-terminal
(since each has an entering edge and a leaving edge). And the tile $\tilde
\tau$ that shares the edge $(b(\lambda),v)$ with $\lambda$ has $v$ as its
top vertex (taking into account that $\tilde \tau$ is white and overlaps
neither $\lambda$ nor $\rho$ since $v$ is ordinary); then $b(\lambda)$ is
the left vertex of $\tilde\tau$, and therefore $b(\lambda)$ is
non-terminal. The vertex $b(\rho)$ is non-terminal for a similar reason.
For an illustration, see the left fragment on the picture.

 \begin{center}
  \unitlength=1mm
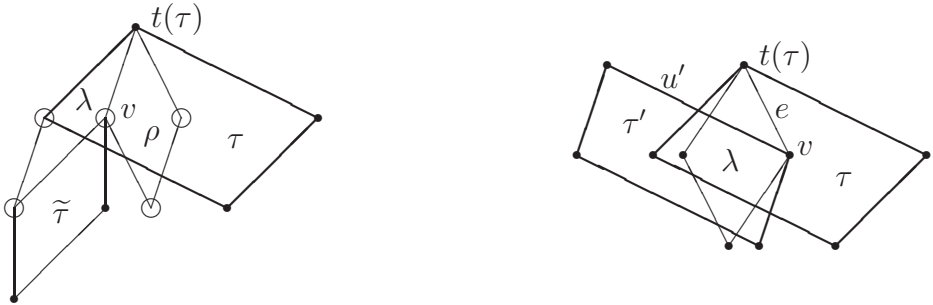
  \begin{picture}(130,38)
   \put(-10,0){\begin{picture}(50,38)
  \put(10,0){\circle*{1}}
  \put(22,12){\circle*{1}}
  \put(38,12){\circle*{1}}
  \put(50,24){\circle*{1}}
  \put(26,36){\circle*{1}}
  \put(10,12){\circle{2.5}}
  \put(28,12){\circle{2.5}}
  \put(14,24){\circle{2.5}}
  \put(22,24){\circle{2.5}}
  \put(32,24){\circle{2.5}}
  \put(10,0){\line(1,1){12}}
  \put(10,12){\line(1,1){12}}
  \put(10,0){\line(0,1){12}}
  \put(22,12){\line(0,1){12}}
  \put(10,12){\line(1,3){4}}
  \put(22,24){\line(1,3){4}}
  \put(28,12){\line(1,3){4}}
  \put(28,12){\line(-1,2){6}}
  \put(32,24){\line(-1,2){6}}
{\thicklines
  \put(14,24){\line(1,1){12}}
  \put(38,12){\line(1,1){12}}
  \put(38,12){\line(-2,1){24}}
  \put(50,24){\line(-2,1){24}}
  }
  \put(15,10){$\tilde\tau$}
  \put(18,25){$\lambda$}
  \put(24,24){$v$}
  \put(27,21){$\rho$}
  \put(38,20){$\tau$}
  \put(28,36){$t(\tau)$}
   \end{picture}}
   \put(70,7){\begin{picture}(50,30)
  \put(38,0){\circle*{1}}
  \put(50,12){\circle*{1}}
  \put(26,24){\circle*{1}}
  \put(24,0){\circle*{1}}
  \put(28,0){\circle*{1}}
  \put(14,12){\circle*{1}}
  \put(18,12){\circle*{1}}
  \put(32,12){\circle*{1}}
  \put(4,12){\circle*{1}}
  \put(8,24){\circle*{1}}
  \put(18,12){\line(2,3){8}}
  \put(24,0){\line(-1,2){6}}
  \put(32,12){\line(-1,2){6}}
  \put(24,0){\line(2,3){8}}
{\thicklines
  \put(14,12){\line(1,1){12}}
  \put(38,0){\line(1,1){12}}
  \put(38,0){\line(-2,1){24}}
  \put(50,12){\line(-2,1){24}}
  \put(28,0){\line(1,3){4}}
  \put(4,12){\line(1,3){4}}
  \put(28,0){\line(-2,1){24}}
  \put(32,12){\line(-2,1){24}}
  }
  \put(33,12){$v$}
  \put(10,15){$\tau'$}
  \put(23,10){$\lambda$}
  \put(38,8){$\tau$}
  \put(28,24){$t(\tau)$}
  \put(30,17){$e$}
  \put(15,21){$u'$}
   \end{picture}}
  \end{picture}
   \end{center}

We assert that a black tile $\tau$ of height $h$ whose set $M(\tau)$
contains an ordinary vertex does exist (yielding the result).

Suppose this is not so. Let us construct an alternating sequence of white and
black edges as follows. Choose a black tile $\tau$ of height $h$ and a vertex
$v\in M(\tau)$. Let $e$ be the white edge $(v,t(\tau))$. Since $v$ is mixed (by
the supposition), there is a black tile $\tau'$ (of height $h$) such that
either (a) $v=r(\tau')$ or (b) $v=\ell(\tau')$. We say that $\tau'$ {\em lies
on the left from $\tau$} in the former case, and {\em lies on the right from
$\tau$} in the latter case. Let $u'$ be the right-to-top edge
$(r(\tau'),t(\tau'))$ of $\tau'$ in case (a), and the left-to-top edge
$(\ell(\tau'),t(\tau'))$ in case (b). Case (a) is illustrated on the right
fragment of the above picture.

Repeat the procedure for $\tau'$: choose $v'\in M(\tau')$ (which is mixed
again by the supposition); put $e':=(v',t(\tau'))$; choose a black tile
$\tau''$ such that either (a) $v'=r(\tau'')$ or (b) $v'=\ell(\tau'')$; and
define $u'$ to be be the edge $(r(\tau''),t(\tau''))$ in case (a), and the
edge $(\ell(\tau''),t(\tau''))$ in case (b). Repeat the procedure for
$\tau''$, and so on. Sooner or later we must return to a black tile that
has occurred earlier in the process. Then we obtain an alternating cycle
of white and black edges.

More precisely, there appear a cyclic sequence of different black tiles
$\tau_1,\ldots,\tau_{r-1},\tau_r=\tau_0$ of height $h$ and an alternating
sequence of white and black edges $C=(e_0=e_r,u_1,e_1,\ldots,u_r=u_0)$
(forming a cycle in $G_T$) with the following properties, for
$q=1,\ldots,r$: (a) $e_q$ is the edge $(v_q,t(\tau_q))$ for some $v_q\in
M(\tau_q)$; (b) $\tau_{q+1}$ is a black tile whose right of left vertex is
$v_q$; and (c) $u_{q+1}=(r(\tau_{q+1}),t(\tau_{q+1}))$ when
$r(\tau_{q+1})=v_q$, and $u_{q+1}=(\ell(\tau_{q+1}),t(\tau_{q+1}))$ when
$\ell(\tau_{q+1})=v_q$, where the indices are taken modulo $r$. We
consider $C$ up renumbering the indices cyclically and assume that
$\tau_q$ is an $i_qk_q$-tile, that $e_q$ is a $j_q$-edge, and that $u_q$
is a $p_q$-edge. Then $i_q<j_q<k_q$, $p_q=i_q$ if $\tau_q$ lies on the
left from $\tau_{q-1}$, and $p_q=k_q$ if $\tau_q$ lies on the right from
$\tau_{q-1}$. Note that in the former (latter) case the vertex $v_q$ lies
on the left (resp. right) from $v_{q-1}$ in the horizontal line at height
$h$ in $Z$. This implies that there exists a $q$ such that $\tau_q$ lies
on the left from $\tau_{q-1}$, and there exists a $q'$ such that $\tau_q'$
lies on the right from $\tau_{q'-1}$.

To come to a contradiction, consider a maximal subsequence $Q$ of consecutive
tiles in which each but first tile lies on the left from the previous one; ~one
may assume that $Q=(\tau_1,\tau_2,\ldots,\tau_d)$. Then $\tau_1$ lies on the
right from $\tau_0$; ~so $u_1$ is the left-to-top edge of $\tau_1$, whence
$p_1=k_1$. Also we observe that
  \begin{equation} \label{eq:monotone}
  k_1\ge k_2\ge\ldots\ge k_d.
  \end{equation}
Indeed, for $1\le q<d$, let $\lambda$ be the (white) tile containing the edge
$e_q$ and such that $r(\lambda)=v_q$. This tile lies in the cone of $\tau_q$ at
$t(\tau_q)$. So $\lambda$ is an $i'k'$-tile with $i_q\le i'<k'\le k_q$, and
therefore, the (bottom-to-right) edge $\tilde e$ of $\lambda$ entering $v_q$
has color $k'\le k_q$. Using~\refeq{mixedvert}(iii), we observe that $\tilde e$
lies in the cone of the black tile $\tau_{q+1}$ at $v_q$ (taking into account
that $v_q=r(\tau_{q+1})$). This implies that the bottom-to-right edge
$e'=(b(\lambda),v_q)$ of $\tau_{q+1}$ has color at most $k'$. Since $u_{q+1}$
is parallel to $e'$, we obtain $k_{q+1}\le k'\le k_q$, as required.

By~\refeq{monotone}, we have $j_q<k_q\le k_1=p_1$ for all $q=1,\ldots,d$. Also
if a tile $\tau_{q'}$ lies on the right from the previous tile $\tau_{q'-1}$,
then $u_{q'}$ is the left-to-top edge of $\tau_{q'}$, whence
$j_{q'}<k_{q'}=p_{q'}$. Thus, the maximum of $p_1,\ldots,p_r$ is strictly
greater than the maximum of $j_1,\ldots,j_r$. This is impossible since all
$u_1,\ldots,u_r$ are forward edges, all $e_1\ldots,e_r$ are backward edges in
$C$, and the graph $G_T$ is graded.
 \end{proof}

This completes the proof of Theorem~\ref{tm:bas-til}. \medskip

 \noindent{\bf Remark 3}
~For black tiles $\tau,\tau'\in T^b$, let us denote $\tau'\prec^\diamond \tau$
if there is a white edge $(v,t(\tau))$ such that $v$ is the right or left
vertex of $\tau'$. The proof of Proposition~\ref{pr:downflip} gives the
following additional result.

 \begin{corollary} \label{cor:poset}
The relation $\prec^\diamond$ determines a partial order on $T^b$.
  \end{corollary}
Similarly, the relation $\prec_\diamond$ determines a partial order on $T^b$,
where for $\tau,\tau'\in T^b$, we write $\tau'\prec_\diamond \tau$ if there is
a white edge $(b(\tau),v)$ such that $v$ is the right or left vertex of
$\tau'$. It is possible that the graph on $T^b$ induced by $\prec^\diamond$ (or
$\prec_\diamond$) is a forest, but we do not go in our analysis so far.
\medskip

We conclude this section with one more result which easily follows from
Proposition~\ref{pr:downflip}.
   \begin{prop}  \label{pr:int-lflip}
Let a g-tiling $T$ be such that, for some $h<n$, all non-terminal vertices of
height $h+1$ are intervals in $[n]$ and there is no feasible W-configuration of
height $h$. Then all non-terminal vertices of height $h$ are intervals as well.
Symmetrically, if for some $h>0$, all non-terminal vertices of height $h-1$ are
co-intervals and there is no feasible M-configuration of height $h$, then all
non-terminal vertices of height $h$ are co-intervals.
  \end{prop}
  \begin{proof}
~If $T$ has a black tile of height $h$, then there exists a feasible
W-configuration of height $h$, by Proposition~\ref{pr:downflip}. So this is not
the case.

Let $u$ be a non-terminal vertex of height $h$ and take an edge $(u,v)$. Then
the vertex $v$ (of height $h+1$) is non-terminal (for otherwise $v$ would be
the top vertex of a black tile of height $h$). Let $\tau,\tau'$ be the tiles
sharing the edge $(u,v)$; then both tiles are white and non-overlapping.
Suppose $b(\tau)=u$. Then both vertices $\ell(\tau),r(\tau)$ lie in level
$h+1$, and therefore, they are intervals. This easily implies that $u$ is an
interval as well. Similarly, $u$ is an interval if $b(\tau')=u$.

Now suppose that $u$ is neither $b(\tau)$ nor $b(\tau')$. Then
$t(\tau)=t(\tau')=v$. Letting for definiteness that $u=r(\tau)=\ell(\tau')$, we
obtain that the vertices $u,b(\tau),b(\tau'),\ell(\tau),r(\tau')$ form a
feasible W-configuration of height $h$; a contradiction (the case when
$b(\tau)$ or $b(\tau')$ is terminal is impossible, otherwise it would belong to
a black tile of height $h$).

The second assertion in the theorem follows from the first one by considering
the reversed tiling.
  \end{proof}

(Note that if all non-terminal vertices of height $h$ in a g-tiling are
intervals, then there is no feasible W-configuration of height $h$. Indeed,
suppose such a configuration $CW(X;i,j,k)$ exists. Then $Xik$ is a non-terminal
vertex of height $h$. Since $i<j<k$ and $j\not\in X$, the set $Xik$ is not an
interval.)
\smallskip

In view of the coincidence of the set of spectra of g-tilings on $Z_n$ with the
set of largest weakly separated collections in $2^{[n]}$ (proved in
Section~\ref{sec:wsf}), Proposition~\ref{pr:int-lflip} answers affirmatively
Conjecture~5.5 in~Leclerc and Zelevinsky~\cite{LZ}.

Finally, analyzing the proof of Proposition~\ref{pr:int-lflip}, one can see
that this proposition can be slightly strengthened as follows: if there is no
feasible W-configuration of height $h$, then each non-terminal vertex
$Y\subset[n]$ of height $h$ not contained in the boundary of $Z_n$ is
representable as $Y'\cap Y''$ for some non-terminal vertices $Y',Y''\subset[n]$
of height $h+1$. (Similarly, if there is no feasible M-configuration of height
$h$, then each non-terminal vertex $Y$ of height $h$ not contained in the
boundary of $Z_n$ is representable as $Y'\cup Y''$ for some non-terminal
vertices $Y',Y''$ of height $h-1$.)

\section{From generalized tilings to proper wirings}  \label{sec:til-to-wir}

In this section we show the following
   \begin{prop} \label{pr:til-wir}
For any g-tiling $T$ on $Z_n$, there exists a proper wiring $W$ on $Z_n$ such
that $B_W=B_T$.
   \end{prop}

This and the converse assertion established in the next section will imply that
the collection $\Bscr\Tscr_n$ of the spectra $B_T$ of g-tilings on $Z_n$
coincides with the collection $\Bscr\Wscr_n$ of the spectra $B_W$ of proper
wirings $W$ on $Z_n$, and then Theorem~\ref{tm:bas-wir} will follow from
Theorem~\ref{tm:bas-til}.
\medskip

   \noindent {\bf Proof}~
For convenience we identify and use the same notation for vertices, edges and
tiles concerning a g-tiling $T$ on $Z=Z_n$ and their corresponding points,
line-segments and squares (respectively) in the disc $D_T$ (every time it will
be clear from the context or explicitly indicated which of $Z$ and $D_T$ we
deal with). Accordingly, the planar graph $G_T=(V_T,E_T)$ is regarded as
properly embedded in $D_T$.

In order to construct the desired wiring, we first draw curves on $D_T$
associated to strips (dual paths) in $G_T$. More precisely, for each $i\in[n]$,
take the $i$-strip $Q_i=(e_0,\tau_1,e_1, \ldots,\tau_r,e_r)$ for $T$ (defined
in Section~\ref{sec:bas-to-til}), considering it as the corresponding sequence
of edges and squares in $D_T$. (Recall that $Q_i$ contains all $i\ast$- and
$\ast i$-tiles in $T$, $e_0$ is the edge $p_{i-1}p_i$ on the left boundary
$\ell bd(Z)$, and $e_r$ is the edge $p'_ip'_{i-1}$ on $rbd(Z)$; cf.
Corollary~\ref{cor:strips}.) For $q=1,\ldots,r$, draw the line-segment on the
square $\tau_q$ connecting the median points of the edges $e_{r-1}$ and $e_r$.
This segment meets the central point of $\tau_q$, denoted by $c(\tau_q)$. The
concatenation of these segments gives the desired (piece-wise linear) curve
$\zeta_i$ corresponding to $Q_i$; we direct $\zeta_i$ according to the
direction of $Q_i$.

Fix a homeomorphic map $\gamma:D_T\to Z$ such that each boundary edge of $D_T$
is linearly mapped onto the corresponding edge of $bd(Z)$. Then the curves
(``wires'') $\zeta_i$ on $D_T$ generate the wires $w_i:=\gamma(\zeta_i)$ on $Z$
(where $w_i$ begins at the median point $s_i$ of $p_{i-1}p_i$ on $\ell bd(Z)$
and ends at the medial point $s'_i$ of $p'_ip'_{i-1}$ on $rbd(Z)$). We assert
that the wiring $W=(w_1,\ldots,w_n)$ is as required in the proposition.

Obviously, $W$ satisfies axiom~(W1) for $W$. To verify the other axioms, we
first explain how the planar graphs $G_T$ and $G_\zeta$ on $D_T$ are related to
each other, where $G_\zeta$ is the ``preimage'' by $\gamma$ of the graph $G_W$
(it is defined as in Subsection~\ref{ssec:wiring} by considering
$\zeta_1,\ldots,\zeta_n$ in place of $w_1,\ldots,w_n$). The vertices of
$G_\zeta$ are the central points $c(\tau)$ of squares $\tau$ (where
corresponding wires cross one another) and the points $s_i,s'_i$. (We identify
corresponding points on the boundaries of $D_T$ and $Z$, writing $s_i$ for
$\gamma^{-1}(s_i)$.) Each vertex $v$ of $G_T$ one to-one corresponds to the
face of $G_\zeta$ where $v$ is located, denoted by $v^\ast$. The edges of color
$i$ in $G_\zeta$ (which are the pieces of $\zeta_i$ obtained by its subdivision
by the central points of squares lying on $\zeta_i$) one-to-one correspond to
the $i$-edges of $G_T$. More precisely, if an $i$-edge $e\in E_T$ belongs to
squares $\tau,\tau'$ and if $\tau,e,\tau'$ occur in this order in the
$i$-strip, then the $i$-edge of $G_\zeta$ corresponding to $e$, denoted by
$e^\ast$, is the piece of $\zeta_i$ between $c(\tau)$ and $c(\tau')$, and this
$e^\ast$ is directed from $c(\tau)$ to $c(\tau')$. Observe that $e$ crosses
$e^\ast$ {\em from right to left} on the disc. (We assume that the clockwise
orientation on $D_T$ is agreeable by $\gamma$ with that of $Z$.) The first and
last pieces of $\zeta_i$ correspond to the boundary $i$-edges $p_{i-1}p_i$ and
$p'_ip'_{i-1}$ of $G_T$, respectively.

Consider an $ij$-tile $\tau\in T$, and let $e,e'$ be its $i$-edges, and $u,u'$
its $j$-edges, where $e,u$ leave $b(\tau)$ and $e',u'$ enter $t(\tau)$. We know
that: (a) if $\tau$ is white, then $e$ occurs in $Q_i$ before $e'$, while $u$
occurs in $Q_j$ after $u'$, and (b) if $\tau$ is black, then $e$ occurs in
$Q_i$ after $e'$, while $u$ occurs in $Q_j$ before $u'$. In each case, in the
disc $D_T$, both $e,e'$ cross the wire $\zeta_i$ from right to left (w.r.t. the
direction of $\zeta_i$), and similarly both $u,u'$ cross $\zeta_j$ from right
to left. Also it is not difficult to see (using~(T1),(T2)) that when $\tau$ is
white, the orientation of the tile $\tau$ in $Z$ coincides with that of the
square $\tau$ in $D_T$, whereas when $\tau$ is black, the clockwise orientation
of $\tau$ in $Z$ turns in the counterclockwise orientation of $\tau$ in $D_T$
(causing the ``orientation-reversing'' behavior of wires at the vertex
$c(\tau)$ of $G_\zeta$). It follows that: in case~(a), $\zeta_j$ crosses
$\zeta_i$ at $c(\tau)$ from left to right, and therefore, the vertex $c(\tau)$
of $G_\zeta$ is white, and in case~(b), $\zeta_j$ crosses $\zeta_i$ at
$c(\tau)$ from right to left, and therefore, the vertex $c(\tau)$ is black.
(Both cases are illustrated in the picture.) So the white (black) tiles of $T$
generate the white (resp. black) vertices of $G_\zeta$.

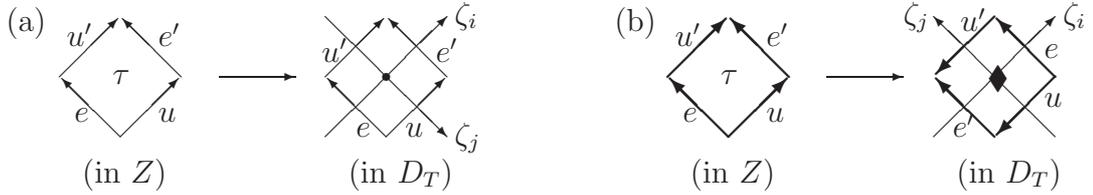
 \begin{center}
  \unitlength=1mm
  \begin{picture}(140,30)
   \put(0,0){\begin{picture}(60,30)
  \put(15,6){\vector(1,1){7.7}}
  \put(7,14){\vector(1,1){7.7}}
  \put(15,6){\vector(-1,1){7.7}}
  \put(23,14){\vector(-1,1){7.7}}
  \put(0,20){(a)}
  \put(9,8){$e$}
  \put(20,8){$u$}
  \put(8,18){$u'$}
  \put(20,18){$e'$}
  \put(14,13){$\tau$}
  \put(10,0){(in $Z$)}
  \put(28,14){\vector(1,0){10}}
  \put(50,6){\vector(1,1){7.7}}
  \put(42,14){\vector(1,1){7.7}}
  \put(50,6){\vector(-1,1){7.7}}
  \put(58,14){\vector(-1,1){7.7}}
  \put(50,14){\circle*{1}}
  \put(46,6){$e$}
  \put(52.5,6){$u$}
  \put(41.5,16){$u'$}
  \put(56.5,16){$e'$}
  \put(45,0){(in $D_T$)}
  \put(42,6){\vector(1,1){16}}
  \put(42,22){\vector(1,-1){16}}
  \put(59,21){$\zeta_i$}
  \put(59,5){$\zeta_j$}
    \end{picture}}
   \put(80,0){\begin{picture}(60,30)
{ \thicklines
  \put(15,6){\vector(1,1){7.7}}
  \put(7,14){\vector(1,1){7.7}}
  \put(15,6){\vector(-1,1){7.7}}
  \put(23,14){\vector(-1,1){7.7}}
  }
  \put(0,20){(b)}
  \put(9,8){$e$}
  \put(20,8){$u$}
  \put(8,18){$u'$}
  \put(20,18){$e'$}
  \put(14,13){$\tau$}
  \put(10,0){(in $Z$)}
  \put(28,14){\vector(1,0){10}}
 {\thicklines
  \put(58,14){\vector(-1,-1){7.7}}
  \put(50,22){\vector(-1,-1){7.7}}
  \put(50,6){\vector(-1,1){7.7}}
  \put(58,14){\vector(-1,1){7.7}}
  }
  \put(49,12.5){$\blacklozenge$}
  \put(44.5,6){$e'$}
  \put(56.5,10){$u$}
  \put(45.5,20){$u'$}
  \put(56.5,16.5){$e$}
  \put(45,0){(in $D_T$)}
  \put(42,6){\vector(1,1){16}}
  \put(58,6){\vector(-1,1){16}}
  \put(59,21){$\zeta_i$}
  \put(38,21){$\zeta_j$}
    \end{picture}}
  \end{picture}
   \end{center}

Consider a vertex $v$ of $G_T$ and an edge $e\in E_T(v)$. Then the edge
$e^\ast$ belongs to the boundary of the face $v^\ast$ of $G_\zeta$.  As
mentioned above, $e$ crosses $e^\ast$ from right to left on $D_T$. This implies
that $e^\ast$ is directed clockwise around $v^\ast$ if $e$ leaves $v$, and
counterclockwise if $e$ enters $v$. In view of axiom~(T3), we obtain that
  \begin{numitem1}
the terminal vertices of $G_T$ and only these generate cyclic faces of
$G_W\simeq G_\zeta$; moreover, for $\tau\in T^b$, the boundary cycle of
$(t(\tau))^\ast$ is directed counterclockwise, while the boundary cycle of
$(b(\tau))^\ast$ is directed clockwise.
  \label{eq:cyc-term}
  \end{numitem1}
(We use the fact that any non-terminal vertex $v\ne p_0,p_n$ has both entering
and leaving edges, and therefore, the boundary of the face $v^\ast$ has edges
in both directions. When $v=p_0$ ~($v=p_n$), a similar fact for $v^\ast$ is
valid as well, since $v^\ast$ contains the edges $(p_0,s_1)$ and $(p_0,s'_{n})$
~(resp. $(s_n,p_n)$ and $(s'_1,p_n)$) lying on the boundary of $Z$.)

Next, for each $i\in[n]$, removing from $D_T$ the interior of the $i$-strip
$Q_i$ (i.e., the relative interiors of all edges and tiles in it) results in
two closed regions $\Omega_1,\Omega_2$, the former containing the vertex
$\emptyset$, and the latter containing the vertex $[n]$ (regarding the vertices
as subsets of $[n]$). The fact that all edges in $Q_i$ (which are the $i$-edges
of $G_T$) go from $\Omega_1$ to $\Omega_2$ implies that each vertex $v$ of
$G_T$ occurring in $\Omega_1$ (in $\Omega_2$) is a subset in $[n]$ not
containing (resp. containing) the element $i$. So $i\not\in X(v^\ast)$ if
$v\in\Omega_1$, and $i\in X(v^\ast)$ if $v\in\Omega_2$. This implies the
desired equality for spectra: $B_W=B_T$.

A less trivial task is to check validity of~(W2) for $W$. One can see that
axiom~(W2) is equivalent to the condition that if wires $w_i,w_j$ intersect at
a point $x$ and this point is white, then the parts of $w_i,w_j$ after $x$ do
not intersect. So it suffices to show the following \medskip

 \noindent {\bf Claim}~
Let wires $\zeta_i,\zeta_j$ with $i<j$ intersect at a white point $x$.
Then the part $\zeta$ of $\zeta_i$ from $x$ to $s'_i$ and the part
$\zeta'$ of $\zeta_j$ from $x$ to $s'_j$ have no other common points.
\medskip

  \noindent {\bf Proof of the Claim}~
Suppose this is not so and let $y$ be the common point of $\zeta,\zeta'$ next
to $x$ in $\zeta$. Since $x$ is white, $y$ is black. Therefore, the $ij$-tile
$\tau$ such that $x=c(\tau)$ is white, and the $ij$-tile $\tau'$ such that
$y=c(\tau')$ is black. Also in both strips $Q_i$ and $Q_j$, the tile $\tau$
occurs earlier than $\tau'$. One can see (cf.~\refeq{strip_tiles}) that in the
strip $Q_i$, the edge succeeding $\tau$ is $(r(\tau),t(\tau))$ and the edge
preceding $\tau'$ is $(r(\tau'),t(\tau'))$, whereas in the strip $Q_j$, the
edge succeeding $\tau$ is $(b(\tau),r(\tau))$ and the edge preceding $\tau'$ is
$(b(\tau'),r(\tau'))$. So the right boundary of $Q_i$ passes the vertices
$r(\tau)$ and $r(\tau')$, in this order, and similarly for the left boundary of
$Q_j$.

Consider the part $R$ of $R_{Q_i}$ from $r(\tau)$ to $r(\tau')$ and the part
$L$ of $L_{Q_j}$ from $r(\tau)$ to $r(\tau')$. For $q\in[n]$, let $\alpha_q$,
$\alpha'_q$, $\beta_q$, $\beta'_q$ be the numbers of $q$-edges that are forward
in $R$, forward in $L$, backward in $R$, and backward in $L$, respectively.
Since $G_T$ is graded, we have ($\ast$) $\alpha_q-\beta_q=\alpha'_q-\beta'_q$.

Next we argue in a similar spirit as in the proof of Lemma~\ref{lm:no-cycl}.
Define $I:=[i+1..j-1]$, ~$\Delta:=\sum_{q\in I}(\alpha_q-\beta_q)$, and
~$\Delta':=\sum_{q\in I}(\alpha'_q-\beta'_q)$. We assert that $\Delta>0$ and
$\Delta'<0$, which leads to a contradiction with ($\ast$) above.

To see $\Delta>0$, consider a $q$-edge $e$ in $R$ with $q\in I$, and let
$\tau^e$ denote the tile in $Q_i$ containing $e$. Since $q>i$, ~$\tau^e$
is white if $e$ is forward, and $\tau^e$ is black if $e$ is backward in
$R$ (cf.~\refeq{strip_boundary}). Using this, one can see that:

(i) for a $q$-edge $e\in R$ such that $q\in I$ and $\tau^e$ is black, the next
edge $e'$ in $R_{Q_i}$ is a forward $q'$-edge in $R$ with $q'\in I$ (since the
fact that $\tau^e$ is a black $iq$-tile implies that $\tau^{e'}$ is a white
$q'q$-tile with $i<q'<q$, in view of~\refeq{termv}); a similar property holds
for the edge in $R_{Q_i}$ preceding $e$;

(ii) the last edge $e$ of $R$ is a forward $q$-edge with $q\in I$ (since
the tile $\tau^e$ shares an edge with the black $ij$-tile $\tau'$);

(iii) if the first edge $e$ of $R$ is backward, then it is a $q$-edge with
$q\not\in I$ (since $\tau^e$ is black and shares an edge with the white
$ij$-tile $\tau$).

These observations show that the first and last edges of any maximal
alternating $I$-subpath $P$ in $R$ are forward, and therefore, $P$ contributes
$+1$ to $\Delta$. Also at least one such $P$ exists, by~(ii). So $\Delta>0$, as
required.

The inequality $\Delta'<0$ is shown in a similar way, by considering $L$ and
swapping ``forward'' and ``backward'' in the above reasonings (due to replacing
$q>i$ by $q<j$). More precisely, for a $q$-edge $e$ in $L$ with $q\in I$, the
tile $\tau^e$ in $Q_j$ containing $e$ is black if $e$ is forward, and white if
$e$ is backward (in view of $q<j$ and~\refeq{strip_boundary}). This implies
that:

(i') for a $q$-edge $e\in L$ such that $q\in I$ and $\tau^e$ is black, the next
edge $e'$ in $L_{Q_j}$ is a backward $q'$-edge in $L$ with $q'\in I$; and
similarly for the previous edge in $L_{Q_j}$;

(ii') the last edge $e$ of $L$ is a backward $q$-edge with $q\in I$;

(iii') if the first edge $e$ of $R$ is forward, then it is a $q$-edge with
$q\not\in I$.

Then the first and last edges of any maximal alternating $I$-subpath $P$ in $L$
are backward, and therefore, $P$ contributes $-1$ to $\Delta'$. Also at least
one such $P$ exists, by~(ii'). Thus, $\Delta'<0$, obtaining a contradiction
with $\Delta=\Delta'$. \hfill\qed

 \medskip
Thus, (W2) is valid. Considering lenses formed by a pair of wires and
using~\refeq{cyc-term} and~(W2), one can easily obtain~(W3). Finally, since
$|E_T(v)|\ge 3$ holds for each terminal vertex in $G_T$ (by
Corollary~\ref{cor:termv}(ii)), each cyclic face in $G_W$ is surrounded by at
least three edges, and therefore, this face cannot be a lens. So the wiring $W$
is proper.

This completes the proof of Proposition~\ref{pr:til-wir}.
  \hfill\qed\qed

\section{From proper wirings to generalized tilings}  \label{sec:wir-to-til}

In this section we complete the proof of Theorem~\ref{tm:bas-wir} by showing
the converse to Proposition~\ref{pr:til-wir}.
   \begin{prop} \label{pr:wir-til}
For a proper wiring $W$ on $Z=Z_n$, there exists a g-tiling $T$ on $Z$
such that $B_T=B_W$.
   \end{prop}

 \begin{proof}
The construction of the desired $T$ is converse, in a sense, to that described
in the proof of Proposition~\ref{pr:til-wir}; it combines planar duality
techniques and geometric arrangements.

We associate to each (inner) face $F$ of the graph $G_W$ the point (viz.
the subset) $X(F)$ in the zonogon, also denoted as $F^\ast$. These points
are just the vertices of tiles in $T$. The edges concerning $T$ are
defined as follows. Let faces $F,F'\in\Fscr_W$ have a common edge $e$
formed by a piece of a wire $w_i$, and let $F$ lie on the right from $w_i$
according to the direction of this wire (and $F'$ lies on the left from
$w_i$). Then the vertices $F^\ast,{F'}^\ast$ are connected by edge
$e^\ast$ going from $F^\ast$ to ${F'}^\ast$. Note that in view of the
evident relation $X(F')=X(F)\cup\{i\}$, the direction of $e^\ast$ matches
the edge direction for g-tilings.

The tiles in $T$ one-to-one correspond to the intersection points of wires
in $W$. More precisely, let $v$ be a common point of wires $w_i,w_j$ with
$i<j$. Then the vertex $v$ of $G_W$ has four incident edges $e_i,\bar
e_i,e_j,\bar e_j$ such that: $e_i,\bar e_i\subset w_i$; ~$e_j,\bar
e_j\subset w_j$; ~$e_i,e_j$ enter $v$; ~and $\bar e_i,\bar e_j$ leave $v$.
Also one can see that for the four faces $F$ containing $v$, the subsets
$X(F)$ are of the form $X,Xi,Xj,Xij$ for some $X\subset[n]$. The tile
surrounded by the edges $e^\ast_i,\bar e^\ast_i,e^\ast_j,\bar e^\ast_j$
connecting these subsets (regarded as points) is just the $ij$-tile in $T$
corresponding to $v$, denoted as $v^\ast$. Observe that the edges
$e_i,e_j,\bar e_i,\bar e_j$ follow in this order counterclockwise around
$v$ if $v$ is black (orientation-reversing), and clockwise otherwise (when
$v$ is ``orientation-respecting''). The tile $v^\ast$ is regarded as black
in $T$ if $v$ is black, and white otherwise. Both cases are illustrated in
the picture where the right fragment concerns the orientation-reversing
case.

 \begin{center}
  \unitlength=1mm
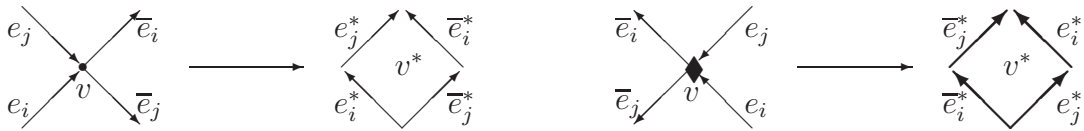
  \begin{picture}(140,16)
   \put(0,0){\begin{picture}(60,16)
  \put(8,8){\circle*{1}}
  \put(0,0){\vector(1,1){7.5}}
  \put(0,16){\vector(1,-1){7.5}}
  \put(8,8){\vector(1,1){7.5}}
  \put(8,8){\vector(1,-1){7.5}}
  \put(22,8){\vector(1,0){15}}
  \put(50,0){\vector(1,1){7.5}}
  \put(50,0){\vector(-1,1){7.5}}
  \put(42,8){\vector(1,1){7.5}}
  \put(58,8){\vector(-1,1){7.5}}
  \put(-2,2){$e_i$}
  \put(-2,12){$e_j$}
  \put(15,2){$\bar e_j$}
  \put(15,12){$\bar e_i$}
  \put(7,4){$v$}
  \put(41,2){$e^\ast_i$}
  \put(41,12){$e^\ast_j$}
  \put(56,2){$\bar e^\ast_j$}
  \put(56,12){$\bar e^\ast_i$}
  \put(49,7){$v^\ast$}
    \end{picture}}
   \put(80,0){\begin{picture}(60,16)
  \put(7,6.5){$\blacklozenge$}
  \put(8,8){\vector(-1,-1){7.5}}
  \put(8,8){\vector(-1,1){7.5}}
  \put(16,16){\vector(-1,-1){7.0}}
  \put(16,0){\vector(-1,1){7.0}}
  \put(22,8){\vector(1,0){15}}
\thicklines{
  \put(50,0){\vector(1,1){7.5}}
  \put(50,0){\vector(-1,1){7.5}}
  \put(42,8){\vector(1,1){7.5}}
  \put(58,8){\vector(-1,1){7.5}}
  }
  \put(-2,2.5){$\bar e_j$}
  \put(-2,12){$\bar e_i$}
  \put(15,2){$e_i$}
  \put(15,12){$e_j$}
  \put(7,3.5){$v$}
  \put(41,2){$\bar e^\ast_i$}
  \put(41,12){$\bar e^\ast_j$}
  \put(56,2){$e^\ast_j$}
  \put(56,12){$e^\ast_i$}
  \put(49,7){$v^\ast$}

    \end{picture}}
  \end{picture}
   \end{center}

Next we examine properties of the obtained collection $T$ of tiles. The
first and second conditions in~(T2) (concerning overlapping and
non-overlapping tiles with a common edge) follow from the above
construction and explanations. \smallskip

1) Consider an $i$-edge $e=(u,v)$ in $G_W$ (a piece of the wire $w_i$). If
$u\ne s_i$ and $v\ne s'_i$, then the dual edge $e^\ast$ belongs to exactly
two tiles, namely, $u^\ast$ and $v^\ast$. If $u=s_i$, then $e^\ast$
belongs to the unique tile $v^\ast$. Furthermore, for the faces
$F,F'\in\Fscr_W$ containing $e$, the sets $X(F),X(F')$ are the principal
intervals $[i-1]$ and $[i]$ (letting $[0]:=\emptyset$). This implies that
$e^\ast$ is the boundary edge $p_{i-1}p_i$ of $Z$, and this edge belongs
to a unique tile in $T$ (which is, obviously, white). Considering
$v=s'_i$, we obtain a similar property for the edges in $rbd(Z)$. This
gives the first and second condition in~(T1).

The proper wiring $W$ possesses the following important property, which
will be proved later (see Lemma~\ref{lm:nocopies}): ($\ast$) each face $F$
in $G_W$ has at most one $i$-edge for each $i$, and all sets $X(F)$ among
$F\in\Fscr_W$ are different. This implies that $T$ has no tile copies,
yielding the third condition in~(T1). Also property ($\ast$) and the
planarity of $G_W$ imply validity of axiom~(T4). \smallskip

2) For a face $F\in\Fscr_W$, let $E(F)$ denote the set of its edges not
contained in $bd(Z)$. By the construction and explanations above,
  \begin{numitem1}
the edges in $E(F)$ one-to-one correspond to the edges incident to the
vertex $v=F^\ast$ of $G_T$; moreover, for $e\in E(F)$, the corresponding
edge $e^\ast$ enters $v$ if $e$ is directed counterclockwise (around $F$),
and leaves $v$ otherwise.
  \label{eq:Fast}
  \end{numitem1}
This implies that $v$ has both entering and leaving edges if and only if $F$ is
non-cyclic, unless $v=p_0$ or $p_n$. (Here we also use an easy observation that
if $F$ contains a vertex $p_i$ or $p'_i$ for some $1\le i<n$, then $E(F)$ has
edges in both directions.)

Consider a cyclic face $F\in\Fscr_W^{cyc}$, and let
$C=(v_0,e_1,v_1,\ldots,e_r,v_r=v_0$) be its boundary cycle, where for
$p=1,\ldots,r$, the edge $e_p$ goes from $v_{p-1}$ to $v_p$. Denote the color
of $e_p$ by $i_p$. Suppose $C$ is directed clockwise. Then for each $p$, we
have $i_p<i_{p+1}$ if $v_p$ is white, and $i_p>i_{p+1}$ if $v_p$ is black
(taking the indices modulo $r$). Hence $C$ contains at least one black point
(for otherwise we would have $i_1<\ldots<i_r<i_1$). Moreover,
   \begin{numitem1}
$C$ contains exactly one black point.
   \label{eq:one_black}
   \end{numitem1}
Indeed, let $v_p$ be black. Then $v_p$ is the root of the (even) lens $L$ of
wires $w_{i_{p-1}}$ and $w_{i_p}$ such that $F\subseteq L$. By axiom~(W3), $L$
is bijective to $F$. The existence of another black vertex in $C$ would cause
the appearance of another lens bijective to $F$, which is impossible.
So~\refeq{one_black} is valid. This implies that the vertex $F^\ast$ (which has
leaving edges only, by~\refeq{Fast}) is the bottom vertex of exactly one black
tile. When $C$ is directed counterclockwise, we have $i_p>i_{p+1}$ if $v_p$ is
white, and $i_p<i_{p+1}$ if $v_p$ is black, implying~\refeq{one_black} again,
which in turn implies that $F^\ast$ is the top vertex of exactly one black
tile. Thus, $T$ obeys~(T3).
\smallskip

3) If a cyclic face $F$ and another face $F'$ in $G_W$ have a common edge
$e=(u,v)$, then $F'$ is non-cyclic. Indeed, the edge $e'$ preceding $e$ in
the boundary cycle of $F$ enters the vertex $u$. The wire in $W$ passing
through $e'$ leaves $u$ by an edge $e''$. Obviously, $e''$ belongs to
$F'$. Since the edges $e,e''$ of $F'$ have the same beginning vertex, $F'$
is non-cyclic. Hence the cyclic faces in $G_W$ are pairwise disjoint,
implying that no pair of black tiles in $T$ share an edge (the third
condition in~(T2)). \smallskip

Thus, $T$ is a g-tiling. If a face $F$ of $G_W$ lies on the left from a wire
$w_i$, then the vertices $F^\ast$ and $[n]$ occur in the same region when the
interior of the $i$-strip is removed from the disc $D_T$. This implies that the
sets $X(F)$, $F\in\Fscr_W$, are just the vertices of $T$, i.e., the full
spectra for $T$ and $W$ are the same. Now the correspondence between cyclic
faces for $W$ and terminal vertices for $T$ yields $B_T=B_W$, as required.
  \end{proof}

It remains to show the following (cf.~($\ast$) in the above proof).
  \begin{lemma} \label{lm:nocopies}
Let $W$ be a proper wiring. Then:

(i) for each face $F$ in $G_W$, all edges surrounding $F$ belong to different
wires;

(ii) there are no different faces $F,F'\in\Fscr_W$ such that $X(F)=X(F')$.
  \end{lemma}
  \begin{proof}
Suppose that a face $F$ contains two $i$-edges $e,e'$ for some $i$. One can see
that: (a) $e,e'$ have the same direction in the boundary of $F$, and (b) the
face $F'\ne F$ containing $e$ is different from the face $F''\ne F$ containing
$e'$. Property~(a) implies $X(F')=X(F'')$. Therefore, (i) follows from (ii).

To show (ii), we use induction on $n$ (the assertion is obvious if $n=2$). Let
$W':=(w_1,\ldots,w_{n-1})$. Clearly $W'$ obeys axioms~(W1),(W2). One can see
that if none of cyclic faces of $G_{W'}$ is separated by $w_n$, then~(W3) is
valid for $W'$ as well. So suppose that some cyclic face $F$ of $G_{W'}$ is
separated by $w_n$. Let $v$ be a black vertex in $F$ (such a vertex exists, for
otherwise the edge colors along the boundary cycle of $F$ would be monotone
increasing or monotone decreasing, which is impossible). Then there is a lens
$L$ for $W'$ such that $F\subseteq L$ and $v$ is the root of $L$. Take the face
$\tilde F$ of $G_W$ such that $v\in \tilde F\subset F$. By~(W3) for $W$,
~$\tilde F$ is cyclic (since $L$ and $v$ continue to be a lens and its root
when $w_n$ is added to $W'$). Assume that the boundary cycle $C$ of $\tilde F$
is directed clockwise (when $C$ is directed counterclockwise, the argument is
similar). Since $\tilde F\ne F$, ~$C$ should contain consecutive edges
$e'=(u',u),e''=(u,u'')$  with colors $n$ and $i<n$, respectively. Then the wire
$w_i$ crosses $w_n$ at $u$ from left to right. This implies that the vertex $u$
is black and there is an $in$-lens $L'$ rooted at $u$ and containing $\tilde
F$. But $\tilde F$ is bijective to $L$; a contradiction.

Thus, $W'$ is a wiring (obeying (W1)--(W3)). We first prove~(ii) in the
assumption that $W'$ is proper. Then by induction all sets $X(F)$, $F\in
\Fscr_{W'}$, are different. Suppose that there are different faces $\tilde
F,\tilde F'\in\Fscr_W$ such that $X(\tilde F)=X(\tilde F')$. Let $F$ and $F'$
be the faces for $W'$ containing $\tilde F$ and $\tilde F'$, respectively. Then
$X(F)=X(\tilde F)-\{n\}$ and $X(F')=X(\tilde F')-\{n\}$. This implies
$X(F)=X(F')$, and therefore $F=F'$. Furthermore, $w_n$ separates $F$ at most
twice (for otherwise we would have $F=\tilde F\cup\tilde F'$, which implies
$X(\tilde F)\ne X(\tilde F')$).

It follows that $w_n$ and the boundary of $F$ have two common points $u,v$ such
that: (a) $u$ occurs in $w_n$ earlier than $v$, and (b) the piece $P$ of $w_n$
between $u$ and $v$ (not including $u,v$) lies outside $F$. Let $Q$ be the part
of the boundary of $F$ between $u$ and $v$ such that the simply connected
region $\Omega$ surrounded by $P$ and $Q$ is disjoint from the interior of $F$.
Consider the case when $P$ goes clockwise around $\Omega$; see the picture.
 \begin{center}
  \unitlength=1mm
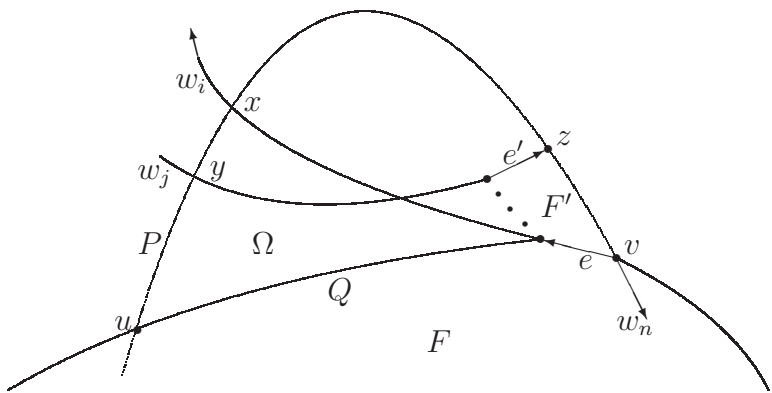
  \begin{picture}(100,55)
  \put(17,8){\circle*{1}}
  \put(70,20){\circle*{1}}
  \put(80,17.5){\circle*{1}}
  \put(80,17.5){\vector(-4,1){9.5}}
  \qbezier(0,0)(25,15)(70,20)
  \qbezier(80,17.5)(95,10)(100,0)
  \qbezier(15,2)(40,90)(80,17.5)
  \qbezier(70,20)(30,30)(25,44)
  \put(80,17.5){\vector(1,-2){4}}
  \put(25,44){\vector(-1,4){1}}
  \put(71,32){\circle*{1}}
  \put(63,28){\circle*{1}}
  \put(63,28){\vector(2,1){7.5}}
  \qbezier(20,31)(35,20)(63,28)
  \put(64.5,26){\circle*{0.7}}
  \put(66,24){\circle*{0.7}}
  \put(68,22){\circle*{0.7}}
  \put(14,8){$u$}
  \put(55,5){$F$}
  \put(42,12){$Q$}
  \put(75,16){$e$}
  \put(81,18){$v$}
  \put(17,28){$w_j$}
  \put(70,23){$F'$}
  \put(65,30){$e'$}
  \put(22,40){$w_i$}
  \put(32,18){$\Omega$}
  \put(17,18){$P$}
  \put(80,8){$w_n$}
  \put(26.5,28.5){$y$}
  \put(31,37){$x$}
  \put(72,32.5){$z$}
  \end{picture}
   \end{center}

Let $e$ be the edge in $E_W$ contained in $Q$ and incident to $v$; then $e$ has
color $i<n$. Take the maximal connected piece $Q'$ of $w_i$ lying in $\Omega$
and containing $e$. Since $w_i$ does not meet the interior of $F$, the end $x$
of $Q'$ different from $v$ lies on $P$. Then $Q$ and the piece $P'$ of $P$ from
$x$ to $v$ form an $in$-lens $L$ for $W$. Since $P'$ is directed from $x$ to
$v$, ~$Q'$ must be directed from $v$ to $x$ (by~(W2)); in particular, $e$
leaves $v$. So $w_n$ crosses $w_i$ at $v$ from right to left, and therefore,
the vertex $v$ is black and is the root of $L$. Let $F'$ be the cyclic face in
$G_W$ lying in $L$ and containing $v$, and let $C$ be its boundary cycle. Since
$W$ is proper, $F'\ne L$, whence $C\ne P'\cup Q'$. Then $C$ contains an edge
$e'$ with color $j\ne i,n$ (one can take as $e'$ the edge of $C$ that either
succeeds $e$ or precedes the last edge on $P$). Take the maximal connected
piece $R$ of $w_j$, from a point $y$ to a point $z$ say, that lies in $\Omega$
and contains $e'$. It is not difficult to realize that $y$ occurs in $P$
earlier than $z$. This violates~(W2) for $w_j,w_n$.

When $P$ goes counterclockwise around $\Omega$, a contradiction is shown in a
similar way. (In this case, we take as $e$ the edge on $Q$ incident to $u$; one
shows that $e$ enters $u$, whence the vertex $u$ is black.)
\smallskip

Finally, we assert that the wiring $W'$ is always proper. Indeed, suppose this
is not so and consider an ``empty'' $ij$-lens $L$ (where $i<j$), i.e., forming
a face in $G_{W'}$. One may assume that $L$ is an odd lens with lower vertex
$u$ and upper vertex $v$ (the case of an even ``empty'' lens is examined in a
similar way). Then $v$ is black (the root of $L$), and the boundary of $L$ is
formed by the piece $P$ of $w_i$ from $u$ to $v$ and the piece $Q$ of $w_j$
from $v$ to $u$ (giving the edges in $G_{W'}$ connecting $u$ and $v$). Consider
the (cyclic) face $F$ in $G_W$ lying in $L$ and containing $v$, and let $C$ be
its boundary cycle. Since $W$ is proper, $F\ne L$. Let $e=(x,v)$ be the
$i$-edge in $P$ entering $v$, and $e'=(v,y)$ the $j$-edge in $Q$ leaving $v$.
Besides these, the cycle $C$ contains some $n$-edge $e''$. Note that $e''$
cannot connect two points on $P$ or two points on $Q$, for otherwise there
would appear an ``empty'' $in$- or $jn$-lens. This implies that $e''$ goes from
$y$ to $x$ (respecting the direction on $C$). But then $w_n$ crosses $w_j$ from
right to left, whence the vertex $y$ is black. So the face $F$ contains two
black vertices, contradicting~\refeq{one_black}.

Thus, Lemma~\ref{lm:nocopies} is proven, and this completes the proof of
Proposition~\ref{pr:wir-til}. \hfill \qed
  \end{proof}

Propositions~\ref{pr:til-wir} and~\ref{pr:wir-til} imply the desired equality
$\Bscr\Tscr_n=\Bscr\Wscr_n$, and now Theorem~\ref{tm:bas-wir} follows from
Theorem~\ref{tm:bas-til}. Analyzing the transformation of a g-tiling into a
proper wiring described in Section~\ref{sec:til-to-wir} and the converse
transformation described above, one can conclude that their composition returns
the initial g-tiling (or the initial proper wiring). This implies the following
result (where, as before, $B_T$ and $\hat B_T$ stand for the effective and full
spectra of a g-tiling $T$, respectively, and similarly for wirings).
   \begin{theorem} \label{tm:TWbiject}
There is a bijection $\beta$ of the set ${\bf T}_n$ of g-tilings to the set
${\bf W}_n$ of proper wirings on $Z_n$ such that $B_T=B_{\beta(T)}$ holds for
each $T\in{\bf T}_n$. Furthermore, for each proper wiring $W$, all subsets
$X(F)\subseteq[n]$ determined by the faces $F$ for $W$ are different, and one
holds $\hat B_W=\hat B_T$, where $T=\beta^{-1}(W)$.
   \end{theorem}

We conclude this section with several remarks and additional results. \medskip

 \noindent{\bf Remark 4}
~As is shown in the proof of Lemma~\ref{lm:nocopies}, for any proper wiring
$W=(w_1,\ldots,w_n)$, the set $W'=(w_1,\ldots,w_{n-1})$ forms a proper wiring
as well (concerning the zonogon $Z_{n-1}$). Clearly a similar result takes
place when we remove the wire $w_1$. As a generalization, we obtain that for
any $1\le i<j\le n$, the set $(w_i,\ldots,w_j)$ forms a proper wiring on the
corresponding subzonogon. One can see that removing $w_n$ from $W$ corresponds
to the contracting operation concerning $n$ in the g-tiling $\beta^{-1}(W)$,
and this results in the set of tiles corresponding to $W'$. This gives the
following important result to which we have appealed in
Section~\ref{sec:bas-to-til}.
  \begin{corollary}  \label{cor:contract}
For a g-tiling $T$ on $Z_n$, its $n$-contraction $T/n$ is a (feasible)
g-tiling on $Z_{n-1}$.
  \end{corollary}

 \noindent{\bf Remark 5}
~Properties of g-tilings and proper wirings established during the proofs of
Theorems~\ref{tm:bas-wir} and~\ref{tm:bas-til} enable us to obtain the
following result saying that these objects are determined by their spectra.
   \begin{theorem} \label{tm:unique}
For each semi-normal basis $B$, there are a unique g-tiling $T$ and a unique
proper wiring $W$ such that $B=B_T=B_W$.
   \end{theorem}
  \begin{proof}
~Due to Theorem~\ref{tm:TWbiject}, it suffices to prove this uniqueness
property for g-tilings. We apply induction on $h(B):=\sum(|X|\colon X\in B)$.
Suppose there are different g-tiles $T,T'$ with $B_T=B_{T'}=:B$. This is
impossible when none of $T,T'$ has black tiles. Indeed, the vertices of $G_T$
and $G_{T'}$ (which are the sets in $B$) are the same and they determine the
edges of these graphs, by~\refeq{edge}. So $G_T=G_{T'}$. This graph is planar
and subdivides $Z_n$ into little parallelograms, which are just the tiles in
$T$ and the tiles in $T'$. Then $T=T'$. Now let $T$ (say) have a black tile. By
Proposition~\ref{pr:downflip}, $T$ has a feasible W-configuration
$CW(X;i,j,k)$, and we can make the corresponding lowering flip for $T$,
obtaining a g-tiling $\tilde T$ with $B_{\tilde T}= (B-\{Xik\})\cup\{Xj\}$.
Since $B_T=B_{T'}$, ~$CW(X;i,j,k)$ is a feasible W-configuration for $T'$ as
well, and making the corresponding lowering flip for $T'$, we obtain a g-tiling
$\tilde T'$ such that $B_{\tilde T'}=B_{\tilde T}$. We have $h(B_{\tilde
T})<h(B)$, whence, by induction, $\tilde T=\tilde T'$. But the raising flip in
$\tilde T$ w.r.t. the (feasible) M-configuration $CM(X;i,j,k)$ returns $T$, as
mentioned in Remark~2 in Subsection~\ref{ssec:part1}. Hence $T=T'$; a
contradiction.
  \end{proof}

\noindent One can develop an efficient procedure that, given the spectrum $B_T$
of a g-tiling $T$, restores $T$ itself (in essence, the procedure uses only
``local'' operations). This is provided by the possibility of constructing the
graph $G_T$, as follows. We know that $B_T$ is the set of non-terminal vertices
of $G_T$, and the edges connecting these vertices are of the form $(X,Xi)$ for
all corresponding $X,i$; let $G'$ be the graph formed by these vertices and
edges. The goal is to construct the terminal vertices (if any) and the
remaining edges of $G_T$ (in particular, obtaining the full spectrum $\hat
B_T$). This relies on the observation that each terminal vertex $X$ one-to-one
corresponds to a maximal collection $\Yscr\subset B_T$ such that $|\Yscr|\ge 3$
and: either~(a) each $Y\in\Yscr$ satisfies $Y=Xi$ for some $i$ and at least one
member of $\Yscr$ has no entering edge in $G'$; or~(b) each $Y\in\Yscr$
satisfies $Y=X-\{i\}$ for some $i$ and at least one member of $\Yscr$ has no
leaving edge in $G'$. In case~(a), $X$ is the bottom vertex of a black tile
$\tau$, and $E_T(X)=\{(X,Y)\colon Y\in\Yscr\}$. In case~(b), $X$ is the top
vertex of a black tile $\tau$, and $E_T(X)=\{(Y,X)\colon Y\in\Yscr\}$. So, by
extracting all such collections $\Yscr$, we are able to obtain the whole $G_T$.
Now to construct the tiles of $T$ is easy (using~\refeq{mixedvert}). A slightly
modified efficient procedure can be applied to an arbitrary $B\subset 2^{[n]}$
to decide whether or not $B$ is a semi-normal basis. \medskip

  \noindent{\bf Remark 6}~
The authors can propose an alternative method of proving
Theorems~\ref{tm:bas-wir} and~\ref{tm:bas-til} in which the former theorem is
proved directly and then the latter is obtained via the relationship of
g-tilings and proper wirings established in this and previous sections. (Other
possible methods: prove the first (second) assertion in
Theorem~\ref{tm:bas-wir} and the second (resp. first) assertion in
Theorem~\ref{tm:bas-til}.) The alternative method is based on ideas and
techniques different from those applied in
Sections~\ref{sec:bas-to-til},\ref{sec:til-to-bas}: the former extensively
exploit Jordan curve theorem, while the latter typically appeal to the fact
that the graph of a g-tiling is graded for each color. For some illustration,
we outline how the lowering flip is viewed on the language of wirings, for
simplicity considering the situation corresponding to Case~1(a) in the proof of
Proposition~\ref{pr:Tflip}. Here we handle three wires $w_i,w_j,w_k$ of a
wiring $W$ such that $i<j<k$ and there are five non-cyclic faces $A,B,C,D,E$
whose local configuration is as illustrated on the left fragment of the picture
below. The sets $X(A),X(B),X(C),X(D),X(E)$ are, respectively,
$Xij,Xi,Xik,Xk,Xjk$, the face $C$ looks like a triangle, and no other wire in
$W$ traverses some open neighborhood $\Omega$ of $C$.

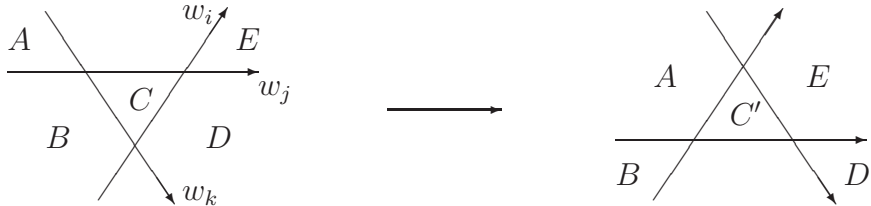
 \begin{center}
  \unitlength=1mm
  \begin{picture}(120,25)
   \put(0,0){\begin{picture}(70,25)
  \put(0,17){\vector(1,0){33}}
  \put(12,0){\vector(2,3){17}}
  \put(5,25){\vector(2,-3){17}}
  \put(0,20){$A$}
  \put(5,7){$B$}
  \put(16,12){$C$}
  \put(26,7){$D$}
  \put(30,20){$E$}
  \put(23,24){$w_i$}
  \put(33,14){$w_j$}
  \put(23,0){$w_k$}
  \put(50,12){\vector(1,0){15}}
    \end{picture}}
   \put(80,0){\begin{picture}(40,25)
  \put(0,8){\vector(1,0){33}}
  \put(5,0){\vector(2,3){17}}
  \put(12,25){\vector(2,-3){17}}
  \put(5,15){$A$}
  \put(0,2){$B$}
  \put(15,10){$C'$}
  \put(30,2){$D$}
  \put(25,15){$E$}
    \end{picture}}
  \end{picture}
   \end{center}
The lowering flip replaces $Xik$ by $Xj$. This corresponds to a deformation of
the wire $w_j$ within $\Omega$ which makes it pass below the intersection point
of $w_i$ and $w_k$, as illustrated on the right fragment of the picture. The
triangle-shaped face $C'$ arising instead of $C$ satisfies $X(C')=Xj$ and is
non-cyclic (as well as the modified $A,B,D,E$).

Note that if, in the initial wiring $W$, the face above $C$ is also
triangle-shaped (i.e., the wires $w_i,w_k$ form a lens $L$ containing $C$ and
such that $w_j$ is the unique wire going across $L$), then the above flip turns
$L$ into a face and the wiring becomes non-proper. In this case the flip should
be followed by the corresponding operation \intlens$\to$\apartlens which
eliminates $L$ and makes the wiring proper again. Such a transformation of $W$
corresponds to that in Case~2(a') in the proof of Proposition~\ref{pr:Tflip}.


 \section{$n$-contraction and $n$-expansion} \label{sec:contr-exp}

In Subsection~\ref{ssec:part3} we introduced the $n$-contraction operation for a
g-tiling on the zonogon $Z_n$. In this section we examine this operation more
systematically. Then we introduce and study a converse operation that transforms
a g-tiling on $Z_{n-1}$ into a g-tiling on $Z_n$. The results obtained here
(which are interesting by its own right) will be essentially used in
Section~\ref{sec:wsf}.

Consider a g-tiling $T$ on $Z_n$. Let $P$ be the reversed path to the right
boundary $R_Q$ of the $n$-strip $Q$. It possesses a number of important features,
as follows:
  \begin{numitem1}
For the path $P=(v_0,e_1,v_1,\ldots,e_r,v_r)$ as above and the colors
$i_1,\ldots,i_r$ of its edges $e_1,\ldots,e_r$, respectively, the following
hold:
  \begin{itemize}
 \item[(i)]
$P$ begins at the minimal point $p_0$ of $Z_n$ and ends at $p_{n-1}$;
 \item[(ii)]
none of $v_0,\ldots,v_r$ is the top or bottom vertex of a black $ij$-tile with
$i,j<n$;
 \item[(iii)]
$P$ has no pair of consecutive backward edges;
 \item[(iv)]
if $e_q=(v_{q-1},v_q)$ and $e_{q+1}=(v_{q+1},v_q)$ (i.e., $e_q$ is forward and
$e_{q+1}$ is backward in $P$), then $i_q>i_{q+1}$;
 \item[(v)]
if $e_q=(v_q,v_{q-1})$ and $e_{q+1}=(v_q,v_{q+1})$ (i.e., $e_q$ is backward and
$e_{q+1}$ is forward in $P$), then $i_q<i_{q+1}$.
  \end{itemize}
   \label{eq:legal}
   \end{numitem1}

Indeed, the first and last edges of $Q$ are $p_{n-1}p_n$ and $p_0p'_{n-1}$,
yielding~(i). Property~(ii) follows from the facts that each vertex $v_q$ has
an incident $n$-edge (which belongs to $Q$) and that all edges incident to the
top or bottom vertex of a black $ij$-tile have colors between $i$ and $j$ (see
Corollary~\ref{cor:termv}(iii)). The forward (backward) edges of $P$ are the
backward (resp. forward) edges of $R_Q$. Therefore, each forward (backward)
edge $e_q$ of $P$ belongs to a white (resp. black) $i_q n$-tile, taking into
account the maximality of color $n$; cf.~\refeq{strip_boundary}. Then for any
two consecutive edges $e_q,e_{q+1}$, at least one of them is forward,
yielding~(iii) (for otherwise the vertex $v_q$ is terminal and has both
entering and leaving edges, which is impossible). Next, let $\tau$ be the $i_q
n$-tile (in $Q$) containing $e_q$, and $\tau'$ the $i_{q+1} n$-tile containing
$e_{q+1}$. If $e_q$ is forward and $e_{q+1}$ is backward in $P$, then $\tau'$
is black, $v_q$ is the left vertex of $\tau'$, and the $i_q$-edge $e$ opposite
to $e_q$ in $\tau$ enters the top vertex of $\tau'$. Since $e$ lies in the cone
of $\tau'$ at $t(\tau')$, we have $i_{q+1}<i_q<n$, as required is~(iv). And if
$e_q$ is backward and $e_{q+1}$ is forward, then $\tau$ is black and $v_q$ is
its bottom vertex. Since $e_{q+1}$ lies in the cone of $\tau$ at $b(\tau)$, we
have $i_q<i_{q+1}<n$, as required in~(v). \medskip

Recall that the $n$-contraction operation applied to $T$ shrinks the $n$-strip
in such a way that $R_Q$ merge with the left boundary $L_Q$ of $Q$.
From~\refeq{legal}(ii) it follows that in the resulting g-tiling $T'$ on
$Z_{n-1}$, the path $P$ as above no longer contains terminal vertices at all.

Next we describe the converse operation that transforms a pair consisting of an
arbitrary g-tiling $T'$ on $Z_{n-1}$ and a certain path in $G_{T'}$ into a
g-tiling on $Z_n$. To explain the construction, we first consider an arbitrary
simple path $P$ in $G_{T'}$ which begins at $p_0$, ends at the maximal point
$p_{n-1}$ of $Z_{n-1}$, and may contain backward edges. Since the graph $G_{T'}$
has a planar layout (without intersection of non-adjacent edges) in the disc
$D_{T'}$), the path $P$ subdivides $G_{T'}$ into two connected subgraphs
$G'=G'_P$ and $G''=G''_P$ such that: $G'\cup G''=G_{T'}$, ~$G'\cap G''=P$, ~$G'$
contains $\ell bd(Z_{n-1})$, and $G''$ contains $rbd(Z_{n-1})$; we call $G'$
~($G''$) the {\em left} (resp. {\em right}) {\em subgraph} w.r.t. $P$. Then each
tile of $T'$ becomes a face of exactly one of $G',G''$ (and all inner faces of
$G',G''$ are such), and for an edge $e$ of $P$ not in $bd(Z_{n-1})$, the two
tiles sharing $e$ occur in different subgraphs. So $T'$ is partitioned into two
subsets, one being the set of faces of $G'$, and the other of $G''$.

The $n$-{\em expansion operation} for $(T',P)$ disconnects $G',G''$ by cutting
$G_{T'}$ along $P$ and then glue them by adding the corresponding $n$-strip. More
precisely, we shift the vertices of $G''$ by the vector $\xi_n$, i.e., each
vertex $X$ in it changes to $Xn$; this induces the corresponding shift of edges
and tiles in $G''$. The vertices of $G'$ preserve. So each vertex $X$ occurring
in the path $P$ produces two vertices, namely, $X$ and $Xn$. As a result, for
each edge $e=(X,Xi)$ of $P$, there appears its copy $\tilde e=(Xn,Xin)$ in the
shifted $G''$; we connect $e$ and $\tilde e$ by the corresponding (new)
$in$-tile, namely, by $\tau(X;i,n)$. This added tile is colored white if $e$ is a
forward edge of $P$, and black if $e$ is backward. The colors of all old tiles
preserve.

We refer to the resulting set $T$ of tiles, with the partition into white and
black ones, as the $n$-{\em expansion of $T'$ along $P$}. Since the right
boundary of the shifted $G''$ becomes the part of $rbd(Z_n)$ from the point
$p'_{n-1}$ (=$\{n\}$) to $p_n$ (=$[n]$), it follows that the union of the tiles
in $T$ is $Z_n$. Also it is easy to see that the shape $D_T$ in $conv(2^{[n]})$
associated to $T$ is again a disc (as required in axiom~(T4)), and that $T$
obeys axiom~(T1). The path $P$ generates the $n$-strip $Q$ for $T$ (consisting
of the added $\ast n$-tiles and the edges of the form $(X,Xn)$), and we observe
that $R_Q=P^{rev}$ and that $L_Q$ is the shift of $P^{rev}$ by $\xi_n$.
Therefore, the $n$-contraction operation applied to $T$ returns $T'$.

To ensure validity of the remaining axioms~(T2) and~(T3), we have to impose
additional conditions on the path $P$. In fact, they are similar to those
exposed in~\refeq{legal}. Moreover, these conditions are necessary and
sufficient.
   \begin{lemma} \label{lm:leg-exp}
Let $P=(p_0=v_0,e_1,v_1,\ldots,e_r,v_r=p_{n-1})$ be a simple path in $G_{T'}$.
Then the following are equivalent:

(i) the $n$-expansion $T$ of ~$T'$ along $P$ is a (feasible) g-tiling on $Z_n$;

(ii) $P$ contains no terminal vertices for $T'$ and
satisfies~\refeq{legal}(iii),(iv),(v).
   \end{lemma}
   \begin{proof}
Let $P$ be as in~(ii). We have to verify axioms~(T2),(T3) for $T$. Let
$P'=(v'_0,e'_1,v'_1,\ldots,e'_r,v'_r)$ and
$P''=(v''_0,e''_1,v''_1,\ldots,e''_r,v''_r)$ be the copies of $P$ in the graphs
$G'$ and $G''$ (taken apart), respectively. It suffices to check conditions
in~(T2),(T3) for objects involving elements of $P',P''$ (since for any vertex of
$G'$ not in $P'$, the structure of its incident edges and tiles, as well as the
white/black coloring of tiles, is inherited from $G_{T'}$, and similarly for
$G''$).

Consider a vertex $v_q$ with $1\le q<r$. Let $E_q^L$ ($E_q^R$) denote the set
of edges in $E_{T'}(v_q)$ lying on the left (resp. right) when we move along
$P$ and pass through $e_q,v_q,e_{q+1}$; we include $e_q,e_{q+1}$ in both
$E_q^L$ and $E_q^R$. Let $F_q^L$ ($F_q^R$) denote the set of tiles in
$F_{T'}(v_q)$ of which both edges incident to $v_q$ belong to $E_q^L$ (resp.
$E_q^R$). Note that each tile $\tau\in F_{T'}(v_q)$ must occur in either
$F_q^L$ or $F_q^R$, i.e., $\tau$ is not separated by $e_q$ or $e_{q+1}$ (taking
into account that all vertices of $P$ are non-terminal, and therefore the edges
of $P$ are white, and considering the behavior of edges and tiles at a
non-terminal vertex exhibited in~\refeq{mixedvert}). By the construction of
$G',G''$, any two tiles of $T'$ that share an edge not in $P$ are faces of the
same graph among $G',G''$, and if a tile $\tau\in T'$ has an edge contained in
$\ell bd(Z_{n-1})-P$ ~(resp. $rbd(Z_{n-1})-P$), then $\tau$ is a face of $G'$
(resp. $G''$). Using these observations, one can conclude that
   \begin{numitem1}
for $1\le q<r$, $E_q^L$ and $F_q^L$ are entirely contained in $G'$, while
$E_q^R$ and $F_q^R$ are entirely contained in $G''$.
   \label{eq:entirely}
   \end{numitem1}

For $q=1,\ldots,r$, let $\tau_q^L$ ($\tau_q^R$) denote the tile in $T'$ (if
exists) that contains the edge $e_q$ and lies on the left (resp. right) when we
traverse $e_q$ from $v_{q-1}$ to $v_q$. By~\refeq{entirely}, $\tau_q^L$ is in
$G'$ and $\tau_q^R$ is in $G''$. Also each of $\tau_q^L,\tau_q^R$ is white. Let
$\tau_q$ be the $i_qn$-tile in $T$ that was added to connect the edges $e'_q$
and $e''_q$. Then
  \begin{equation}  \label{eq:epq-eppq}
  e'_q=(b(\tau_q),\ell(\tau_q))\quad \mbox{and}\quad
  e''_q=(r(\tau_q),t(\tau_q)).
  \end{equation}

Suppose that $e_q$ is forward in $P$. Then $\tau_q$ is white. Since $e_q$ is
directed from $v_{q-1}$ to $v_q$ and $\tau_q^L$ lies on the left from $e_q$ when
moving from $v_{q-1}$ to $v_q$, ~$e_q$ belongs to the right boundary of
$\tau_q^L$. This and~\refeq{epq-eppq} imply that $\tau_q$ and $\tau_q^L$ do not
overlap. In its turn, $\tau_q^R$ contains $e_q$ in its left boundary; this
together with~\refeq{epq-eppq} implies that $\tau_q$ and the shifted $\tau_q^R$
(sharing the edge $e''_q$) do not overlap as well. Now suppose that $e_q$ is
backward in $P$. Then $\tau_q$ is black. Since $e_q$ is directed from $v_q$ to
$v_{q-1}$ and $\tau_q^L$ lies on the left from $e_q$ when moving from $v_{q-1}$
to $v_q$, ~$e_q$ belongs to the left boundary of $\tau_q^L$. This implies that
$\tau_q$ and $\tau_q^L$ overlap. Similarly, $\tau_q$ and $\tau_q^R$ overlap.
Thus,~(T2) holds for $\tau_q,\tau_q^L$ and for $\tau_q,\tau_q^R$, as required.
Also the non-existence of pairs of consecutive reverse edges in $P$ implies that
no two black tiles in $T$ share an edge.

To verify~(T3), consider a black tile $\tau_q$. Then $1<q<r$, the edges
$e_{q-1},e_{q+1}$ are forward, and $e_q$ is backward in $P$. Also
$i_{q-1},i_{q+1}>i_q$ (by~\refeq{legal}(iv),(v)). Observe that the set
$E_{q-1}^R$ consists of the edges in $E_{T'}(v_{q-1})$ that enter $v_{q-1}$ and
have color $j$ such that $i_q\le j\le i_{q-1}$ (including $e_{q-1},e_q$). All
these edges are white (as is seen from~\refeq{mixedvert}). The second copies of
these edges (shifted by $\xi_n$) plus the $n$-edge $(v'_{q-1},v''_{q-1})$ are
exactly those edges of $G_T$ that are incident to the top vertex $v''_{q-1}$ of
$\tau_q$. It its turn, the set $E_q^L$ consists of the edges in $E_{T'}(v_q)$
that leave $v_q$ and have color $j$ such that $i_q\le j\le i_{q+1}$, and these
edges are white. Exactly these edges plus the $n$-edge $(v'_q,v''_q)$ form the
set of edges of $G_T$ incident to $b(\tau_q)$. (See the picture.) This
gives~(T3) for $T$.

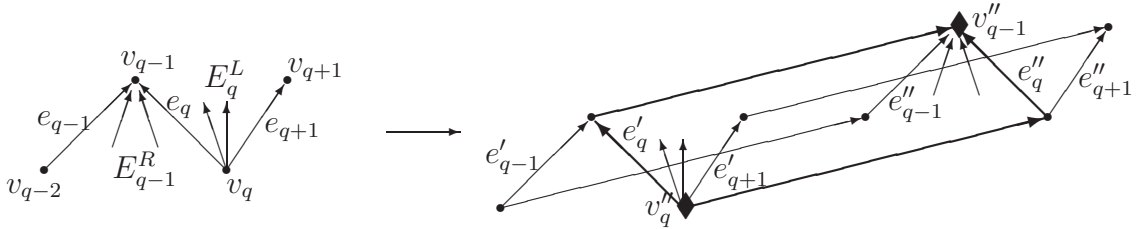
 \begin{center}
  \unitlength=1mm
  \begin{picture}(145,25)
   \put(5,0){\begin{picture}(60,20)
  \put(0,5){\circle*{1}}
  \put(12,17){\circle*{1}}
  \put(24,5){\circle*{1}}
  \put(32,17){\circle*{1}}
  \put(0,5){\vector(1,1){11.5}}
  \put(24,5){\vector(-1,1){11.5}}
  \put(24,5){\vector(2,3){7.6}}
  \put(9,8){\vector(1,3){2.5}}
  \put(15,8){\vector(-1,3){2.5}}
  \put(24,5){\vector(-1,3){3}}
  \put(24,5){\vector(0,1){9}}
  \put(-5,2){$v_{q-2}$}
  \put(10,19){$v_{q-1}$}
  \put(24,2){$v_q$}
  \put(32,18){$v_{q+1}$}
  \put(-1,11){$e_{q-1}$}
  \put(16,13){$e_q$}
  \put(29,10){$e_{q+1}$}
  \put(9,4){$E_{q-1}^R$}
  \put(21,16){$E_q^L$}
  \put(45,10){\vector(1,0){9.7}}
    \end{picture}}
%
   \put(65,0){\begin{picture}(80,25)
  \put(0,0){\circle*{1}}
  \put(12,12){\circle*{1}}
  \put(24,0){\circle*{1}}
  \put(32,12){\circle*{1}}
  \put(48,12){\circle*{1}}
  \put(60,24){\circle*{1}}
  \put(72,12){\circle*{1}}
  \put(80,24){\circle*{1}}
  \put(0,0){\vector(1,1){11.5}}
  \put(24,0){\vector(2,3){7.6}}
  \put(48,12){\vector(1,1){11}}
  \put(72,12){\vector(2,3){7.6}}
  \put(0,0){\vector(4,1){47.5}}
  \put(32,12){\vector(4,1){47.5}}
  \put(57,15){\vector(1,3){2.5}}
  \put(63,15){\vector(-1,3){2.5}}
  \put(24,0){\vector(-1,3){3}}
  \put(24,0){\vector(0,1){9}}

\thicklines{
  \put(24,0){\vector(-1,1){11.5}}
  \put(72,12){\vector(-1,1){11.5}}
  \put(12,12){\vector(4,1){47.5}}
  \put(24,0){\vector(4,1){47.5}}
 }
  \put(23,-1){$\blacklozenge$}
  \put(59,23){$\blacklozenge$}
  \put(-2,6){$e'_{q-1}$}
  \put(16,9){$e'_q$}
  \put(28,4){$e'_{q+1}$}
  \put(51,13){$e''_{q-1}$}
  \put(68,17){$e''_q$}
  \put(76,16){$e''_{q+1}$}
  \put(19,-1){$v''_q$}
  \put(62,24){$v''_{q-1}$}
    \end{picture}}
  \end{picture}
   \end{center}

Thus, (ii) implies~(i) in the lemma. The converse implication (i)$\to$(ii)
follows from~\refeq{legal} and the fact (mentioned earlier) that for the
$n$-expansion $T$ of $T'$ along $P$, the $n$-contraction operation applied to
$T$ produces $T'$, and under this operation the $n$-strip for $T$ shrinks into
$P^{rev}$. This completes the proof of the lemma.
  \end{proof}

Let us call a path $P$ as in (ii) of Lemma~\ref{lm:leg-exp} {\em legal}. It is
the concatenation of $P_1,\ldots,P_{n-1}$, where $P_h$ is the maximal subpath of
$P$ whose edges connect levels $h-1$ and $h$, i.e., are of the form $(X,Xi)$ with
$|X|=h-1$. We refer to $P_h$ as $h$-th {\em subpath} of $P$ and say that this
subpath is {\em ordinary} if it has only one edge, and {\em zigzag} otherwise.
The beginning vertices of these subpaths together with $h_{n-1}$ are called {\em
critical} in $P$ (so there is exactly one critical vertex in each level); these
vertices will play an important role in what follows. Note that the critical
vertices of a legal path $P=(v_0,e_1,v_1,\ldots, e_r,v_r)$ are $v_0=p_0$,
~$v_r=p_{n-1}$ and the intermediate vertices $v_q$ such that $e_q$ enters and
$e_{q+1}$ leaves $v_q$. We distinguish between two sorts of non-critical vertices
$v_q$ by saying that $v_q$ is a $\vee$-{\em vertex} if both $e_q,e_{q+1}$ leave
$v_q$, and a $\wedge$-{\em vertex} if both $e_q,e_{q+1}$ enter $v_q$. Observe
that
   \begin{numitem1}
regarding a vertex of $P$ as a subset $X$ of $[n-1]$, the following hold:
(a)~if $X$ is critical, then both $X,Xn$ are in $B_T$; (b)~if $X$ is a
$\wedge$-vertex, then $X\in B_T$ and $Xn\not\in B_T$; and (c)~if $X$ is a
$\vee$-vertex, then $X\not\in B_T$ and $Xn\in B_T$ (where $T$ is the
$n$-expansion of $T'$ along $P$).
   \label{eq:c-w-v}
   \end{numitem1}
Indeed, from the proof of Lemma~\ref{lm:leg-exp} one can see that: if $X$ is
critical, then both vertices $X,Xn$ of $G_T$ have entering and leaving edges, so
they are non-terminal; if $X$ is a $\wedge$-vertex, then $Xn$ is terminal while
$X$ is not; and if $X$ is a $\vee$-vertex, then $X$ is terminal while $Xn$ is not
(see the above picture).

It follows that
  \begin{numitem1}
(i) $B_T=B'\cup B''$, where $B'$ consists of all non-terminal vertices $X$ in
$G'_P$ that are not $\vee$-vertices in $P$, and $B''$ consists of all $Xn$ such
that $X$ is a non-terminal vertex in $G''_P$ that is not a $\wedge$-vertex of
$P$;

(ii) for each $h=0,\ldots,n-1$, there is exactly one set $X\subseteq[n-1]$ with
$|X|=h$ such that both $X$ and $Xn$ belong to $B_T$; moreover, this $X$ is just
the unique critical vertex of $P$ in level $h$.
  \label{eq:decompBT}
  \end{numitem1}

Summing up the above results, we can conclude with the following
  \begin{corollary}  \label{cor:TP-T}
The correspondence $(T',P)\mapsto T$, where $T'$ is a g-tiling on $Z_{n-1}$,
~$P$ is a legal path for $T'$, and $T$ is the $n$-expansion of $T'$ along $P$,
gives a bijective map $\epsilon$ of the set of such pairs $(T',P)$ to the set
of g-tilings on $Z_n$. Moreover, for any g-tiling $T$ on $Z_n$, the pair
$\eps^{-1}(T)$ consists of the $n$-contraction $T'$ of $T$ and the legal path
for $T'$ corresponding to the $n$-strip in $T$.
  \end{corollary}

We conclude this section with an additional result which will be important for
purposes of the next section. For a g-tiling $T$ on $Z_n$ and for $1\le h\le
n$, let $H_h$ denote the subgraph of $G_T$ induced by the set of {\em white}
edges connecting levels $h-1$ and $h$.
   \begin{lemma} \label{lm:Hh}
For each $h=1,\ldots,n$, the graph $H_h$ is a forest. Furthermore:

{\rm (i)} there exists a (connected) component $K$ of $H_h$ that contains both
boundary edges $p_{h-1}p_h$ and $p'_hp'_{h-1}$ and such that all vertices of $K$
are non-terminal;

{\rm (ii)} any other component $K'$ of $H_h$ contains exactly one terminal vertex
$v$ and all edges of $K'$ are incident to $v$ (i.e., $K'$ is a star).
   \end{lemma}
  \begin{proof} ~We observe that
  \begin{numitem1}
any 3-edge path $(v_0,e_1,v_1,e_2,v_2,e_3,v_3)$ in $H_h$ satisfies either
$i_1,i_3<i_2$ or $i_1,i_3>i_2$, where $i_1,i_2,i_3$ are the colors of
$e_1,e_2,e_3$, respectively; equivalently: the edges $e_1$ and $e_3$ (regarding
as line-segments in $Z_n$) do not intersect.
   \label{eq:3-path}
   \end{numitem1}
Indeed, suppose this is not so. W.l.o.g., one may assume that $i_1<i_2<i_3$ and
that the edges $e_1,e_2$ leave $v_1$. Then $v_1$ is in level $h-1$, ~$v_2$ is
in level $h$, and $e_3$ enters $v_2$. Take the edge $e$ leaving $v_1$ whose
color $j$ is maximum subject to $i_1\le j<i_2$, and the edge $e'$ entering
$v_2$ whose color $j'$ is minimum subject to $i_2<j'\le i_3$.
By~\refeq{mixedvert} (applied to $v_1$ and to $v_2$), both edges $e,e'$ are
white,  the edges $e,e_2$ belong to a white tile $\tau$ with the bottom vertex
$v_1$, and $e',e_2$ belong to a white tile $\tau'$ with the top vertex $v_2$.
But such tiles $\tau,\tau'$ (sharing the edge $e_2$) overlap, contrary to~(T2).

In view of~\refeq{3-path}, no path in $H_h$ can be cyclic (it monotonically goes
in one direction, either from left to right or from right to left). Hence $H_h$
is a forest in which any component $K$ (a tree) has a planar layout on $Z_n$,
i.e., non-adjacent edges in $K$ do not intersect. Suppose that $H_h$ contains a
terminal vertex $X$ in level $h-1$, and consider an edge in $H_h$ incident to
$X$, say, $e=(X,Xi)$ (taking into account that $e$ goes from level $h-1$ to level
$h$). Since $X$ is terminal, each of the two white tiles $\tau',\tau''$
containing $e$ has the bottom vertex at $X$ and the right of left vertex at $Xi$.
By~\refeq{3edges}, there is no white edge incident to $Xi$ and lying strictly
inside the cone $C(\tau',Xi)$, and similarly for $\tau''$. This implies that $e$
is the unique edge of $H_h$ incident to $Xi$. Hence the component of $H_h$
containing $X$ is a star of which all edges are incident to $X$. A similar
property holds for the components of $H_h$ meeting a terminal vertex in level
$h$.

Finally, consider a component $K$ without terminal vertices (it exists since the
boundary edge $p_{h-1}p_h$ is white and both of its ends are non-terminal). We
assert that $K$ contains $p_{h-1}p_h$. Indeed, take the leftmost edge in $K$,
say, $e=(X,Xi)$, and suppose that $e\ne p_{h-1}p_h$. Then there is a white tile
$\tau$ containing $e$ on its right boundary. Assume that $b(\tau)=X$; then
$r(\tau)=Xi$. Then the edge $e':=(b(\tau),\ell(\tau))$ is black (as $e'$ connects
levels $h-1$ and $h$ and lies on the left from $e$). Since $\ell(\tau)$ has both
entering and leaving edges, it cannot be terminal. So $b(\tau)$ is terminal,
contradicting the choice of $K$. The case $t(\tau)=Xi$ leads to a similar
contradiction. Thus, $K$ contains $p_{h-1}p_h$. Considering the rightmost edge of
$K$ and arguing similarly, we conclude that $K$ contains the boundary edge
$p'_hp'_{h-1}$ as well.
  \end{proof}

We will refer to the component $K$ as in~(i) of this lemma as the {\em principal}
one. Considering a legal path $P$ for $T$ and taking into account that all
vertices of $P$ are non-terminal and that $h$-th subpath in it is contained in
$H_h$, for each $h$, we obtain the following property as a consequence of
Lemma~\ref{lm:Hh}.
  \begin{corollary}  \label{cor:critic}
Any legal path for a g-tiling is determined by the set of its critical
vertices.
   \end{corollary}

Finally, let $LP_T$ denote the subgraph of $G_T$ that is the union of all legal
paths for $T$. One can construct $LP_T$ as follows. At the beginning, put $G'$
to be the union of principal components in $H_h$, $h=1,\ldots,n$. If $G'$
contains a vertex $v$ having exactly one incident edge in $G'$, then we remove
this vertex and edge from $G'$, repeat the procedure for the new $G'$, and so
on until no such $v$ exists. Then the final $G'$ is just $LP_T$. This graph
gives a nice compact representation for the set of legal paths for $T$. (It
contains every legal path for $T$. Conversely, starting from $p_0$ and forming
a maximal path in $LP_T$ as the concatenation of ordinary or zigzag paths going
to the right, we always reach $p_n$, obtaining a legal path. Note that the same
graph $LP_T$ appears when we are interested in expansions w.r.t. the new
minimal color, say, 0; in this case one should consider zigzag paths in $LP_T$
going to the left.)


\section{\Large Weakly separated set-systems}  \label{sec:wsf}

The goal of this section is to prove the following theorem answering
Leclerc--Zelevinsky's conjecture mentioned in the Introduction.
 \begin{theorem} \label{tm:LZ}
Any largest weakly separated collection $C\subseteq 2^{[n]}$ is a semi-normal
TP-basis.
 \end{theorem}

For brevity we will abbreviate ``weakly separated collection'' as
``ws-collection''. Recall that a ws-collection $C\subseteq 2^{[n]}$ is {\em
largest} if its cardinality $|C|$ is maximum among all ws-collections in
$2^{[n]}$; this maximum is equal to $\binom{n+1}{2}+1$~\cite{LZ}. An important
example is the set $\Iscr_n$ of intervals in $[n]$ (including the empty set).
Also it was shown in~\cite{LZ} that a (lowering or raising) flip in a
ws-collection produces again a ws-collection. Moreover, its cardinality
preserves under a flip since it replaces one set in some pair $\{Xj,Xik\}$
(say) by the other. Due to these facts, the set $\Cscr_n$ of largest
ws-collections includes $\Bscr_n$ (the set of semi-normal bases for
$\Tscr\Pscr_n$). Theorem~\ref{tm:LZ} says that the converse inclusion takes
place as well. As a result, we will conclude with the following
   \begin{corollary} \label{cor:C=B}
$\Cscr_n=\Bscr_n$.
  \end{corollary}

In view of Theorem~\ref{tm:bas-til}, to obtain Theorem~\ref{tm:LZ}, it suffices
to show the following
  \begin{theorem} \label{tm:ws-til}
Any $C\in\Cscr_n$ is the spectrum $B_T$ of some g-tiling $T$ on $Z_n$.
  \end{theorem}

This theorem is proved by combining additional facts established in~\cite{LZ}
and results from the previous sections. Let $C\in\Cscr_n$. To construct the
desired tiling for $C$, we consider the projection $C'$ of $C$ into
$2^{[n-1]}$, i.e., the collection of subsets $X\subseteq[n-1]$ such that either
$X\in C$ or $Xn\in C$ or both. Partition $C'$ into three subcollections
$M,N,S$, where
  $$
  M:=\{X\colon X\in C\not\ni Xn\},\;\;
  N:=\{X\colon Xn\in C\not\ni X\},\;\;
  S:=\{X\colon X,Xn\in C\}.
  $$
Also for $h=0,\ldots,n-1$, define
  $$
  C'_h:=\{X\in C'\colon |X|=h\},\;\; M_h:=M\cap C'_h,\;\;  N_h:=N\cap C'_h.
  $$
It is shown in~\cite{LZ} that
  \begin{numitem1}
for each $h=0,\ldots,n-1$, ~$S\cap C'_h$ contains exactly one element.
  \label{eq:separ}
  \end{numitem1}
We call $S$ the {\em separator} of $C'$ and denote its elements by
$S_0,\ldots,S_{n-1}$, where $|S_h|=h$. Property~\refeq{separ} implies that
$|C'|=|C|-|S|=\binom{n+1}{2}+1-n=\binom{n}{2}+1$, and as is shown in~\cite{LZ},
  \begin{numitem1}
$C'$ is a ws-collection, and therefore it is a {\em largest} ws-collection in
$2^{[n-1]}$.
  \label{eq:Cq}
   \end{numitem1}
Two more observations in~\cite{LZ} are:
  \begin{numitem1}
(i) $S_0\prec S_1\prec\cdots\prec S_{n-1}$;

(ii) for each $h=0,\ldots,n-1$, any sets $Y\in M_h$ and $Y'\in N_h$ satisfy
$Y\prec S_h$ and $\hphantom{(ii)}$ $S_h\prec Y'$.
  \label{eq:prec-sep}
  \end{numitem1}

An important consequence of~\refeq{prec-sep} is that the collection $C$ can be
uniquely restored from the pair $C',S$. Indeed, $C$ consists of the sets $X,Xn$
such that $X\in S$, the sets $X\subseteq[n-1]$ such that $X\prec S_{|X|}$, and
the sets $Xn$ such that $X\subseteq[n-1]$ and $X\succ S_{|X|}$.

The proof of Theorem~\ref{tm:ws-til} is led by induction on $n$. The result is
trivial for $n\le 2$. Let $n>2$ and assume by induction that there is a
g-tiling $T'$ on $Z_{n-1}$ such that $B_{T'}=C'$. Our aim is to transform $T'$
into a g-tiling on $Z_n$ whose spectrum is $C$. The crucial claim is the
following
  \begin{lemma} \label{lm:MNS}
There exists a legal path $P$ for $T'$ whose set of critical vertices coincides
with the separator $S$.
   \end{lemma}
 \begin{proof}~
It uses the following fact from~\cite{LZ}:
  \begin{numitem1}
if sets $A,A',A''\subseteq[n']$ are weakly separated, and if $|A|\le |A'|\le
|A''|$, ~$A\prec A'$ and $A'\prec A''$, then $A\prec A''$.
   \label{eq:AAA}
   \end{numitem1}

The desired path $P$ is constructed by relying on Lemma~\ref{lm:Hh}. For
$h=1,\ldots,n-1$, let $H_h$ be the subgraph of $G_{T'}$ induced by the white
edges connecting levels $h-1$ and $h$, and let $K_h$ be the principal component
of $H_h$. By Lemma~\ref{lm:Hh}, all vertices $X$ of $K_h$ are non-terminal,
whence $X\in C'$.

Consider two vertices $X,Y$ with $|X|\le|Y|$ in $K_h$. If they are connected by
edge, then $Y=Xi$ for some $i\in[n]$ and, obviously, $X\prec Y$. If the path
$P$ from $X$ to $Y$ in $K_h$ is such that its length (number of edges) $|P|$ is
at least two and $P$ goes from left to right, then $X\prec Y$ as well. (When
$X,Y$ belong to different levels, we say that $P$ goes from left to right if
the vertex $Y'$ preceding $Y$ in $P$ lies on the right from $X$ in the level
containing $X,Y'$.) Indeed, for any three consecutive vertices $Z,Z',Z''$ in
$P$ (occurring in this order) either $Z=Z'i$ and $Z''=Z'j$, or $Z=Z'-\{j\}$ and
$Z''=Z-\{i\}$ for some $i<j$ (since $Z''$ lies on the right from $Z$), which
implies $Z\prec Z''$. Then $X\prec Y$ follows from~\refeq{AAA} by the
transitivity.

We assert that the separating vertex $S_{h-1}$ belongs to $K_h$. Indeed,
suppose this is not so. Then $S_{h-1}$ belongs to a star component $K'$ of
$H_h$, and therefore, the white edge $e$ in $H_h$ leaving $S_{h-1}$ enters the
top vertex of some black tile $\tau=\tau(X;p,q)$. We have $S_{h-1}=Xpq-\{i\}$
and $p<i<q$, where $i$ is the color of $e$. On the other hand, the bottom
vertex $X$ of $\tau$ has a leaving $j$-edge $(X,Xj)$ for some $p<j<q$ (cf.
Corollary~\ref{cor:termv}). The vertex $Xj$ is non-terminal, and we have
$|Xj|=|S_{h-1}|$, ~$S_{h-1}-Xj=\{p,q\}$ and $Xj-S_{h-1}=\{i,j\}$. Since
$p<i,j<q$, we come to a contradiction with the fact that $S_{h-1}$ is
comparable by $\prec$ with any non-terminal vertex in level $h-1$. Thus,
$S_{h-1}$ is in $K_h$. Arguing similarly, one shows that $S_h$ is in $K_h$ as
well. Let $P_h$ be the path from $S_{h-1}$ to $S_h$ in $K_h$. Since
$S_{h-1}\prec S_h$ (by~\refeq{prec-sep}), ~$P_h$ goes from left to right (when
$|P_h|>1$).

Concatenating $P_1,\ldots,P_{n-1}$, we obtain a legal path $P$ for $T'$ of
which $h$-th subpath is $P_h$, and the critical vertices are
$S_0,\ldots,S_{n-1}$, as required.
  \end{proof}

Now we finish the proof of Theorem~\ref{tm:ws-til} as follows. Let $T$ be the
$n$-expansion of $T'$ along $P$. Then $B_T$ is a ws-collection, moreover, it is
a largest ws-collection since $|B_T|=\binom{n+1}{2}+1$. By~\refeq{decompBT} and
Corollary~\ref{cor:TP-T}, the projection of $B_T$ into $2^{[n-1]}$ (defined by
$X\mapsto X-\{n\}$) is just $B_{T'}=C'$ and, moreover, the set of $X\subseteq
[n-1]$ such that $X,Xn\in B_T$ is exactly the set of critical vertices in $P$,
i.e., $S$. Since the pair $C',S$ generates the corresponding largest
ws-collection in $2^{[n]}$ in a unique way, we obtain $B_T=C$, and
Theorem~\ref{tm:ws-til} follows.    \hfill \qed\qed

This completes the proof of Theorem~\ref{tm:LZ}.


\section{\Large Generalizations}  \label{sec:gen}

In this concluding section we outline two generalizations, omitting proofs.
They will be discussed in full, with details and related topics, in a separate
paper.
\smallskip

 {\bf A.}
The obtained relationships between semi-normal bases, proper wirings and
generalized tilings are extendable to the case of an integer $n$-box ${\bf
B}^{n,a}=\{x\in \Zset^{[n]}\colon 0\le x\le a\}$, where $a\in\Zset_+^n$. Recall
that a function $f$ on ${\bf B}^{n,a}$ is a TP-function if it satisfies
  \begin{eqnarray}
 &&f(x+1_i+1_k)+f(x+1_j)   \label{eq:boxTP} \\
 &&\qquad   =\max\{f(x+1_i+1_j)+f(x+1_k),\; f(x+1_i)+f(x+1_j+1_k)\}
       \nonumber
  \end{eqnarray}
for any $x$ and $1\le i<j<k\le n$, provided that all six vectors occurring as
arguments in this relation belong to ${\bf B}^{n,a}$, where $1_q$ denotes
$q$-th unit base vector. In this case the standard basis of the TP-functions
consists of the vectors $x$ such that $x_i,x_j>0$ for $i<j$ implies $x_q=a_q$
for $q=i+1,\ldots,j-1$ (see~\cite{DKK-08}, where such vectors are called {\em
fuzzy-intervals}). Normal and semi-normal bases are corresponding collections
of integer vectors in ${\bf B}^{n,a}$, defined by a direct analogy with the
Boolean case.

The semi-normal bases in the box case admit representations via natural
generalizations of proper wiring and g-tiling diagrams for the Boolean case.
They are described as follows.

The zonogon for a given $a$ is the set $Z_{n,a}:=\{\lambda_1\xi_1+\ldots+
\lambda_n\xi_n\colon \lambda_i\in\Rset,\; 0\le\lambda_i\le a_i,\;
i=1,\ldots,n\}$, where the vectors $\xi_i$ are chosen as above. For each
$i\in[n]$ and $q=0,1,\ldots,a_i$, define the point
$p_{i,q}:=a_1\xi_1+\ldots+a_{i-1}\xi_{i-1}+q\xi_i$ (on the left boundary of
$Z_{n,a}$) and the point $p'_{i,q}:=a_n\xi_n+\ldots a_{i+1}\xi_{i+1}+q\xi_i$
(on the right boundary). These points are regarded as the vertices on the
boundary of $Z_{n,a}$, and the edges in it are the directed line-segments
$p_{i,q-1}p_{i,q}$ and $p'_{i,q}p'_{i,q-1}$ When $q\ge 1$, we define $s_{i,q}$
($s'_{i,q}$) to be the median point on the edge $p_{i,q-1}p_{i,q}$ (resp.
$p'_{i,q}p'_{i,q-1}$).

A generalized tiling $T$ on $Z=Z_{n,a}$ is defined by essentially the same
axioms~(T1)--(T4) from Subsection~\ref{ssec:tiling}. A wiring $W$ on $Z$
consists of wires $w_{i,q}$ going from $s_{i,q}$ to $s'_{i,q}$, $i=1,\ldots,n$,
~$q=1,\ldots,a_i$. Again, it is defined by the same axioms~(W1)--(W3) from
Subsection~\ref{ssec:wiring}.

Note that for any $i$ and $1\le q<q'\le a_i$, the point $s'_{i,q}$ occurs
earlier than $s'_{i,q'}$ in the right boundary of $Z$ (beginning at $p_{0}$),
which corresponds to the order of $s_{i,q},s_{i,q'}$ in the left boundary of
$Z$. This and axiom~(W2) imply that the wires $w:=w_{i,q}$ and $w':=w_{i,q'}$
are always disjoint. Indeed, suppose that $w$ and $w'$ meet and take the first
point $x$ of $w'$ that belongs to $w$. Let $\Omega_0,\Omega_1$ be the connected
components of $Z-(P\cup P')$, where $P$ is the part of $w$ from $x$ to
$s'_{i,q}$, ~$P'$ is the part of $w'$ from $s_{i,q'}$ to $x$, and $\Omega_0$
contains $p_{0}$. Then the end point $s'_{i,q'}$ of $w'$ is in $\Omega_1$.
Furthermore, $w'$ crosses $w$ at $x$ from left to right (since $x$ is the first
point of $w'$ where it meets $w$); this implies that when passing $x$, the wire
$w'$ enters the region $\Omega_0$. Therefore, the part of $w'$ from $x$ to
$s'_{i,q'}$ must intersect $P\cup P'$ at some point $y\ne x$. But $y\in P$ is
impossible by~(W2) and $y\in P'$ is impossible because $w'$ is not
self-intersecting.

Like the Boolean case, for a g-tiling $T$, the spectrum $B_T$ is defined to be
the set of non-terminal vertices (viz. $n$-vectors) for $T$. For a wiring $W$
and an (inner) face $F$ of its associated planar graph, let $x(F)$ denote the
$n$-vector whose $i$-th entry is the number of wires $w_{i,q}$ such that $F$
lies on the left from $w_{i,q}$. Then $B_W$ is defined to be the collection of
vectors $x(F)$ over all non-cyclic faces $F$.

Theorems~\ref{tm:bas-wir} and~\ref{tm:bas-til} remain valid for these extended
settings (where $Z_n$ is replaced by $Z_{n,a}$), and proving methods are
essentially the same as those in
Sections~\ref{sec:bas-to-til}--\ref{sec:wir-to-til}, with minor refinements on
some steps. (E.g., instead of a unique dual $i$-path ($i$-strip) for each $i$,
we now deal with $a_i$ dual $i$-paths $Q_{i,1},\ldots,Q_{i,a_i}$, each
$Q_{i,q}$ connecting a boundary edge $p_{i,q-1}p_{i,q}$ to
$p'_{i,q-1}p'_{i,q}$, which does not cause additional difficulty in the proof.)
  \medskip

 {\bf B.}
The second generalization involves an arbitrary permutation $\omega$ on $[n]$.
(In fact, so far we have dealt with the longest permutation $\omega_0$, where
$\omega_0(i)=n+1-i$). For $i,j\in [n]$, we write $i\prec_\omega j$ if $i<j$ and
$\omega(i)<\omega(j)$. This relation is transitive and gives a partial order on
$[n]$. Let $\Xscr_\omega\subseteq 2^{[n]}$ be the set (lattice) of ideals $X$
of $([n],\prec_\omega)$, i.e., $i\prec_\omega j$ and $j\in X$ implies $i\in X$.
In particular, $\Xscr_\omega$ is closed under taking a union or intersection of
its members. Below we specify settings and outline how results concerning
$\omega_0$ can be extended to $\omega$.
\smallskip

(i) By a TP-function for $\omega$, or an $\omega$-{\em TP-function}, we mean a
function $f$ defined on the set $\Xscr_\omega$ (rather than $2^{[n]}$) and
satisfying~\refeq{P3} when all six sets in it belong to $\Xscr_\omega$. Note
that $Xi,Xk,Xij,Xjk\in\Xscr_\omega$ implies that each of $X,Xj,Xik,Xijk$ is in
$\Xscr_\omega$ as well (since each of the latter is obtained as the
intersection or union of a pair among the former). The notion of TP-basis is
extended to the set $\Tscr\Pscr_{\omega}$ of $\omega$-TP-functions in a natural
way. It turns out that the role of standard basis is now played by the set
$\Iscr_\omega$ of $\omega$-{\em dense} sets $X\in\Xscr_\omega$, which means
that there are no triples $i<j<k$ such that $i,k\in X\not\ni j$ and each of the
sets $X-\{i\}$, $X-\{k\}$ and $(X-\{i,k\})\cup\{j\}$ belongs to $\Xscr_\omega$.
In particular, $\Iscr_\omega$ contains the sets $[i]$, ~$\{i'\colon
i'\preceq_\omega i\}$ and $\{i'\colon \omega(i')\le\omega(i)\}$ for each
$i\in[n]$; ~when $\omega=\omega_0$, ~$\Iscr_\omega$ turns into the set
$\Iscr_n$ of intervals in $[n]$. (It is rather easy to prove that any
$\omega$-TP-function is determined by its values on $\Iscr_\omega$; this is
done by exactly the same method as applied in~\cite{DKK-08} to show a similar
fact for $\Tscr\Pscr_n$ and $\Iscr_n$. The fact that the restriction map
$\Tscr\Pscr_\omega\to \Rset^{\Iscr_\omega}$ is surjective (which is more
intricate) can be shown by extending a flow approach developed in~\cite{DKK-08}
for the cases of TP-functions on Boolean cubes and integer boxes.) Normal and
semi-normal bases for the $\omega$-TP-functions are defined via flips from the
standard basis $\Iscr_\omega$, by analogy with those for $\omega_0$.
\smallskip

(ii) Instead of the zonogon $Z_n$, we now should consider the region in the
plane bounded by two paths: the left boundary of $Z_n$ and the path $P_\omega$
formed by the points $p'_{0,\omega}:=p_0$ and
$p'_{i,\omega}:=\xi_{\omega^{-1}(1)}+\ldots+ \xi_{\omega^{-1}(i)}$ for
$i=1,\ldots,n$ ~connected by the (directed) line-segments
$e'_{1,\omega},\ldots,e'_{n,\omega}$, where $e'_{j,\omega}$ begins at
$p'_{\omega(j)-1,\omega}$ and ends at $p'_{\omega(j),\omega}$. (Then
$e'_{j,\omega}$ is a parallel translation of $\xi_j$. When $\omega=\omega_0$,
each $p'_{i,\omega}$ is just the point $p'_{n-i}$ on the right boundary of
$Z_n$. When $\omega$ is the identity, $p'_{i,\omega}$ coincides with the point
$p_i$ on $\ell bd(Z_n)$.) We denote this region as $Z_{\omega}$ and call it the
$\omega$-{\em deformation} of the zonogon $Z_n$, or, liberally, the
$\omega$-{\em zonogon}. A wiring for $\omega$ is a collection $W$ of wires
$w_1,\ldots,w_n$ in $Z_\omega$ satisfying axioms~(W1)--(W3) and such that each
$w_i$ begins at the point $s_i$ (as before) and ends at the median point
$s'_{i,\omega}$ of $e'_{i,\omega}$ (a wire $w_i$ degenerates into a point if
the boundary edges $p_{i-1}p_i$ and $e'_{j,\omega}$ coincide). Note that if
$i\prec_\omega j$ then $s'_{i,\omega}$ occurs earlier than $s'_{j,\omega}$ in
the right boundary $P_\omega$ of $Z_\omega$, and therefore, the wires $w_i$ and
$w_j$ does not meet (as explained in part~A above). This implies that all sets
in the full spectrum of $W$ belong to $\Xscr_\omega$.

In its turn, a generalized tiling $T$ for $\omega$ is defined in the same way
as for $\omega_0$, with the only differences that now the union of tiles in $T$
is $Z_\omega$ and that the corresponding shape $D_T$ is required to be simply
connected (then $D_T$ is homeomorphic to $Z_\omega$). (Depending on $\omega$,
points $p_{i,\omega}$ and $p'_{\omega(i),\omega}$ may coincide for some $i$, so
$D_T$ need not be a disc in general.) The constructions and arguments in
Sections~\ref{sec:til-to-wir},\ref{sec:wir-to-til}, based on planar duality,
can be transferred without essential changes to the $\omega$ case, giving a
natural one-to-one correspondence between the g-tilings and proper wirings for
$\omega$. (In particular, the fact that $i$-th wire $w_i$ in a proper wiring
$W$ for $\omega$ turns into the $i$-strip $Q_i$ in the corresponding g-tiling
$T$ (which begins with the $i$-edge $p_{i-1}p_i$ in $\ell bd(Z_\omega)$ and
ends with the $i$-edge $e'_{i,\omega}$ in $rbd(Z_\omega)$) implies that all
vertices of $G_T$ represent sets in $\Xscr_\omega$.) The arguments in
Sections~\ref{sec:bas-to-til},\ref{sec:til-to-bas} continue to work in the
$\omega$ case as well. As a result, we obtain direct generalizations of
Theorems~\ref{tm:bas-wir} and~\ref{tm:bas-til} to an arbitrary permutation
$\omega$.
\smallskip

(iii) A majority of results from Sections~\ref{sec:contr-exp} and~\ref{sec:wsf}
can be extended, with a due care, to the $\omega$ case as well. Below we give a
brief commentary, not coming into particular details. Let $\Cscr_\omega$ denote
the set of weakly separated collections $C\subseteq\Xscr_\omega$ whose
cardinality $|C|$ is maximum; let us call them {\em largest
$\omega$-ws-collections}, and denote the maximum $|C|$ by $c_\omega$. It is
shown in~\cite{LZ} that $c_\omega$ equals $\ell(\omega)+n+1$ (where
$\ell(\omega)$ is the minimum number of inversions for $\omega$); it follows
that
  $$
  c_\omega=c_{\omega'}+n-\omega(n)+1,
  $$
where $\omega'$ is the permutation on $[n-1]$ defined by
$\omega'(i):=\omega(i)-1$ if $\omega(i)>\omega(n)$, and $\omega'(i):=\omega(i)$
otherwise. Due to this, the projection $C'$ of $C\in\Cscr_\omega$ into
$2^{[n-1]}$ (defined by $X\mapsto X-\{n\}$) is a largest
$\omega'$-ws-collection, and the corresponding ``separator'' $S:=\{X\colon
X,Xn\in C\}$ consists of $n-\omega(n)+1$ sets $S_q$ of size $|S_q|=q$ for
$q=\omega(n)-1,\ldots,n-1$. The sets $M_q,N_q$ for $q=0,\ldots,n-1$ are defined
accordingly (in particular, all sets $X\in C$ with $|X|=q$ belong to $M_q$ when
$q<\omega(n)-1$).

Given a g-tiling $T'$ for $\omega'$, we define a legal path $P$ in $G_{T'}$ in
the same way as before (i.e., so as to satisfy conditions (iii),(iv),(v)
in~\refeq{legal} and the requirement that $P$ has no terminal vertices) with
the only difference that now $P$ should begin at the vertex
$p'_{\omega(n)-1,\omega}$ on the right boundary of the $\omega'$-zonogon.

The desired generalizations ($\omega$-analogs) of
Theorems~\ref{tm:LZ},~\ref{tm:ws-til} and Lemma~\ref{lm:MNS} are formulated in
a natural way, and their proofs remain essentially the same as before, relying
on the above definition of a legal path and corresponding direct extensions of
results on $n$-contractions and $n$-expansions to the $\omega$ case.

 \medskip
 \noindent{\bf Remark 7}~
In fact, the generalization in part~A is a special case of the one in part~B.
More precisely, given $a\in\Zset_+^n$, define $\bar a_i:=a_1+\ldots+a_i$,
$i=0,\ldots,n$ (letting $\bar a_0:=0$). Let us  form a permutation $\omega'$ on
$[\bar a_n]$ as follows: for $i=1,\ldots,n$ and $q=1,\ldots,a_i$,
  $$
  \omega'(\bar a_{i-1}+q):=\bar a_n-\bar a_{i}+q,
  $$
i.e., $\omega'$ permutes the blocks $B_1,\ldots,B_n$, where $B_i:=\{\bar
a_{i-1}+1,\ldots, \bar a_i\}$, according to the permutation $\omega_0$ on
$[n]$, and preserves the order of elements within each block. Then there is a
one-to-one correspondence between the vectors $x\in{\bf B}^{n,a}$ and the
ideals $X$ of $([\bar a_n],\prec_{\omega'})$, namely: $X\cap B_i$ consists of
the first $x_i$ elements of $B_i$, for each $i$. Under this
correspondence,~\refeq{boxTP} is equivalent to~\refeq{P3}. Although the shape
of the zonogon $Z_{n,a}$ looks somewhat different compared with $Z_{\omega'}$
(since the generating vectors $\xi_\bullet$ for different elements in a block
are non-colinear), it is easy to see that the wirings for the former and the
latter are, in fact, the same. (This implies an equivalence of the g-tilings
for these two cases, which is not seen immediately.) So the integer box case is
reduced, in all aspects we deal with, to the permutation one.


\section*{Appendix: TP-bases and weakly separated set-systems \\
\hphantom{Appendix:\;} on a hyper-simplex}

\refstepcounter{section}

Our results on TP-bases, generalized tilings and weakly separated set-systems
and techniques elaborated in previous sections enable us to obtain an analog of
the equivalence (i)$\Longleftrightarrow$(iv) in Theorem~A to hyper-simplexes.
Let us start with basic definitions and backgrounds.

When dealing with a hyper-simplex $\Delta_n^m=\{S\subseteq [n]\colon |S|=m\}$
rather than the Boolean cube $2^{[n]}$, the notion of TP-functions and TP-bases
are modified as follows. Let $f:\Delta_n^m\to\Rset$. Instead of
relation~\refeq{P3} involving triples $i<j<k$, one considers relation
  \begin{equation}  \label{eq:P4}
 f(Xik)+f(Xj\ell)=\max\{f(Xij)+f(Xk\ell),\; f(Xi\ell)+f(Xjk)\}
  \end{equation}
for a quadruple $i<j<k<\ell$ in $[n]$ and a subset $X\subseteq
[n]-\{i,j,k,\ell\}$ of size $m-2$. When this holds for all such $X,i,j,k,\ell$,
~$f$ is said to be a {\em TP-function} on $\Delta_n^m$. Let $\Tscr\Pscr_n^m$
denote the set of such functions $f$. By an analogy with the Boolean cube case,
a subset $B\subseteq\Delta_n^m$ is called a {\em TP-basis} if the restriction
map $\Tscr\Pscr_n^m\to\Rset^B$ is bijective.

An important instance of TP-bases for $\Delta_n^m$ is the collection
$\Iscr\Sscr_n^m=\Iscr^m\cup\Sscr^m$, where $\Iscr^m=\Iscr_n^m$ consists of the
intervals of size $m$ and $\Sscr^m=\Sscr_n^m$ consists of the sets of size $m$
representable as the union of two nonempty intervals $[1..p]$ and $[q..r]$ with
$q>p+1$ (see~\cite{DKK-08}, where the elements of $\Sscr^m$ are called {\em
sesquialteral intervals}).

When a TP-basis contains four sets $Xij,Xk\ell,Xi\ell,Xjk$ as above and one set
$Y\in\{Xik,Xj\ell\}$, the replacement of $Y$ by the other set $Y'$ in
$\{Xik,Xj\ell\}$ gives another TP-basis $B'$. We call such a transformation
$Y\rightsquigarrow Y'$ (or $B\mapsto B'$) a {\em raising (lowering) 4-flip} if
$Y=Xik$ (resp. $Y=Xj\ell$). One can see that $\Iscr\Sscr_n^m$ does not admit
lowering flips and we call this TP-basis {\em standard for} $\Delta_n^m$
(analogously to the basis $\Iscr_n$ for $2^{[n]}$ where weak lowering flips are
absent as well).

The object of our interest is the class $\Bscr_n^m$ of TP-bases that can be
obtained by making a series of 4-flips starting from $\Iscr\Sscr_n^m$. It is
analogous to the class $\Bscr_m$ of semi-normal bases for the Boolean cube case
(and $\Bscr_n^m$ along with the 4-flips on its members represents another
interesting sample of Pl\"ucker environments).

A direct calculation shows that $|\Iscr\Sscr_n^m|=m(n-m)+1$; so all TP-bases
for $\Delta_n^m$ have this cardinality. Besides, one can associate to
$B\in\Bscr_n^m$ the number $\eta(B):=\sum_{X\in B}\sum_{i\in X}i$. Clearly any
lowering 4-flip decreases $\eta$; we shall see later that any $B\in\Bscr_n^m$
is reachable from $\Iscr\Sscr_n^m$ by a series of merely raising 4-flips, and
therefore, $\eta(B)>\eta(\Iscr\Sscr_n^m)$ unless $B=\Iscr\Sscr_n^m$.

Our goal is to show that $\Bscr_n^m$ coincides with the set $\Cscr_n^m$ of {\em
largest} weakly separated collections $C\subseteq\Delta_n^m$, i.e., having
maximum possible cardinality $|C|$. We rely on two known facts. \smallskip

First, Scott~\cite{Sc-1} showed that if a ws-collection $C\subseteq \Delta_n^m$
contains four sets $Xij,Xk\ell,Xi\ell,Xjk$ (with $X,i,j,k,\ell$ as above), then
each of $Xik,Xj\ell$ is weakly separated from any member of $C$. (Note that
$Xik,Xj\ell$ are not weakly separated from each other.) \smallskip

This implies that any TP-basis in $\Bscr_n^m$ is weakly separated, since the
standard basis is such (which is easy to check; cf.~\cite{LZ}). \smallskip

Second, a simple, but important, fact noticed in~\cite{LZ} is that: for $0\le
m'\le m\le n$ and a ws-collection $C\subseteq 2^{[n]}$ whose members have size
at least $m'$ and at most $m$, if we add to $C$ all intervals of size $>m$ and
all co-intervals of size $<m'$, then the resulting collection is again weakly
separated. Let us call the latter collection the {\em straight extension} of
$C$ and denote it by $C^\ast$. Note that the numbers of added intervals and
co-intervals are $\binom{n-m+1}{2}$ and $\binom{m'+1}{2}$, respectively. When
$m'=m$, the fact that the maximum cardinality of a ws-collection in $2^{[n]}$
amounts to $\binom{n+1}{2}+1$ and the identity $\binom{n+1}{2}=\binom{n-m+1}{2}
+\binom{m+1}{2}+m(n-m)+1$ imply $|C|=|C^\ast|-\binom{n-m+1}{2}- \binom{m+1}{2}
\le m(n-m)+1=|\Iscr\Sscr_n^m|$.

Thus, $\Iscr\Sscr_n^m$ is a largest ws-collection in $\Delta_n^m$, implying
that all members of $\Bscr_n^m$ are such, i.e., $\Bscr_n^m\subseteq\Cscr_n^m$.
We show the following analog of a result from Section~\ref{sec:wsf} to
hyper-simplexes.
  \begin{theorem}  \label{tm:hyper}
Let $C\in\Cscr_n^m$ and $C\ne \Iscr\Sscr_n^m$. Then $C$ admits a lowering
4-flip (defined in the same way as for TP-bases). In particular, all members of
$\Cscr_n^m$ belong to one and the same orbit w.r.t. 4-flips.
  \end{theorem}

In light of the above discussion, this gives the desired result:
  \begin{corollary} \label{cor:Bm=Cm}
$\Bscr_n^m=\Cscr_n^m$.
  \end{corollary}

  \noindent{\bf Proof of Theorem~\ref{tm:hyper}}
~Note that $C$ contains all intervals and co-intervals of size $m$ (since they
are weakly separated from any member of $C^\ast$).

We use induction on $n$, assuming w.l.o.g that $0<m<n$. The ws-collection
$C^\ast$ is largest in $2^{[n]}$ and its projection $C'$ (defined as in
Section~\ref{sec:wsf}) is a largest ws-collection in $2^{[n-1]}$. Let $T$ be a
g-tiling on $Z_{n-1}$ whose spectrum $B_T$ is $C'$. Clearly for $h>m$ (resp.
$h<m-1$), the set $C'_h:=\{X\in C'\colon |X|=h\}$ consists of all intervals
(resp. all co-intervals) of size $h$ in $[n-1]$.

As to level $m$, all sets $X\in C'_m$ are exactly the members of $C$ not
containing the element $n$ (since there exists only one set $Y\in C^\ast$ with
$|Y|>m$, namely, the interval $[n-m..n]$, whose projection occurs in $C'_m$,
but the latter is the interval $[n-m..n-1]$, which belongs to $C\cap C'_m$).
And in level $m-1$, all members of $C'_{m-1}$ are exactly the projections of
those members of $C$ that contain $n$ (since there exists only one set $Y\in
C^\ast$ with $|Y|<m$, namely, the ``co-interval'' $[1..m-1]$, that occurs in
$C'_{m-1}$, but it is the projection of the co-interval $[1..m-1]\cup\{n\}$,
which belongs to $C$).

Now consider two possible cases. \smallskip

{\em Case 1}. ~Let $T$ have no feasible M-configuration of height $m-1$
(defined as in Subsection~\ref{ssec:part1}). Then, by
Proposition~\ref{pr:int-lflip}, $C'_{m-1}$ contains only co-intervals. It is
easy to see that, under the above projection map, the preimages in $C$ of these
co-intervals $X$ (which are exactly the members of $C$ containing $n$) are only
intervals and sesquialteral intervals. Also the facts that $C'$ is the straight
extension of $C'_m$ and that $C'$ is a largest ws-collection in $2^{[n-1]}$ (by
Corollary~\ref{cor:C=B}) imply $C'_m\in\Cscr_{n-1}^m$. Since
$C\ne\Iscr\Sscr_n^m$, the subcollection $C'_m$ of $C$ differs from
$\Iscr\Sscr_{n-1}^m$. By induction $C'_m$ admits a lowering 4-flip; this is
just a required 4-flip for $C$. \smallskip

{\em Case 2}. ~Let $T$ have a feasible M-configuration $CM(X;i,j,k)$ of height
$m-1$, i.e., $C'=B_T$ contains sets $Xi,Xj,Xk,Xij,Xjk$ with $|X|=m-2$ and
$i<j<k$. Since the first three sets among them belong to level $m-1$, $C$
contains the sets $Xin,Xjn,Xkn$. These together with the sets $Xij,Xjk$
contained in $C$ (since the latter ones are in level $m$ of $T$) give the
desired configuration (involving the quadruple $i<j<k<n$) for performing a
lowering 4-flip in $C$.

This completes the proof of the theorem.  \qed

  \medskip
  \noindent{\bf Remark 8.}
We can give an alternative method of proving the equality $\Bscr_n^m=\Cscr_n^m$
relying on results of Postnikov~\cite{Po} on alternating strand diagrams, or,
briefly, {\em as-diagrams}. One can outline the idea of this method as follows
(omitting details). Given $C\in\Cscr_n^m$, let $T$ and $W$ be, respectively,
the g-tiling and proper wiring with $B_T=B_W=C^\ast$. Take the maximum zigzag
paths $P,P'$ in the principal components $H_{m-1},H_m$ of the graph $G_T$,
respectively (see Section~\ref{sec:contr-exp}), both going from the vertex
$[m]$ to the vertex $[n-m+1..n]$. Then $P'$ passes (as vertices) all interval
of sizes $m$ and $m+1$, and one can see that the sequence of colors of its
edges is $m+1,1,m+2,2, \ldots, n,n-m$ (in this order on $P'$). In its turn, the
reverse path $P^{-1}$ of $P$ passes all co-intervals of sizes $m-1$ and $m$,
and the sequence of colors of its edges is $n-m+1,1,n-m+2,2, \ldots,n,m$. When
considering a planar layout of $G_T$ on the disc $D_T$, the concatenation
(circuit) $Q$ of $P'$ and $P^{-1}$ cuts out a smaller disc $\tilde D$ in $D_T$;
it is the union of (the squares representing) the tiles of height $m$ in $T$.

Take the parts $w'_i$ of wires $w_i\in W$ going across $\tilde D$, and denote
the beginning and end points of $w'_i$ by $u_i$ and $v_i$, respectively (both
$u_i,v_i$ are interior points of edges of color $i$ in $Q$, by the construction
in Section~\ref{sec:til-to-wir}). Let $S$ be the sequence of $2m$ points
$u_i,v_j$ along $P^{-1}$, and $S'$ the sequence of $2(n-m)$ points $u_i,v_j$
along $P'$. Partition $S$ (resp. $S'$) into $m$ (resp. $n-m$) consecutive
pairs; then each pair contains one ``source'' $u_i$ and one ``sink'' $v_j$ of
the wiring $W'=(w'_1,\ldots,w'_n)$. Now extend each wire in $W'$ within the
boundary $Q$ of $\tilde D$ so that, for the pairs $\{u_i,v_j\}$ as above, the
beginning of $w'_i$ coincide with the end of $w'_j$. One shows (a key) that the
resulting wiring forms an as-diagram $\Dscr$ of~\cite{Po}. Moreover, the
sequences of edge colors in $P^{-1},P'$ indicated above provide that the
corresponding permutation on $[n]$ associated to the set of directed chords
$(u_i,v_i)$ (when the above $n$ pairs are numbered clockwise) is the Grassmann
permutation $\omega$ for $m,n$, namely, $\omega(i)=m+i$ for $i=1,\ldots,n-m$,
and $\omega(i)=i-n+m$ otherwise. Finally, using properties of non-terminal
vertices of $T$ exhibited in~\refeq{mixedvert}, one shows that Postnikov's
operations (M1) (``square moves'') on $\Dscr$ correspond to 4-flips on $C$.
Then Theorem~\ref{tm:hyper} can be derived from a result (Theorem~13.4)
in~\cite{Po}. \bigskip

Next, the fact that the poset $(\Bscr_n^m,\prec)$ (where $B\prec B'$ if $B$ is
obtained from $B'$ by lowering 4-flips) has a unique minimal element (namely,
$\Iscr\Sscr_n^m$) easily implies that this poset has a unique maximal element.
This is $\Iscr_n^m\cup\>$co-$\Sscr_n^m$, where co-$\Sscr_n^m$ consists of the
$m$-sized sets representable as the union of two nonempty intervals $[p..q]$
and $[r..n]$ with $r>q+1$. Let us call this basis {\em co-standard} for
$\Delta_n^m$.

One more useful construction concerns an embedding of the set $\Bscr_n$ of
semi-normal TP-bases for the Boolean cube $2^{[n]}$ into TP-bases for some
hyper-simplex. More precisely, consider a hyper-simplex $\Delta_{n+n'}^n$ with
$n'\ge n$. For $X\subseteq[n]$, define $X^\Delta$ to be the union of $X$ and
the interval $[q..n+n']$ of size $n-|X|$ (which is empty when $X=[n]$). In
addition, let $\Ascr$ be the collection of all $n$-sized intervals $[p..q]$
with $n<q<n+n'$ and all $n$-sized sets $[p..q]\cup[r..n+n']$ with $n+1<q+1<r\le
n+n'$. Then the union of $\Ascr$ and the collection $\{I^\Delta\colon
I\in\Iscr_n\}$ forms the co-standard basis for $\Delta_{n+n'}^n$. In a similar
fashion, we associate to any $\Xscr\subseteq 2^{[n]}$ the collection $\Ascr\cup
\{X^\Delta\colon X\in \Xscr\}$, denoted as $\Xscr^\Delta$.
  \begin{prop} \label{pr:embed_in_Delta}
For any semi-standard TP-basis $B\in\Bscr_n$, one holds
$B^\Delta\in\Bscr_{n+n'}^n$.
  \end{prop}
   \begin{proof}
Let $B\ne \Iscr_n$. Then $B$ contains sets $Xi,Xk,Xij,Xik,Xjk$ for some $i<j<k$
and $X\subseteq[n]-\{i,j,k\}$. Under the transformation $Y\mapsto Y^\Delta$,
these sets turn into the sets (respectively) $Xi\cup[\ell..n+n'],
Xk\cup[\ell..n+n'], Xij\cup[\ell+1..n+n'],Xik\cup[\ell+1..n+n'],
Xjk\cup[\ell+1..n+n']$ in $\Delta_{n+n'}^m$, where $\ell:=n'+|X|+2$. Then the
lowering flip $Xik\rightsquigarrow Xj$ in $B$ corresponds to the raising flip
$X'ik\rightsquigarrow X'j\ell$ in $B^\Delta$, where $X'=:X\cup[\ell+1..n+n']$,
and the result follows.
   \end{proof}

The embedding $B\mapsto B^\Delta$ of the TP-bases for $2^{[n]}$ into TP-bases
for $\Delta_{n+n'}^n$ can be regarded as a certain dual analog of the straight
extension map $B'\mapsto (B')^\ast$ of the TP-bases for $\Delta_n^m$ to
TP-bases for $2^{[n]}$. \medskip

We conclude this section with a generalization to a truncated Boolean cube
$\Delta_n^{m',m}:= \{S\subseteq[n]\colon m'\le |S|\le m\}$, where $0\le m'\le
m\le n$. In this case, one considers the class $\Tscr\Pscr_n^{m',m}$ of
functions $f:\Delta_n^{m',m}\to\Rset$ obeying relations~\refeq{P3}
and~\refeq{P4} for all corresponding corteges where the six sets occurring as
arguments belong to $\Delta_n^{m',m}$. (In fact, it suffices to require
that~\refeq{P3} be imposed everywhere but that~\refeq{P4} be explicitly imposed
only for the corteges related to the lowest level, i.e., when $|X|=m'-2$.
Then~\refeq{P4} for the other corteges will follow; see~\cite{DKK-08}.)

The class $\Tscr\Pscr_n^{m',m}$ has as a basis the set $\Iscr\Sscr_n^{m',m}:=
\Sscr_n^{m'}\cup\Iscr_n^{m'}\cup\Iscr_n^{m'+1}\cup\ldots
\cup\Iscr_n^m$~\cite{DKK-08}, called the standard basis for this case. We
define $\Bscr_n^{m',m}$ to be the set of bases obtained from the standard one
by a series of weak 3-flips (i.e., related to~\refeq{P3}) and 4-flips. (From
the theorem below it follows that the latter ones are important only in level
$m'$.) On the other hand, we can consider the set $\Cscr_n^{m,m'}$ of largest
ws-collections in $\Delta_n^{m',m}$ (whose cardinality is $\binom{n+1}{2}
-\binom{n-m+1}{2}-\binom{m'+1}{2}=|\Iscr\Sscr_n^{m',m}|$). By explanations
above, $\Bscr_n^{m,m'}\subseteq\Cscr_n^{m,m'}$.
   \begin{theorem}  \label{tm:truncube}
Let $C\in\Cscr_n^{m,m'}$. Then $\Iscr\Sscr_n^{m',m}$ can be obtained from $C$
by a series of lowering weak 3-flips followed by a series of lowering 4-flips
in level $m'$. Therefore, $\Bscr_n^{m,m'}=\Cscr_n^{m,m'}$.
   \end{theorem}
   \begin{proof}
~If $C$ admits a lowering 3-flip, then performing such a flip produces a
collection in $\Cscr_n^{m,m'}$ with a smaller total size of its members. If
such a flip is impossible, then all sets $X\in C$ with $|X|>m'$ are intervals,
by Proposition~\ref{pr:int-lflip}. Then $\{X\in C\colon |X|=m'\}$ is a largest
ws-collection for the hyper-simplex $\Delta_n^{m'}$, and the result follows
from Theorem~\ref{tm:hyper}.
  \end{proof}



\begin{thebibliography}{99}
 %
\bibitem{BFZ} A.~Berenstein, S.~Fomin, and  A.~Zelevinsky,
Parametrizations of canonical bases and totally positive matrices, {\sl Adv.
Math.} {\bf 122} (1996), 49–-149.
 %
\bibitem{DKK-08} V.~Danilov, A.~Karzanov and G.~Koshevoy, On bases of
tropical Pl\"ucker functions, {\sl ArXiv}:0712.3996[math.CO], 2007.
 %
\bibitem{HS} A.~Henriques and D.E.~Speyer, The multidimensional
cube recurrence, {\sl ArXiv}:0708.2478v1[math.CO], 2007.
 %
\bibitem{K1} J.~Kamnitzer, The crystal structure on the set of
Mirkovic-Vilonen polytopes, {\sl Adv.~Math.} {\bf 215} (1) (2007), 66--93.
 %
\bibitem{KS} R.~Kenyon and J.-M.~Schlenker, Rhombic embeddings of planar graphs
with faces of degree 4, {\sl ArXiv}:math-ph/0305057, 2003.

\bibitem{L} A.~Lauve, Quasideterminants and q-commuting minors,
{\sl ArXiv}:0602.448 [math.QA], 2006.
 %
\bibitem{LZ} B.~Leclerc and A.~Zelevinsky: Quasicommuting families of
quantum Pl\"ucker coordinates, {\em Amer. Math. Soc. Trans., Ser.~2} ~{\bf 181}
(1998), 85--108.
 %
\bibitem{Po} A.~Postnikov, Total positivity, Grassmannians, and networks,
{\sl ArXiv}:math.CO/0609764, 2006.
 %
\bibitem{Sc-1} J.~Scott, Quasi-commuting families of quantum minors,
{\sl J.~Algebra} {\bf 290} (1) (2005) 204--220.
 %
\bibitem{Sc} J.~Scott, Grassmannians and cluster algebras,
{\sl Proc. London Math. Soc. (3)} ~{\bf 92} (2) (2006), 345--380.
 %
\bibitem{Sp} D.~Speyer, Perfect matchings and the octahedron recurrence,
{\sl J.~Algebraic Combin.} {\bf 25} (2007), 309-348.
 %
\bibitem{SW} D.~Speyer, L.~Williams, The tropical totally positive
Grassmanian, {\sl J.~Algebraic Combin.} {\bf 22} (2) (2005), 189--210.
 %
\end{thebibliography}
\end{document}